\documentclass[10pt]{amsart}
\usepackage{amsthm,amsmath,amssymb,amscd,graphics,enumerate, stmaryrd,xspace,verbatim, epic, eepic,url,epigraph}

\usepackage[usenames,dvipsnames]{color}
\usepackage[colorlinks=true]{hyperref}
\makeindex

\usepackage[active]{srcltx}
\usepackage[all]{xypic}
\SelectTips{cm}{}

{
   \newtheorem{theorem}[subsubsection]{Theorem}
      \newtheorem*{theorem*}{Theorem}
   \newtheorem{proposition}[subsubsection]{Proposition}
   
   \newtheorem{lemma}[subsubsection]{Lemma}

   \newtheorem{corollary}[subsubsection]{Corollary}
   
   \newtheorem*{conjecture*}{Conjecture}
   
}
{\theoremstyle{definition}
          \newtheorem*{exercise*}{Exercise}
   
   \newtheorem{example}[subsubsection]{Example}
   \newtheorem*{example*}{Example}

   \newtheorem*{definition*}{Definition}
   \newtheorem{rem}[subsubsection]{Remark}
   \newtheorem{remark}[subsubsection]{Remark}

}
\newcommand{\RR}{{\mathbb{R}}}

\newcommand{\FF}{{\mathbb{F}}}

\newcommand{\QQ}{{\mathbb{Q}}}
\newcommand{\NN}{{\mathbb{N}}}

\newcommand{\ZZ}{{\mathbb{Z}}}

\newcommand{\GG}{{\mathbb{G}}}
\newcommand{\LL}{{\mathbb{L}}}

\newcommand{\fp}{{\boldsymbol{\mathfrak p}}}
\newcommand{\fq}{{\boldsymbol{\mathfrak q}}}

\newcommand{\fX}{{\mathfrak{X}}}
\newcommand{\fY}{{\mathfrak{Y}}}

\newcommand{\tL}{{\tilde{L}}}
\newcommand{\tI}{{\tilde{I}}}

\newcommand{\tJ}{{\tilde{J}}}

\newcommand{\tn}{{\tilde{n}}}
\newcommand{\tilh}{{\tilde{h}}}

\newcommand{\cI}{{\mathcal I}}
\newcommand{\cJ}{{\mathcal J}}

\newcommand{\cM}{{\mathcal M}}

\newcommand{\cO}{{\mathcal O}}

\newcommand{\cX}{{\mathcal X}}

\newcommand{\ut}{{\underline{t}}}

\newcommand{\oY}{{\overline{Y}}}
\newcommand{\oA}{{\overline{A}}}

\newcommand{\oM}{{\overline{M}}}
\newcommand{\oN}{{\overline{N}}}
\newcommand{\oI}{{\overline{I}}}

\newcommand{\om}{{\overline{m}}}
\newcommand{\ot}{{\overline{t}}}

\def\<{\langle}
\def\>{\rangle}

\newcommand{\Id}{\operatorname{Id}}
\newcommand{\Spec}{\operatorname{Spec}}

\newcommand{\Proj}{\operatorname{Proj}}

\newcommand{\Ker}{{\operatorname{Ker}}}

\newcommand{\codim}{\operatorname{codim}}

\newcommand{\bfA}{{\mathbf A}}

\newcommand{\ocI}{\overline{{\mathcal I}}}

\newcommand{\ocM}{\overline{{\mathcal M}}}

\newcommand{\oP}{{\overline{P}}}

\newcommand{\oD}{{\overline{D}}}
\newcommand{\oE}{{\overline{E}}}

\newcommand{\oV}{{\overline{V}}}

\newcommand{\red}{{\operatorname{red}}}

\def\nor{{\rm nor}}
\def\sat{{\rm sat}}
\def\:{{\colon}}
\def\.{{,\dots,}}

\newcommand{\double}{\genfrac..{0pt}1
{\raise -1pt\hbox{$\scriptstyle\longrightarrow$}}{\raise 3pt\hbox
{$\scriptstyle\longrightarrow$}}}

\renewcommand{\setminus}{\smallsetminus}

\def\sat{{\rm sat}}

\def\et{{\rm \acute et}}

\def\tototi{\mathbin{\mathop{\otimes}\limits^{\raise-1pt\hbox
{$\scriptscriptstyle {\rm L}$}}}}

\def\indlim{\mathop{\vrule width0pt height7pt depth
4pt\smash{\lim\limits_{\raise 1pt\hbox to 14.5pt
{\rightarrowfill}}}}}
\def\projlim{\mathop{\vrule width0pt height7pt depth
4pt\smash{\lim\limits_{\raise 1pt\hbox to 14.5pt
{\leftarrowfill}}}}}

\newcommand\displaceamount{3pt}

\newcommand{\doubledown}{\ar@<\displaceamount>[d]\ar@<-\displaceamount>[d]}

\newcommand{\doubleup}{\ar@<\displaceamount>[u]\ar@<-\displaceamount>[u]}

\newcommand{\doubleright}{\ar@<\displaceamount>[r]\ar@<-\displaceamount>[r]}

\newcommand{\tor}{{\operatorname{tor}}}

\newcommand{\rk}{\operatorname{rk}}
\newcommand{\height}{\operatorname{ht}}

\def\into{\hookrightarrow}
\def\onto{\twoheadrightarrow}

\def\gp{\text{gp}}

\def\GGm{{\mathbb G}_m}

\def\tilx{{\widetilde x}}
\def\tilf{{\widetilde f}}
\def\tilX{{\widetilde X}}

\def\tilh{{\widetilde h}}
\def\tilY{{\widetilde Y}}

\def\toisom{\xrightarrow{{_\sim}}}
\def\hatX{{\widehat X}}
\def\hatA{{\widehat A}}

\def\hatR{{\widehat R}}
\def\hatZ{{\widehat Z}}

\def\bfD{{\mathbf{D}}}
\def\wtimes{{\widehat\otimes}}

\def\hatcO{{\widehat\cO}}

\def\Zar{{\rm Zar}}

\begin{document}
\title[Torification of diagonalizable group actions]{Torification of diagonalizable group actions on toroidal schemes}

\author[D. Abramovich]{Dan Abramovich}
\address{Department of Mathematics, Box 1917, Brown University,
Providence, RI, 02912, U.S.A}
\email{abrmovic@math.brown.edu}
\author[M. Temkin]{Michael Temkin}
\address{Einstein Institute of Mathematics\\
               The Hebrew University of Jerusalem\\
                Giv'at Ram, Jerusalem, 91904, Israel}
\email{temkin@math.huji.ac.il}
\thanks{This research is supported by  BSF grant 2010255}

\date{\today}

\begin{abstract}
We study actions of diagonalizable groups on toroidal schemes (i.e. logarithmically regular logarithmic schemes). In particular, we show that for so-called toroidal actions the quotient is again a toroidal scheme. Our main result constructs for an arbitrary action a canonical torification  by an equivariant blowings up. This extends earlier results of Abramovich-de Jong, Abramovich-Karu-Matsuki-W{\l}odarczyk, and Gabber in various aspects.
\end{abstract}

\maketitle
\setcounter{tocdepth}{1}

\tableofcontents

\section{Introduction}
\subsection{Toroidal actions, quotients and torification} Consider a variety $X$ with toroidal structure and an action of a group $G$ on $X$.
{\em Torification} is a blowing-up process $X'\to X$  which guarantees that the quotient map $X' \to X' \sslash G$ is toroidal. It was  introduced in \cite{Abramovich-deJong} when $G$ is finite for the purpose of proving resolution of singularities;   and in \cite{AKMW} when $G=\GG_m$ for proving factorization of birational maps.

In this paper we consider $G$ diagonalizable, and prove a general torification result for arbitrary toroidal schemes (see Section \ref{Sec:toroidal-schemes}), not necessarily over a field:

\begin{theorem}[See Theorem \ref{maintheorem2}]\label{Th:main-with-barycentric}
Assume that a diagonalizable group $G$ acts in a relatively affine manner on a toroidal scheme $(X,D)$. Then there is a $G$-equivariant modification $F_{(X,D)}\:X'\to X$, such that, denoting by  $D'$ be the union of the preimage of $D$ and the exceptional divisor of $F_{(X,D)}$, we have

(i) The pair $(X',D')$ is toroidal and the natural $G$-action on $(X',D')$ is toroidal.

(ii) The morphism $F_{(X,D)}$ is functorial with respect to surjective strongly equivariant strict morphisms $h\:(Y,E)\to(X,D)$ of toroidal schemes in the sense that $F_{(Y,E)}$ is the base change of $F_{(X,D)}$.
\end{theorem}

We refer to \cite[Section 5.3.1]{ATLuna} for the notions of relatively affine actions and strongly equivariant morphisms, and to Section \ref{Sec:toroidal-morphisms} for the notion of morphisms between toroidal schemes. A slightly more precise statement of Theorem \ref{maintheorem2} in terms of blowing up is provided in Theorem \ref{Th:general-normal}.

Theorem \ref{Th:main-with-barycentric} builds on a more detailed Theorem  \ref{torificationth} which assumes the action to be $G$-simple, see   Section  \ref{Sec:G-simple}. We further optimize that result as follows:

\begin{theorem}[See Theorem \ref{Maintheorem}]
Assume that a toroidal scheme $(X,D)$ is provided with a relatively affine,  {\em $G$-simple} action of a diagonalizable group $G=\bfD_L$. Assume $X$ contains a strongly equivariant dense open set $U$ on which $G$ acts freely. There exist ideal sheaves $I_X$ and $\tI_X$ on $X$ and $\tilX=X\sslash G$ with resulting blowings up $f_{(X,D)}\:X'\to X$ and $\tilf_{(X,D)}\:\tilX'\to\tilX$, such that $I_X$ is locally generated by $G$-invariants, and, denoting by $D'$ the union of the preimage of $D$ and the exceptional divisor of $f_{(X,D)}$, we have

(i) The pair $(X',D')$ is toroidal and the natural $G$-action on $(X',D')$ is toroidal.

(ii) The morphism of quotients  $f_{(X,D)} \sslash G: X'\sslash G\to\tilX$ is  $\tilf_{(X,D)}$.

(iii) The blowings up $f_{(X,D)}$ and $\tilf_{(X,D)}$ are functorial with respect to surjective strongly equivariant strict morphisms $h\:(Y,E)\to(X,D)$ of toroidal schemes: denoting  $\tilh=h\sslash G$, we have $h^*(I_X)=I_Y$, $f^\tor_{(Y,E)}=f^\tor_{(X,D)}\times_XY$, $\tilh^*(\tI_X)=\tI_Y$, and $\tilf_{(Y,E)}=\tilf_{(X,D)}\times_\tilX\tilY$.

(iv) If $V\subseteq X$ is a strongly equivariant open subset such that the action on $(V,D|_V)$ is toroidal then $I_X$ restricts to the unit ideal on $V$ and $V\times_XX'=V$.
\end{theorem}


Part (i) of Theorem  \ref{Maintheorem} says that $f_{(X,D)}$ {\em torifies} the action of $G$. It implies that the quotient $ (X',D') \sslash G$ above is toroidal:

\begin{theorem}[See Theorem \ref{toroidalth}]
Assume that a toroidal scheme $(X,U)$ is provided with a relatively affine toroidal action of a diagonalizable group $G=\bfD_L$. Then,

(i) The toroidal quotient $(X,U)\sslash G$ exists.

(ii) The morphism $(X,U)\to(X,U)\sslash G$ is toroidal whenever the torsion degree of $L$ is invertible on $X$.

(iii) If $h\:(Z,W)\to(X,U)$ is a strongly equivariant strict morphism then the quotient $h\sslash G\:(Z,W)\sslash G\to(X,U)\sslash G$ is a strict morphism of toroidal schemes.
\end{theorem}

\begin{rem}
In case of finite groups, an alternative torification method was suggested by Gabber, see \cite[Theorem~VIII.1.1]{Illusie-Temkin}. It is based on resolving quotient singularities and is quite different from the blowing up method we use here. We note that we use a single blowing up with a non-reduced center rather than  search for a sequence of blowings up with regular centers.
See \cite{Bergh, Buonerba} for  different angles on this problem.
\end{rem}

\begin{rem}
If the torsion degree of $L$ vanishes in $\cO_X$, the quotient morphism $(X,U)\to(X,U)\sslash G$ is inseparable hence not a toroidal morphism. It is, however, modeled on the quotient of a toric variety by a subgroup-scheme of the torus. It might be of interest to study such quotients.
\end{rem}

\subsection{Motivation: factorization of birational maps}
This paper leads to our forthcoming work  \cite{AT2} on weak factorization of birational maps, generalizing the main theorems of \cite{AKMW} and \cite{W-Factor} to the appropriate generality of qe schemes  (wherever suitably strong resolution of singularities applies), and further proving factorization results in other geometric categories of interest.


While \cite{AT2} only requires actions of $\GG_m$, we find it both convenient and fruitful to work with an arbitrary diagonalizable group.

In addition to the main results, the paper \cite{AT2} requires Lemma \ref{Lem:Gm-action} on $\GG_m$-actions and Lemma \ref{torificblowlem} on compatibility of $S$-torific  ideals.

\subsection{Background}
This article relies on a number of results from \cite{ATLuna}, in particular Luna's Fundamental Lemma for diagonalizable groups.
Apart from key results    of \cite{ATLuna}, we use the setup and terminology of that paper. In particular we use the (classical) notion of diagonalizable groups  \cite[VIII.1.1]{SGA3-2}, \cite[3.2.1]{ATLuna}, $L$-local and $L$-complete rings \cite[4.4.4, 4.5.5]{ATLuna}, strongly equivariant and inert morphisms \cite[5.3.1, 5.5.3]{ATLuna}.\index{diagonalizable group (\cite[VIII.1.1]{SGA3-2})}\index{strongly equivariant morphism (\cite[5.3.1, 5.5.3]{ATLuna})}

\subsection{Some tools}
One of the main tools of our method is a formal-local description of toroidal schemes with a toroidal action proved in Theorem~\ref{flcilem}, which extends Kato's formal-local description of toroidal schemes to the $G$-equivariant case. We deduce in Corollary~\ref{gencor} that the property of an action being toroidal is an open condition. Note that already for toroidal schemes stability under generizations is a rather non-trivial property, see \cite[Proposition~7.1]{Kato-toric} and a completed argument in \cite[Theorem~9.5.47]{Gabber-Ramero}. We suggest a short new proof of this result of Kato in Theorem~\ref{localth}.

Another tool of our method is to enlarge or decrease the toroidal divisor of a toroidal scheme. In Theorem~\ref{decreaseth} and Propositions~\ref{decreaseprop} and \ref{pretoroidalprop} we study when the pair and the action remain or become toroidal under these operations.

\subsection{Further directions}

It would be interesting, perhaps in future work,  to further extend the torification procedure to other relatively affine tame actions, i.e. actions having  linearly reductive (or even reductive) stabilizers. Moreover, one may hope to extend this to tame groupoids and their quotients, that one may call tame stacks (if the stabilizers are of dimension zero then those are the tame stacks of \cite{AOV1}).

\section{Toroidal schemes}\label{Section:toroidalschemes}
Toroidal schemes generalize the classical toroidal embeddings of varieties. Although one can describe them by formal charts, it is more convenient to use an equivalent approach, where toroidal schemes are logarithmically  regular logarithmic schemes. For simplicity, we will consider Zariski logarithmic schemes and toroidal schemes without self intersections. We remark, however, that analogously to \cite[Exp. VI, \S1]{Illusie-Temkin} almost everything can be done for general logarithmic schemes at the cost of replacing points and localizations with geometric points and strict henselizations. We do not pursue this further here.

\subsection{Monoids}

\subsubsection{Conventions}
We will use the following notation for commutative monoids: we denote by $\oM=M/M^\times$  the sharpening of $M$, by $M^\gp$  the Grothendieck group of $M$, and by $M^+=M\setminus M^\times$ the maximal ideal of $M$. The {\em rank} of a monoid $M$ is $\rk(M)=\dim_\QQ(M^\gp\otimes_\ZZ\QQ)$.\index{rank!of a monoid}

\subsubsection{Toric monoids}\label{toricmonoidsec}
A {\em toric monoid}\index{monoid!toric} is an fs (i.e. fine and saturated) monoid $M$ without torsion.

\begin{remark}
Usually, one also requires toric monoids to be sharp, namely $M^\times = \{0\}$,\index{monoid!sharp} but we prefer to modify the terminology in this paper.
\end{remark}

\subsubsection{Prime ideals}
Any subset $S\subseteq M$ generates an ideal $(S)=\cup_{f\in S}(f+M)$. An ideal $(f)=f+M$ with $f\in M$ is called {\em principal}. \index{ideal (in monoid)} Prime ideals and their heights are  naturally defined, analogously to the case of rings, see \cite[\S5]{Kato-toric}. In toric monoids, prime ideals are of the form $\fp=M\setminus F$ where $F$ is a face, so $\height(\fp) = \rk(M)-\rk(F)$. In particular, $\height(\fp)=1$ if and only if $F$ is a facet.

\subsubsection{Inner elements}\label{innerelementsec}
An element $v\in M$ will be called {\em inner} if it lies in the interior of the polyhedral cone $M_\RR:= M\cdot \RR_{\geq0} \subset M^\gp\otimes_\ZZ\RR$. In toric monoids the following conditions are easily seen to be equivalent: (a) $v$ is inner, (b) $v$ is not contained in any facet of $M$, (c) $v$ lies in all nonempty prime ideals of $M$.\index{inner element of a monoid}

\subsubsection{Divisorial prime ideals}
Prime ideals of height one are analogous to divisorial ideals. If $\fp$ is prime of height 1 then $F=M\setminus\fp$ is a facet and the image of $M$ in $M^\gp/F^\gp$ is isomorphic to $\NN$. So, to any $f\in M$ we can associate a number $\nu_\fp(f)\in\NN$, which is an analogue of the order of $f$ with respect to $\fp$.

\subsubsection{Splitting faces and facets off}
We say that a face $N\subseteq M$ {\em splits off} if there exists a face $K$ such that $M=N\oplus K$. If $M$ is sharp and toric then $K$ is the face spanned by all edges not contained in $N$. In addition, the ideal $(N^+)=N^++M$ generated by $N^+$ coincides with $M\setminus K$, hence $(N^+)$ is prime and
$$\height(N^+)=\rk(M)-\rk(K)=\rk(N).$$ Various combinations of these facts provide criteria for splitting off of $N$.

\begin{lemma}\label{splitlem}
Let $M$ be a sharp toric monoid with a face $N$. Then the following conditions are equivalent:
\begin{itemize}
\item[(a)]  $N$ splits off,

\item[(b)] $(N^+)$ is prime and $\height((N^+))\ge\rk(N)$,

\item[(c)]  $(N^+)$ is prime and the face $K=M\setminus (N^+)$ satisfies $\rk(K)+\rk(N)\le\rk(M)$,
\end{itemize}
\end{lemma}
\begin{proof}
If $(N^+)$ is prime then $K=M\setminus (N^+)$ is a face and $\height((N^+))=\rk(M)-\rk(K)$. Thus, (b) and (c) are equivalent.

If (a) holds then $M=N\oplus K$, and hence $(N^+)=M\setminus K$ and $\rk(M)=\rk(K)+\rk(N)$. In particular, (a) implies (c).

Conversely, assume (c) is satisfied and set $K=M\setminus (N^+)$. Then any element $m_0\in M^+$ is of the form $m_1+n_1$ with $0\neq n_1\in N+K$ and $m_1\in M$. If $m_1\neq 0$ then $m_1=m_2+n_2$ with $0\neq n_2\in N+K$, etc. We claim that at some stage $m_l=0$ and hence $M=N+K$. Indeed, since $M$ is sharp, for any $m\in M^+$ there exists a number $h(m)$ such that $m$ is the sum of at most $h(m)$ elements of $M^+$. This shows that the homomorphism $N\oplus K\to M$ is surjective, and using the inequality on ranks we obtain that $N^\gp\cap K^\gp=0$ and hence $N\oplus K=M$.
\end{proof}

In the case of a facet, there are two more useful criteria.

\begin{lemma}\label{splitfacet}
Let $M$ be a sharp toric monoid, $\fp\subset M$ a prime ideal of height 1 and $F=M\setminus\fp$ its facet. The following conditions are equivalent:
\begin{itemize}
\item[(a)] $F$ splits off,

\item[(b)] $(F^+)\subset M$ is prime and $\height((F^+))\ge\rk(M)-1$.

\item[(c)] $(F^+)\subset M$ is prime and there is at most one edge $E$ not contained in $F$,

\item[(d)] $(F^+)$ is the union of all prime ideals of height 1 different from $\fp$,

\item[(e)] $\fp$ is principal.

\end{itemize}
In addition, if these conditions are satisfied then $M=F\oplus\NN e$, where $e$ is the generator of $\fp$ and $\NN e$ is the only edge not contained in $F$.
\end{lemma}
\begin{proof}
Conditions (a) and (b) are the same as \ref{splitlem}(a) and  \ref{splitlem}(b), and condition (c) is equivalent to  \ref{splitlem}(c). Indeed, since $F$ is a facet we have $\rk(M)=\rk(F)+1$; writing  $K=M\setminus(F^+)$, the inequality $\rk(K)\leq 1$ holds if and only if $K$ consists of at most one edge $\langle e\rangle$. Thus, conditions (a), (b) and (c) are equivalent by Lemma~\ref{splitlem}.

We claim that (c) and (d) are equivalent. Note that the union of prime ideals is prime. This can be seen by looking at their complements: the intersection of faces is a face.  We therefore have that $(F^+)$ is prime both in (c) and (d). Let $I$ be the union of all prime ideals of height one different from $\fp$. Then $M\setminus I$ is the intersection of all facets different from $F$. A face is the intersection of facets containing it; applying this to edges, if (c) fails and $M\setminus F$ contains two different edges then $M\setminus I=\{0\}$. So, $(F^+)\subsetneq M^+=I$ and (d) fails. If (c) holds then $F$ splits off and a direct check shows that (d) holds.

Assume (e) holds, say $\fp=e+M$. For any element $x\in M$ we have that $x-\nu_\fp(x)e\in F$, hence $M=F+\NN e$ and then necessarily $M=F\oplus\NN e$. In particular, (a) holds. In addition, it is clear that $\NN e$ is the only edge not in $F$, hence (e) implies all the claims in the end of the lemma. Assume (a) holds, say $M=F\oplus K$. Since $F$ is a facet, $K$ is of rank 1, and so $K=\NN e$ for an element $e\in\fp$. In particular, $\fp=(e)$.
\end{proof}

\subsection{Logarithmic schemes}

\subsubsection{Conventiones}
We refer to \cite{Kato-log} for the general definition of logarithmic schemes. A {\em Zariski logarithmic scheme} is defined similarly but using the Zariski topology: $(X,\cM_X,\alpha)$, where $\alpha\:\cM_X\to(\cO_X,\cdot)$ is a logarithmic structure in the Zariski topology. See also \cite{Kato-toric}. Unless said otherwise, all logarithmic schemes are assumed to be fine and Zariski, and we will use the shorter notation $(X,\cM_X)$.\index{logarithmic!scheme}

\begin{remark}
In fact, one can view Zariski logarithmic schemes as usual logarithmic schemes $(X,\cM_X)$ such that $\varepsilon^*\varepsilon_*\cM_X=\cM_X$, where $\varepsilon\:X_\et\to X_\Zar$ is the natural morphism. In this case, $\cM_X$ is determined by its restriction onto the Zariski site.
\end{remark}

\subsubsection{Ranks}
Given a fine logarithmic scheme $(X,\cM_X)$ we will use the notation $r(x)=\rk(\ocM_{X,x})$ for $x\in X$, and $\rk(\ocM_X)=\max_{x\in X}r(x)$.\index{rank!of a logarithmic scheme}

\subsubsection{Charts}
By a {\em monoidal chart} of a fine logarithmic scheme $(X,\cM_X)$ we mean an open subscheme $V\into X$, a fine monoid $M$, and a homomorphism $\phi\:M\to(\cO_X(V),\cdot)$ such that $\cM_X|_V$ is the logarithmic structure associated with the pre logarithmic structure induced by $\phi$. To give this data is the same as to give a strict morphism $f\:(V,\cM_X|_V)\to(\bfA_M,\cM_{\bfA_M})$, where $\bfA_M=\Spec(\ZZ[M])$ and $\cM_{\bfA_M}$ is the logarithmic structure associated with $M$. Accordingly, we will also call $f$ a monoidal chart of $(X,\cM_X)$. Recall that any fine logarithmic scheme is coherent by definition, hence it can be covered by monoidal charts.\index{chart}

\subsubsection{Sharp charts}
We say that a monoidal chart is {\em sharp} or {\em fs}, if the defining monoid $M$ is sharp or fs, respectively.

\begin{lemma}\label{sharpchartlem}
If an fs logarithmic scheme $(X,\cM_X)$ admits an fs monoidal chart $M\to\cO_X(V)$ then it also admits a sharp monoidal chart $\oM\to\cO_X(V)$.
\end{lemma}
\begin{proof}
Any fs monoid (non-canonically) splits as $M=\oM\oplus M^\times$, see \cite[Lemma 3.2.10]{Gabber-Ramero}. The composition $\oM\into M\to\cO_X(V)$ is a sharp chart.
\end{proof}

\subsubsection{The center of a chart}
The {\em center}\index{center!of a chart} of a monoidal chart $\phi\:M\to\cO_X(V)$ is the closed subscheme $C(\phi):= V(\phi(M^+))$ of $V$. If it is non-empty then we say that the chart is {\em central}.\index{chart!central} It follows from the definitions that the following conditions are equivalent: (i) $x$ lies in the center of $\phi$, (ii) $r(x)=\rk(\oM)$, (iii) the induced homomorphism $\phi_x\:M\to\cO_{X,x}$ is local, i.e. satisfies $\phi_x^{-1}(\cO_{X,x}^\times)=M^\times$. In particular, a chart is central if and only if $\rk(\oM)=\rk(\ocM_X)$.

\begin{remark}
The definition of the center is of global nature and does not make sense for sheaves because $\rk(\ocM_x)$ can jump. For example, the ideal $\alpha(\cM_X^+)\cO_X$ is not coherent already for the affine plane with its toric logarithmic structure.
\end{remark}

\begin{lemma}\label{centerchartlemma}
Assume that a fine logarithmic scheme $(X,\cM_X)$ admits a global central monoidal chart $\phi\:M\to\Gamma(\cO_X)$. Then the center $C(\phi)$ of the chart depends only on $(X,\cM_X)$.
\end{lemma}
\begin{proof}
Set-theoretically, $C(\phi)$ is the set of points where $r(x)$ is maximal. In particular, $|C(\phi)|$ is independent of the chart, and it remains to check that the scheme structure is independent of the chart, which is a local question. For any $x\in C(\phi)$ the homomorphism $\phi_x\:M\to\cO_{X,x}$ is local, hence the homomorphism $\oM\to\ocM_{X,x}$ is an isomorphism, and we obtain that the ideal $\phi_x(M^+)\cO_{X,x}$ coincides with $\alpha_x(\cM_{X,x}^+)\cO_{X,x}$. In particular, the local scheme $$C(\phi)\times_X\Spec(\cO_{X,x})=V(\phi_x(M^+)\cO_{X,x})$$ is independent of the chart, and hence the same is true globally for $C(\phi)$.
\end{proof}

\subsubsection{Logarithmic stratification}
Assume that $(X,\cM_X)$ is a fine logarithmic scheme. For any point $x\in X$ there exists a monoidal chart $\phi\:M\to\cO_X(V)$ such that $x\in V$ and $\oM=\ocM_{X,x}$, in particular, the chart is central. Cover $X$ with central charts $\phi_i\:M_i\to\cO_X(V_i)$ so that any point $x\in X$ lies in the center of some chart, and let $C_i$ denote the center of $\phi_i$ and $r_i=\rk(\oM_i)$. If $r_i=r_j$ then the restrictions of $C_i$ and $C_j$ to $V_i\cap V_j$ coincide by Lemma~\ref{centerchartlemma}. It follows that for any $n\in\NN$, all $C_i$ with $r_i=n$ glue to a locally closed subscheme $X(n)\into X$. Furthermore, set-theoretically $X(n)$ is the set of all points $x\in X$ with $r(x)=n$, hence we obtain a stratification of $X$, that will be called the {\em logarithmic stratification}.\index{logarithmic!stratification}

\begin{remark}
(i) The reduction of the logarithmic stratification was considered in \cite[Exp. VI, \S1.5]{Illusie-Temkin} under the name ``canonical stratification" or ``stratification by rank of $\ocM$".

(ii) $X(0)$ is the triviality locus of $\cM_X$, i.e. the open subscheme on which the logarithmic structure is trivial.

(iii) It follows from the definition that the logarithmic stratification is compatible with strict morphisms: if $(Y,\cM_Y)\to(X,\cM_X)$ is strict then $Y(n)=X(n)\times_XY$ for any $n$.
\end{remark}

\subsubsection{Center of a logarithmic scheme}
Given a fine logarithmic scheme $(X,\cM_X)$ let $X_i$ be the connected components of the logarithmic strata of $X$, and let $C(X,\cM_X)$ be the union of those $X_i$ that are closed in $X$. We call $C(X,\cM_X)$ the {\em center}\index{center!of a logarithmic scheme} of $(X,\cM_X)$. Note that given a  chart $\phi\:M\to\cO_X(V)$, we have $C(\phi)\subset C(X,\cM_X)$ with equality when $X$ is local and $\phi$ is central.

\begin{lemma}\label{centerlemma}
Let $(X,\cM_X)$ be an fs logarithmic scheme. Then the sheaf $\ocM_X=\cM_X/\cO_X^\times$ is locally constant along each logarithmic stratum. In particular, it is locally constant along the center of $(X,\cM_X)$.
\end{lemma}
\begin{proof}
If $x$ is a specialization of $y$ in $X$ then a surjective cospecialization map $h\:\ocM_{X,x}\to\ocM_{X,y}$ arises, see \cite[Lemma~2.12(1)]{Niziol} and its proof. If $x$ and $y$ lie in the same logarithmic stratum then the monoids have the same rank and hence $h$ is an isomorphism.
\end{proof}

We emphasize that since $\ocM_X$ is a Zariski sheaf, it is constant along connected components of logarithmic strata.

\subsection{Toroidal schemes}\label{Sec:toroidal-schemes}
We now focus on toroidal schemes, which form the main case we are interested in.

\subsubsection{Logarithmic regularity}
An fs logarithmic noetherian scheme $(X,\cM_X)$ is called {\em logarithmically regular at}\index{logarithmically regular scheme} a point $x$ if the logarithmic stratum $X(n)$ containing $x$ is regular at $x$ and the equality $\dim(\cO_{X,x})=n+\dim(\cO_{X(n),x})$ holds. If this condition is satisfied at all point of $X$ then we say that $(X,\cM_X)$ is {\em logarithmically regular}.

\begin{remark}
(i) Logarithmic regularity of Zariski logarithmic schemes was introduced by Kato in \cite{Kato-toric}. To the general (\'etale) case it was extended by Nizio{\l} in \cite{Niziol}.

(ii) A logarithmically regular scheme $X$ is Cohen-Macaulay and normal by \cite[Theorem 4.1]{Kato-toric}. Hence $X$ is catenary, and therefore each non-empty stratum $X(n)$ is of pure codimension $n$.
\end{remark}

We refer to \cite[Theorem~11.6]{Kato-toric} for the proof of the following result.

\begin{proposition}
Assume that $(X,\cM_X,\alpha)$ is a logarithmically regular  logarithmic scheme. Let $U=X(0)$ denote the triviality locus and $j\:U\into X$ the open embedding. Then $\alpha\:\cM_X\to\cO_X$ is injective, $D=X\setminus U$ is a divisor, and $\cM_X=j_*\cO^\times_U\cap\cO_X$.
\end{proposition}

\subsubsection{Toroidal schemes}
Given a scheme $X$ with an open subscheme $U$, we say that the pair $(X,U)$ is a {\em toroidal scheme}\index{toroidal!scheme} if the logarithmic structure $$\cM_X:=j_*\cO^\times_U\cap\cO_X\into\cO_X$$ makes $X$ into a logarithmically regular  logarithmic scheme. We will identify $\cM_X$ with a submonoid of $\cO_X$, so $\alpha$ will not be indicated. Sometimes we will use the divisor $D=X\setminus U$ instead of $U$ in the notation of toroidal schemes.

\begin{remark}
(i) The correspondences $(X,U)\mapsto(X,j_*\cO^\times_U\cap\cO_X)$ and $(X,\cM_X)\mapsto(X,X(0))$ establish a bijection between toroidal schemes and logarithmically regular logarithmic schemes.

(ii) In the case of varieties over an algebraically closed field, toroidal schemes are the classical toroidal embeddings without self intersections of \cite{KKMS}. Furthermore, \'etale logarithmically regular varieties correspond to general toroidal embeddings, possibly with self intersections.
\end{remark}

\subsubsection{The center}
By the {\em center} of a toroidal scheme we mean the center of the associated logarithmic scheme.

\begin{lemma}\label{toroidalcenter}
The center $C$ of a toroidal scheme $(X,U)$ is regular.
\end{lemma}
\begin{proof}
In fact every stratum is regular. This follows by unraveling the definitions.
\end{proof}

\subsubsection{Formal-local description}\label{Katosec}
Let $X=\Spec(A)$ be a local fs logarithmic scheme with closed point $x\in X$; write  $M:=\ocM_{X,x}$ and $n:=\rk M$. By Lemma~\ref{sharpchartlem} there exists a monoidal chart $f\:X\to Y=\bfA_M$ of $(X,\cM_X)$. Clearly, $y=f(x)$ lies in the center $Y(n)=V(M^+\ZZ[M])$ of $Y$ and, since $X$ is local, $X(n)=X\times_YY(n)$ is the center of $X$. In particular, $(X,\cM_X)$ is logarithmically regular at $x$ if and only if $X(n)$ is regular and $\dim X=\dim X(n)+n$. To make this explicit, Kato gave the following formal-local characterization of logarithmic regularity,  see \cite[Theorem 3.2]{Kato-toric}.

Recall that if $k$ is a field of characteristic $p>0$ then Cohen ring $C(k)$ is a DVR with maximal ideal $(p)$ and residue field $k$. If $\mathrm{char}(k)=0$ then, as in \cite[\S2.2.10]{ATLuna}, we set $C(k)=k$.

\begin{lemma}\label{Katolocal}
Let $M$ be a sharp toric monoid, $X=\Spec(A)$ a local scheme with closed point $x$ and $f\:X\to Y=\bfA_M$ a morphism such that $\cX=\Spec(A/M^+A)$ is regular, contains $x$, and satisfies the equality $\dim X=\dim \cX+\rk(M)$. Fix a regular system of parameters $t_1',\ldots,t_r' \in \cO_{\cX,x}$ and lifts $t_i\in A$. Let $C_x=C(k(x))$ be a Cohen ring of $k(x)$ and similarly $C_y=C(k(y))$ for $y=f(x)$. Then the completion of $R=\cO_{Y,y}\to A$ factors as
$$\xymatrix@R=3mm{C_y\llbracket M\rrbracket & C_x\llbracket M\rrbracket\llbracket t_1\. t_r\rrbracket \\
 \hatR\ar@{=}[u] \ar[r]^\phi & B\ar@{=}[u] \ar@{->>}[r]^\psi & \hatA
}$$
where $\phi$ is formally smooth and $\Ker(\psi)$ is generated by a single element $\theta$ such that
\begin{itemize}
\item[(1)] If ${\rm char}\ k(x)=0$ then $\theta=0$.
\item[(2)] If ${\rm char}\ k(x)=p>0$ then $\theta \equiv p \mod (M^+, t_1\. t_r)$.
\end{itemize}
If ${\rm char}\ k(x)=p>0$, any element $\theta'\in\Ker(\psi)$ that satisfies Condition (2) is a generator of $\Ker(\psi)$.
\end{lemma}
We refer to \cite[Theorem 3.2]{Kato-toric} for a proof. Here we recall the main line: fix a homomorphism $C_x\to\hatA$ inducing an isomorphism of the residue fields. Then the lifts $t_i\in A$  of the regular system of parameters $t'_1\. t'_r\in \cO_{\cX,x}$ induce a  homomorphism $B\to\hatA$ which is observed to be surjective. Its kernel can be described by dimension considerations.

\subsubsection{Logarithmic regularity is a local property}\label{logregsec}
It is a non-trivial fact that logarithmic regularity is preserved under generizations. Kato proved this in \cite[Proposition~7.1]{Kato-toric}, but Gabber noticed that his argument was insufficient and a completed argument can be found in \cite[Theorem~9.5.47]{Gabber-Ramero} or \cite[Theorem IV.3.6.2]{Ogus-logbook}. An equivalent reformulation of this generization result is provided by the following theorem. For completeness, we provide a new proof, which is a little computational but shorter and more elementary than the mentioned arguments.

\begin{theorem}\label{localth}
Let $M$ be a sharp toric monoid, $X=\Spec(A)$ a local scheme with closed point $x$ and $f\:X\to Y=\bfA_M$ a morphism such that $\cX=\Spec(A/M^+A)$ is regular, contains $x$, and satisfies the equality $\dim X=\dim \cX+\rk(M)$. Set $U=f^{-1}(\bfA_{M^\gp})$. Then $(X,U)$ is a toroidal scheme and $f$ gives rise to a global sharp central monoidal chart.
\end{theorem}
\begin{proof} {\sc Step 1: \em restatement in terms of strata.}
We provide $X$ and $Y$ with the logarithmic structures induced by $M$. Then $f$ becomes a strict morphism of log schemes, and hence the log strata are compatible: $X(n)=Y(n)\times_YX$. Clearly, $U=X(0)$, so we should prove that each $X(n)$ is regular and $\dim(\cO_{X,x})=n+\dim(\cO_{X(n),x})$ for $x\in X(n)$. We can work locally over a point $y\in Y(n)$.

{\sc Step 2: \em coordinate rings of  strata.} Note that $M\cap\cO_{Y,y}^\times$ is a face $F$ of $M$. Let $N=M[-F]$ denote the localization of $M$ at $F$. Then $\rk(\oN)=n$ and $\rk(F)=\rk(M)-n$. Note that $y$ lies in the center of the localization $$Y_F=\Spec(\ZZ[M]_F)=\Spec(\ZZ[N]).$$ Since $\rk(\oN)=n$, we have that the center is $Y_F(n)=Y(n)\cap Y_F$, and hence it suffices to prove that $X_F(n)=X\times_YY_F(n)$ is regular and satisfies the equality
$$\dim(\cO_{X,x})=n+\dim(\cO_{X_F(n),x})\quad \text{ for }\quad x\in X_F(n).$$

By definition, $Y_F(n)=\Spec(\ZZ[N]/(N^+))$. Write $I=M\setminus F$. Since $N^+\ZZ[M]_F=I\ZZ[M]_F$  and $\ZZ[M]/I\ZZ[M]) = \ZZ[F]$ we have that
\begin{align*}Y_F(n)=\Spec(\ZZ[N]/(N^+))&=\Spec\left((\ZZ[M]/I\ZZ[M])_F\right)\\ &= \Spec\left(\ZZ[F]_F\right) = \Spec \ZZ[F^\gp].\end{align*}
Also the ideal $I\ZZ[M]$ is of height $n$: the ideal $I$ is of height $n$ in $M$ and a maximal chain of prime ideals of $M$ contained in $I$ gives rise to a maximal chain of prime ideals contained in $I\ZZ[M]$.

{\sc Step 3: \em passing to completions.} Let $\hatA$ denote the completion of $A=\cO_{X,x}$. Then the morphism $h\:\hatX=\Spec(\hatA)\to X=\Spec(A)$ is faithfully flat, and so it suffices to prove that $\hatX_F(n)=\hatX\times_YY_F(n)$ is regular and of pure codimension $n$. Set $C=C(k(x))$. By Lemma~\ref{Katolocal}, the homomorphism $\ZZ[M]\to\hatA$ factors through a regular homomorphism $$\ZZ[M]\ \ \to\ \  B:=C\llbracket M\rrbracket\llbracket \ut\rrbracket,$$ where $\ut=(t_1,\dots,t_{r})$ and $B\to\hatA$ is surjective with kernel $(\theta)$. In particular, the morphism $$\hatZ:=\Spec(B)\ \ \to\ \  Y$$ is regular.

Recall that $Y_F(n)=\Spec(\ZZ[M]/(I))_F$; writing  $\hatZ_F(n)_T:=\hatZ\times_YY_F(n)$ we have  $\hatZ_F(n) = \Spec(D)$, where $$D=(B/IB)_F=C\llbracket\ut\rrbracket\llbracket F\rrbracket_F.$$
The completion homomorphism $C[\ut][F]\to C\llbracket\ut\rrbracket\llbracket F\rrbracket$ is regular, hence its localization $\phi\:C[\ut][F]_F\to D$ is regular. Since the source $C[\ut][F]_F=C[\ut][F^\gp]$ of $\phi$ is regular, we obtain that the ring $D$ and its spectrum  $\hatZ_F(n)$ are regular. Since $Y_F(n)$ is of pure codimension $n$ in $Y$, we also have that $\hatZ_F(n)$ is of pure codimension $n$ in $\hatZ$.

{\sc Step 4: \em the equation $\theta$.} In case (1) of Lemma~\ref{Katolocal}, $\hatX=\hatZ$ hence the morphism $\hatX\to Y$ is regular and we are done, so assume we are in case (2). By our construction, $\hatX_F(n)$ is the vanishing locus of $\theta$ in $\hatZ_F(n)$, hence we should prove that $\theta$ defines a regular subscheme of codimension 1 in the regular scheme $\hatZ_F(n)=\Spec(D)$.

It now suffices to prove that if $z\in\hatX_F(n)\subset \hatZ_F(n)$ and $q\subset D$ is the corresponding prime ideal then $\theta\notin q^2$. Recall that $D$ is a localization of its subring $D^0=C\llbracket \ut\rrbracket\llbracket F\rrbracket$. Since $\theta\in D^0$ and $D^0$ is a local ring with maximal ideal $m=(p,F^+,t_1\.t_r)$, it suffices to check that $\theta\notin m^2$, or even that its image  $$\bar\theta\ \ \in\ \  D^0/ (m^2+(F^+,t_1\.t_r)) \ = \ C/p^2$$ is nonzero. But $\bar\theta = p\neq 0\in C/p^2$, as needed.
\end{proof}

\subsubsection{Toroidal divisors and subschemes}\label{subschsec}
By a {\em toroidal subscheme} of $(X,U)$\index{toroidal!subscheme} we mean any subscheme $Z$ of the form $V(I\cO_X)$ where $I$ is an ideal in $\cM_X$. If, in addition, $Z$ is a divisor then we call it a {\em toroidal divisor}.\index{toroidal!divisor} The following lemma implies, in particular, that the irreducible components of $X\setminus U$ are the same as integral toroidal divisors.

\begin{lemma}\label{toroidaldivisorslem}
Assume that $(X,U)$ is a toroidal scheme and $X$ is local with closed point $x$. Let $\cM_X$ be the logarithmic structure of $(X,U)$, let $M=\ocM_{X,x}$ and let $q\:\cM_{X,x}\to M$ be the quotient map. Then,

(i) The map $I\mapsto V(q^{-1}(I)\cO_X)$ gives rise to a bijection between ideals of $M$ and toroidal subschemes of $(X,U)$, with inverse bijection given by $V(J)\mapsto q(J_x\cap\cM_{X,x})$. Moreover, $I$ is prime if and only if $V(q^{-1}(I)\cO_X)$ is integral, and in this case
$$\height(I) = \codim(V(q^{-1}(I)\cO_X)).$$
(ii) The map $\fp\mapsto V(q^{-1}(\fp)\cO_X)$ establishes a bijection between prime ideals of height 1 in $M$ and the irreducible components of $D:=X\setminus U$.

(iii) The map $m\mapsto V(q^{-1}m)$ gives rise to an isomorphism of $M$ and the monoid of effective toroidal Cartier divisors of $X$.
\end{lemma}
\begin{proof}
Claim (iii) follows from the definition of $\cM_X$. Note also that $q$ induces a bijection between the ideals of $\cM_{X,x}$ and $M$, so it suffices to consider the ideals of $\cM_{X,x}$ in (i) and (ii).

The first claim of (i) is proved in \cite[Lemma~VIII.3.4.3]{Illusie-Temkin}. Furthermore, if $I\subset\cM_{X,x}$ is prime then $V(I\cO_X)$ is integral of the asserted codimension by \cite[Corollary~7.3]{Kato-toric}. Conversely, $I=I\cO_{X,x}\cap\cM_{X,x}$ by the first assertion. Hence if  $V(I\cO_X)$ is integral then $I\cO_{X,x}$ is prime, implying that $I=I\cO_{X,x}\cap\cM_{X,x}$ is prime.

Finally, we prove (ii). Let $H$ be the set of prime ideals of $M$ of height 1. By (i), $V(\fp\cO_X)$ is an integral divisor for any $\fp\in H$, and we have only to prove that any irreducible component of $X\setminus U$ is of this form. Note that $I_D:=\cap_{\fp\in H}\fp$ is precisely the ideal consisting of all inner elements of $M$ (see \S\ref{innerelementsec}), hence the support of $V(I_D\cO_X)$ coincides with $D$. By \cite[Prop. 6.4]{Kato-toric}, $I_D\cO_{X,x}=\cap_{\fp\in H}\fp\cO_{X,x}$ and we obtain that $D$ is the union of the integral divisors $V(\fp\cO_X)$.
\end{proof}

\begin{corollary}\label{centercor}
Keep the assumptions of Lemma~\ref{toroidaldivisorslem}. Then the center $C$ of $(X,U)$ is the scheme-theoretic intersection of the irreducible components of $X\setminus U$.
\end{corollary}
\begin{proof}
Note that $C=V(q^{-1}(M^+)\cO_X)$ and $M^+=\bigcup_{\fp\in H}\fp$, where $H$ is the set of prime ideals of $M$ of height one. Thus $C=V(\sum_{\fp\in H}q^{-1}(\fp)\cO_X)$; by Lemma~\ref{toroidaldivisorslem}(ii), the subschemes $V(q^{-1}(\fp)\cO_X)$ are precisely the irreducible components of $X\setminus U$, and the corollary follows.
\end{proof}

\begin{theorem}\label{decreaseth}
Keep the assumptions of Lemma~\ref{toroidaldivisorslem}. Let $Z$ be an irreducible component of $D=X\setminus U$ with corresponding prime ideal $\fp\subset M$ (\ref{toroidaldivisorslem}(ii)) and let $D'=D-Z$ and $F=M\setminus\fp$. Then the following conditions are equivalent:

(i) $(X,D')$ is a toroidal scheme,

(ii) $Z$ is Cartier,

(iii) the facet $F$ splits off, say, $M=F\oplus\NN e$.
\end{theorem}

\begin{proof}
Set $r=\rk(M)$. Let $C=V(q^{-1}(M^+)\cO_X)$ denote the center of $(X,U)$ and set $C'=V(q^{-1}(F^+)\cO_X)$. Before establishing the equivalences, we note that since logarithmic regularity is a local property by \ref{logregsec}, it suffices to study the logarithmic strata through $x$.

(ii)$\implies$(iii) Recall that the elements of $M$ correspond to toroidal Cartier divisors by Lemma~\ref{toroidaldivisorslem}(iii), and the elements of $F$ correspond to divisors supported on $D'$. If $Z$ is Cartier then $Z=V(q^{-1}(e))$ and any toroidal Cartier divisor is a sum of $nZ$ and a toroidal Cartier divisor supported on $D'$, hence $M=F\oplus\NN e$ and we obtain (iii).

(iii)$\implies$(i) The equality $M=F\oplus\NN e$ implies that $C=C'\times_XZ$. Furthermore, the codimension of $C'$ in $X$ is at most $\rk(F)=r-1$. On the other hand, $C$ is regular and of codimension $r$ since $(X,U)$ is toroidal. Therefore, $C'$ is regular of codimension precisely $\rk(F)$, and thus $(X,D')$ is toroidal.

(i)$\implies$(ii) Assume $(X,D')$ is toroidal. By Corollary~\ref{centercor}, $C'$ is the scheme-theoretic intersection of all irreducible components of $D'$, hence $C'$ corresponds to the ideal $I=\bigcup_{\fq\in H'}\fq$, where $H' := H\setminus \{\fp\}$ is the set of all prime ideals of $M$ of height one that differ from $\fp$. By our assumption, $C'$ is regular. So, by Lemma~\ref{toroidaldivisorslem}(i), $I$ is a prime ideal of $M$. By the equivalence of (d) and (e) in Lemma~\ref{splitfacet}, $\fp$ is principal, say $\fp=(e)$, and then $Z=V(q^{-1}(e))$ is Cartier.
\end{proof}

\subsubsection{Morphisms of toroidal schemes and toroidal morphisms}\label{Sec:toroidal-morphisms}
By a {\em morphism} $f\:(X',U')\to(X,U)$ of toroidal schemes we mean any morphism $f\:X'\to X$ with $f(U')\subseteq U$. Note that $f$ induces a morphism $h\:(X',\cM_{X'})\to(X,\cM_X)$ of the corresponding logarithmic schemes. If $h$ is logarithmically smooth then we say that $f$ is a {\em toroidal morphism}.\index{toroidal!morphism} This generalizes \cite[Definition 1.2]{AK}, where complex varieties were considered.

\subsubsection{Toroidal charts}
If $(X,U)$ is a toroidal scheme and $M\to\cO_X(X)$ gives rise to a global monoidal chart then we obtain a strict morphism of toroidal schemes $(X,U)\to(\bfA_M,\bfA_{M^\gp})$ that will be called a {\em global toroidal chart} of $(X,U)$. 

\section{Toroidal actions}\label{Section:toroidalactions}
Next we study actions of diagonalizable groups on toroidal schemes. We will be especially interested in a special class of actions that generalize toroidal actions from \cite{AKMW}.

\subsection{Actions on logarithmic schemes}
We start with the most general observations that apply to actions on arbitrary logarithmic schemes.

\subsubsection{The definition}
An action of a group scheme $G$ on a logarithmic scheme $(X,\cM_X)$ consists of an action $m\:G\times X\to X$ on $X$ and an isomorphism $\phi\:p^{-1}\cM_X\toisom m^{-1}\cM_X$, where $p\:G\times X\to X$ is the projection and the pullbacks of $\phi$ to $G\times G\times X$ satisfy the usual cocycle condition. In fact, one can view $G$ with the trivial logarithmic structure as a group object in the category of logarithmic schemes, and then this data reduces to an action within this category. In the following we use the notion of {\em strongly equivariant  morphism} between schemes provided with  relatively affine actions,  see \cite[Sections 5.1, 5.3]{ATLuna}.\index{relatively affine action, (\cite[5.1, 5.3]{ATLuna})}

\subsubsection{Gradings of monoids}
By an {\em $L$-grading} on a monoid $M$ we simply mean a homomorphism $\chi\:M\to L$. Such a homomorphism induces an action of the Cartier dual group $\bfD_L$ on the scheme $\bfA_M$, which factors through the action of the group $\bfD_{M^\gp}$.\index{grading (of a monoid)}

\subsubsection{Equivariant charts}\label{equivchartsec}
Assume that $G=\bfD_L$ acts on the logarithmic scheme $(X,\cM_X)$. By a {\em (strongly) equivariant monoidal chart}\index{chart!(strongly) equivariant} we mean a $G$-equivariant open subscheme $V\into X$ and a (strongly) $G$-equivariant strict morphism $f\:(V,\cM_X|_V)\to (\bfA_M,\cM_{\bfA_M})$, where $G$ acts on the target via a grading $h\:M\to L$. Equivalently, an equivariant chart consists of a monoidal chart $f$ and a grading $h\:M\to L$ such that the corresponding homomorphism $\phi\:M\to\cO_X(V)$ is {\em homogeneous}, i.e. takes each $h^{-1}(l)$ to the $l$-homogeneous component of $\cO_X(V)$.

One may wonder when an action possesses a (strongly) equivariant chart. This naturally leads to the definitions of $G$-simple and toroidal actions below.

\subsubsection{$G$-simple actions}\label{Sec:G-simple}
Assume that a group scheme $G$ acts on a logarithmic scheme $(X,\cM_X)$. We say that the action is {\em $G$-simple at}\index{Gsimple action@$G$-simple action} a point $x\in X$ if $G_x$ acts trivially on $\ocM_{X,x}$. The action is {\em $G$-simple} if it is $G$-simple at all points of $X$. We will also say that $G$ \emph{acts simply} when the action is $G$-simple, and even say that the action is {\em simple} if the group $G$ is understood.

\begin{remark}\label{simplerem}
(i) If $G$ is connected then any action is $G$-simple. In general, $G_x$ acts on the finitely generated monoid $\ocM_{X,x}$ through the quotient by its connected component, i.e. through a finite group.

(ii) Our definition is taken from \cite[Exp. VI, 3.1(ii)]{Illusie-Temkin}, where it appears without name. The terminology differs slightly from $G$-strict actions of \cite[7.1]{dejong-curves}, which is concerned with regular schemes with simple normal crossings divisors acted on by a finite group. Since the word ``strict" conflicts with logarithmic strict morphisms, we prefer ``$G$-simple", in analogy with simple normal crossings divisors.
\end{remark}

\begin{lemma}\label{simplelem}
Let $G=\bfD_L$ be a diagonalizable group. Then,

(i) Assume that $f\:(Y,\cM_Y)\to(X,\cM_X)$ is a \emph{strict} $G$-equivariant morphism between fine logarithmic schemes. If the action on $(X,\cM_X)$ is $G$-simple then the action on $(Y,\cM_Y)$ is $G$-simple, and the converse is true whenever $f$ is surjective.

(ii) Assume that $M$ is a toric monoid and $h\:M\to L$ is a homomorphism. Then the induced action on $(\bfA_M,\cM_{\bfA_M})$ is $G$-simple.

(iii) Assume that $G$ acts on an fs logarithmic scheme so that there exists a global equivariant monoidal chart. Then the action is $G$-simple.
\end{lemma}
\begin{proof}
To prove (i) it suffices to note that $f$ induces an equivariant isomorphism of stalks $\ocM_{X,f(y)}=\ocM_{Y,y}$. The group $G$ acts on $\bfA_M$ through the torus $\bfA_{M^\gp}$, which is connected; hence the action is $G$-simple and we obtain (ii). Finally, (iii) follows from (i) and (ii).
\end{proof}

\subsubsection{Toroidal actions}\label{toroidalactionsec}
For simplicity, we assume in the following definition that the logarithmic strata $X(r)$ are reduced. In particular, this assumption is satisfied for toroidal schemes, the case to which we will restrict beginning Section \ref{Sec:existence-equivariant-charts} below. Consider the logarithmic stratum $Z=X(r(x))$ through a point $x\in X$, and let $Z_x$ be its localization at $x$. We say that the action of $G$ on $X$ is {\em toroidal at} $x$ if it is $G$-simple and $G_x$ acts trivially on $Z_x$.  The action is {\em toroidal}\index{toroidal!action} if it is toroidal at all points. This is compatible with the situation when $X$ is toroidal in Section \ref{Sec:toroidal-action-toroidal} below and with the terminology of \cite{Abramovich-deJong,AKMW}.

\begin{remark}\label{toroidalactionrem}
The same notion was introduced by Gabber in \cite[Exp. VI, 3.1]{Illusie-Temkin} under the name ``very tame action". This terminology seems not ideal as the tameness condition \cite[Exp. VI, 3.1(i)]{Illusie-Temkin} simply means that the stabilizers are diagonalizable. We prefer to replace ``very tame" with ``toroidal".
\end{remark}

Recall (e.g. \cite[IV.1.8]{Knutson}) that a $G$-equivariant morphism $f\:Y\to X$ is called {\em fixed-point reflecting} if $G_y=G_x$ for any $y\in Y$ with $x=f(y)$.\index{fixed-point reflecting morphism}

\begin{lemma}\label{toroidalactionlem}
Let $G=\bfD_L$ be a diagonalizable group.

(i) Assume that $G$ acts simply on an fs logarithmic scheme $(X,\cM_X)$ with reduced logarithmic fibers. The action is toroidal if and only if the stabilizer groups $G_x$ are locally constant along each logarithmic stratum $X(r)$.

(ii) Assume that $f\:(Y,\cM_Y)\to(X,\cM_X)$ is a strict fixed-point reflecting $G$-equivariant morphism between fine logarithmic schemes. If the action on $(X,\cM_X)$ is toroidal then the action on $(Y,\cM_Y)$ is toroidal, and the converse is true whenever $f$ is surjective.

(iii) Assume that $M$ is a toric monoid and $h\:M\to L$ is a homomorphism. Then the induced action of $G$ on $(\bfA_M,\cM_{\bfA_M})$ is toroidal.

(iv) Assume that $G$ acts on an fs logarithmic scheme so that there exists a global strongly equivariant monoidal chart. Then the action is toroidal.
\end{lemma}
\begin{proof}
Part (i) is obvious. Let us prove (ii). Being strict, $f$ preserves the logarithmic strata: $Y(r)=X(r)\times_XY$. By our assumption, $f$ also preserves the stabilizers. Hence (ii) follows from (i) and Lemma~\ref{simplelem}(i).

The action in (iii) is $G$-simple by Lemma~\ref{simplelem}(ii), hence we should only check that the stabilizers are constant along the connected components of the logarithmic strata of $\bfA_M$. The closed logarithmic stratum of $\bfA_M$ is its center $V(M^+)=\Spec\ZZ[M^\times]$, and all orbits in $V(M^+)$ have the same stabilizer $\bfD_{L/\phi(M^\times)}$. The general case reduces to this one because $\bfA_M\setminus V(M^+)$ is the union of the logarithmic schemes $\bfA_{M[-m]}$ for $m\in M^+$. Finally, (iv) follows from (ii) and (iii) because any strongly equivariant morphism is fixed-point reflecting.
\end{proof}

\subsection{Equivariant charts}\label{Sec:existence-equivariant-charts}

\subsubsection{Local actions}
Assume that $G$ acts in a relatively affine manner on $X$. Recall that an action of a diagonalizable group $G=\bfD_L$ on $X$ is called {\em local} \index{local action} if there exists a single closed orbit, see \cite[5.1.9]{ATLuna}. This happens if and only if $X=\Spec(A)$ for an {\em $L$-local} ring $A$, i.e. an $L$-graded ring with a single maximal homogeneous ideal $m$, see \cite[4.4]{ATLuna}. A local action is {\em strictly local} \index{strictly local action} if the closed orbit is a point; this happens if and only if $A/m$ is a field, see \cite[4.5]{ATLuna}. Note, however, that $A$ does not have to be a local ring in the usual sense since there might be maximal but non-homogeneous ideals.

\subsubsection{Special orbits and $G$-localization}\label{specialsec}
Recall that any fiber $q^{-1}(y)$ of the quotient morphism $q\:X\to Y=X\sslash G$ contains a single closed orbit $O=O_y$ called {\em special} and $X_O=X\times_Y\Spec\cO_{Y,y}$ is called the equivariant localization or {\em $G$-localization} at $O$ (see \cite[\S5.1]{ATLuna}). In particular, $X$ is covered by its equivariant localizations $X_O$. On the level of sets $X_O$ consists of all orbits whose closure contains $O$. In the sequel, we will denote by $G_O$ the stabilizer of $O$ and we will denote by either $L_O$ or $L_y$ the group of characters of $G_O$.\index{special orbit}\index{Glocaization@$G$-localization}

\begin{lemma}\label{toroidalcenteraction}
Let $(X,U)$ be a toroidal scheme with center $C$, and assume that $G=\bfD_L$ acts locally on $X$ with closed orbit $O$ so that $U$ is $G$-equivariant. Then the $G$-equivariant maps $\pi_0(O)\to\pi_0(C)\to\pi_0(X)$ are surjective. In particular, $G$ acts transitively on the set of connected components of $C$ and $X$, and if $O$ is connected (e.g. the action is strictly local) then $C$ and $X$ are integral.
\end{lemma}
\begin{proof}
Recall that the connected components of $X$ and $C$ are integral since $C$ is regular by Lemma~\ref{toroidalcenter} and $X$ is normal by \cite[Theorem 4.1]{Kato-toric}. It suffices to show that the maps $\pi_0(O)\to\pi_0(C)$ and $\pi_0(O)\to\pi_0(X)$ are surjective. Assume to the contrary that $V$ is a connected component of $C$ or $X$ which is disjoint from $O$. Then the union $\oV$ of all $G$-translates of $V$ is disjoint from $O$. On the other hand, $\oV$ is a closed $G$-equivariant subscheme and hence contains $O$, a contradiction.
\end{proof}

In case of such local action of $G$ on $X$, the quotient group is $G/G_O=\bfD_{K_O}$ with $K_O = \Ker(L \to L_O)$. If $X = \Spec A$ and $O = \Spec A'$ then $A'$ has no non-trivial $G$-homogeneous ideals. In particular for every $n\in K_O$ the homogeneous submodule  $A'_n$ contains a unit. On the other hand, elements having weight in $L \setminus K_O$ vanish on $O$ hence they are not units in $A$. This implies the following  lemma:

\begin{lemma}\label{Lem:unit-criterion} The submodule  $A_n$ contains a unit if and only if $n\in \Ker(L \to L_O)$. \end{lemma}

\subsubsection{$G$-simple actions on toroidal schemes}\label{simplesec}
We introduced $G$-simple actions on logarithmic schemes in Section \ref{Sec:G-simple}. Here we study $G$-simple actions on \emph{toroidal} schemes.

\begin{lemma}\label{simpletoroidallem}
Assume that a diagonalizable group $G$ acts on a toroidal scheme $(X,U)$ in a relatively affine manner. Let $x\in X$ be a point with stabilizer $G_x=\bfD_{L_x}$ and $X_{x,G_x}=\Spec(A)$ the $G_x$-localization at $x$. Then the action is $G$-simple at $x$ if and only if any integral toroidal Weil divisor of $X_{x,G_x}$ is $G_x$-equivariant.
\end{lemma}
\begin{proof}
We can replace $X$, $U$, and $G$ with $X_{x,G_x}$, $U\times_XX_{x,G_x}$, and $G_x$, respectively, so that the action on $X=\Spec(A)$ is strictly local. We will now use Lemma~\ref{toroidaldivisorslem}, but this should be done carefully since the scheme $X$ does not have to be a local scheme. By \ref{toroidaldivisorslem}(iii), the action is $G$-simple if and only if all effective toroidal Cartier divisors are equivariant. This implies the inverse implication. Conversely, assume that a toroidal integral divisor $E$ is not equivariant. The union of all $G$-translates of $E$ is closed and equivariant, therefore it contains $x$, hence $x\in E$. Applying Lemma~\ref{toroidaldivisorslem}(ii) we obtain that $G$ acts non-trivially on the set of prime ideals of $\ocM_{X,x}$, hence it acts non-trivially on $\ocM_{X,x}$ and the action is not $G$-simple.
\end{proof}

\subsubsection{The graded monoid  $\oM_O$ of a special orbit $O$}
If $G$ acts simply on $(X,U)$ then to any special orbit $O$ one can associate a canonical $L_O$-graded toric monoid $\oM_O$:

\begin{lemma}\label{simpletoroidallem2}
Assume that a toroidal scheme $(X,U)$ is provided with a relatively affine action of $G=\bfD_L$, and let $O$ be a special orbit with stabilizer $G_O=\bfD_{L_O}$ and $G$-localization $X_O=\Spec(A)$.

(i) The action induces a canonical isomorphism between all the stalks $\ocM_{X,x}$ with $x\in O$. In other words, the monoid $\oM_O:=\ocM_{X,x}$ is independent of the choice of $x\in O$.

(ii) Any element $\om\in\oM_{O}$ admits a lifting $f\in\cM_X(X_O)\subset A$ which is $l$-homogeneous for some $l\in L$. Moreover, the image of $l$ in $L_O$ is uniquely determined by $\om$ and the resulting map $\chi_O\:\oM_O\to L_O$ is a homomorphism.
\end{lemma}

\begin{proof}
Replacing $X$ by $X_O$ we can assume that the action is local. By Lemma~\ref{toroidalcenteraction} any connected component $O'$ of the orbit lies in a connected component $C'$ of the center. By Lemma~\ref{centerlemma} the stalks $\ocM_{X,x}$ are naturally identified for $x\in C'$, hence the same is true for $x\in O'$. Since $G$ acts transitively on $\pi_0(O)$, we also obtain isomorphisms between the stalks of $\ocM_X$ for different components of $O$, and these isomorphisms are unique because the action is simple. This proves (i).

By Lemma~\ref{toroidaldivisorslem}(iii), locally at a point $x\in O$ we can lift $\om$ to a toroidal Cartier divisor $E_x\subset\Spec(\cO_{X,x})$. The Zariski closure of $E_x$ is a Weil divisor $E\subseteq X\setminus U$ which is Cartier at $x$. Since the action is simple, there are finitely many disjoint $G$-translates of $E$ and their union $\oE$ is $G$-equivariant. By \cite[Proposition~1.6.5]{ATLuna} we obtain that $\oE=V(f)$ for a homogeneous element $f\in A_l$. Let us check that the image of $l$ in $L_O$ depends only on $\om$. If $h\in A_{n}$ is another lifting then $V(f)$ and $V(h)$ coincide at $x$ and hence $f\cO_{X,x}=h\cO_{X,x}$. Thus, the embedding of graded $A$-modules $fA\into fA+hA$ becomes an isomorphism after tensoring with $k(x)$, and hence $fA=fA+hA$ by the graded Nakayama's lemma \cite[Proposition~1.6.4(ii)]{ATLuna}. In the same manner, $hA=fA+hA$ and we obtain that $f=uh$ for a homogeneous unit $u\in A^\times$. Since $A$ is $L$-local we may apply Lemma \ref{Lem:unit-criterion},  hence the degree of $u$ lies in $\Ker(L \to L_O)$. Thus, the map $\chi_O\:\ocM_{X,x}\to L_O$ is well defined, and then it is obviously a homomorphism.
\end{proof}

\subsubsection{A local characterization of $G$-simple actions}
We proved in Lemma~\ref{simplelem}(iii) that any action admitting a toroidal chart is $G$-simple. Here is a result in the opposite direction.

\begin{proposition}\label{simpleprop}
Assume that a toroidal scheme $(X,U)$ is provided with a local $G$-simple action of a diagonalizable group $G=\bfD_L$. Then,

(i) There exists a sharp central equivariant toroidal chart $(X,U)\to(\bfA_{\oM_O},\bfA_{\oM^\gp_O})$, where $O$ is the closed orbit.

(ii) Let $K_O^\tor$ denote the torsion group of $K_O=\Ker(L\to L_O)$. If $N=|K_O^\tor|$ is invertible on $X$ then there exists a central equivariant toroidal chart $$(X,U)\times\bfD_{K_O^\tor}\longrightarrow(\bfA_{\oM_O\oplus K_O},\bfA_{\oM^\gp_O\oplus K_O})\times\Spec\ZZ[1/N]$$ which is fixed-point reflecting along the preimage of $O$. In particular, the chart is \'etale-local on $X$, and if $K_O$ is torsion free (e.g. if the action is strictly local or if $L$ is torsion free) then the chart is global on $X$.
\end{proposition}
\begin{remark}
Without the torsion assumption in part (ii), the same proof as given below shows that charts as in Proposition~\ref{simpleprop}(ii) exist without inverting $N$. In addition, they are only flat-local on $X$. Such charts may be useful but they are not toroidal in our sense since $\Spec(\ZZ[\oM_O\oplus K_O])$ is not toroidal over the primes that divide the torsion order of $K_O$.
\end{remark}
\begin{proof}
(i) By Lemma~\ref{simpletoroidallem2} we have a grading $\chi_O\:\oM_O\to L_O$. Since $\oM^\gp$ is a lattice, $\chi_O$ factors through a finer grading $\chi\:\oM_O\to L$ and we fix any such $\chi$. Recall that any $A_l$ with $l\in K_O$ contains a unit by Lemma~\ref{Lem:unit-criterion}. By Lemma~\ref{simpletoroidallem2} any $t\in \oM_O$ admits a homogeneous lifting $h(t)\in A_l$ and multiplying it by a unit from $A_{\chi(t)-l}$ we can achieve that $h(t)\in A_{\chi(t)}$, thereby obtaining a homogeneous map $h\:\oM_O\to A$. Since $h(t)$ is unique up to a unit of degree zero, for any $a,b\in \oM_O$ there exists a unit $u(a,b)\in A_0^\times$ such that $h(a)h(b)=h(a+b)u(a,b)$. Our next goal is to replace $h$ by a homogeneous homomorphism.

Let $Q$ be the set of all homogeneous elements of $A$ that are not zero divisors. Clearly, $Q$ is an integral monoid, i.e. $Q\subseteq Q^\gp$. We claim that for all $t\in \oM_O$ we have $h(t)\in Q$. Indeed,  by \cite[Th. 4.1]{Kato-toric} the scheme $X$ is normal, so it suffices to show that $h(t)$ does not vanish on any connected component $X'$ of $X$. This follows from Lemma~\ref{centerlemma}, since the component $X'$ intersects $O$, as  the action is assumed local.

Note that $\oM_O^\gp=\ZZ^r$ possesses a basis $t_1\.t_r\in \oM_O$. Let $\phi\:\oM_O^\gp\to Q^\gp$ be the homomorphism with $\phi(t_i)=h(t_i)$, then for any $t\in \oM_O$ the element $\phi(t)$ can be obtained from $h(t)$ by multiplying it with units of the form $u(a,b)^{\pm 1}$. In particular, $\phi(t)\in Q\subset A$ is another homogeneous lifting of $t$, and the map $\phi\:\oM_O\to A$ is a homomorphism.

We claim that $\phi$ is a homogeneous monoidal chart. By construction, the logarithmic structure induced by $\phi$ embeds into $\cM_X$, hence it suffices to show that the homomorphism $\phi_y\:\oM_O\to\ocM_{X,y}$ is onto for any $y\in X$. Furthermore, by Lemma~\ref{centerlemma}, we can replace $y$ with a generic point of its logarithmic stratum. By Lemma~\ref{toroidalcenteraction} $y$ specializes to a point $x\in O$ and it remains to use the fact that the cospecialization map $\oM_O=\ocM_{X,x}\to\ocM_{X,y}$ is surjective. We proved that $\phi$ is a sharp chart, and $\phi$ is central because its center contains $O$.

(ii) Set $A'=A\otimes\ZZ[K_O^\tor])$ so that $X\times\bfD_{K_O^\tor}=\Spec(A')$. Choose a splitting $K_O=\Lambda\oplus K_O^\tor$. By Lemma~\ref{Lem:unit-criterion} any $A_l$ with $l\in K_O$ contains a unit, hence we can find a homomorphism $\psi_\Lambda\:\Lambda\to A'^\times$ such that $\psi(l)\in A'_l$. Combining $\psi_\Lambda$ with the composed map $K_O^\tor\into\ZZ[K_O^\tor]\to A'^\times$ we obtain a homogeneous homomorphism $\psi\:K_O\to A'^\times$. In particular, a homogeneous monoidal chart $(\phi,\psi)\:\oM_O\oplus K_O\to A'$ arises and it induces an equivariant toroidal chart of $(X,D)\times\bfD_{K_O^\tor}$. The chart is central because $\bfA_{\oM_O}$ is a central toroidal chart of both $X$ and $\bfA_{\oM_O\oplus K_O}$, and the chart is fixed-point reflecting along the preimage of $O'$ because all relevant stabilizers are easily seen to be equal to $\bfD_{L_O}$.
\end{proof}

\subsubsection{Toroidal actions on toroidal schemes}\label{Sec:toroidal-action-toroidal}
We defined toroidal actions of diagonalizable groups on a wide class of logarithmic schemes in Section \ref{toroidalactionsec}. In the particular case when the group acts on a toroidal scheme, one obtains a generalization of the definition of toroidal action of \cite{AKMW}. A basic example of a toroidal action is provided by Lemma~\ref{toroidalactionlem}(iii).

\subsubsection{Formal-local description of toroidal actions}
The following result will be our main tool for studying toroidal schemes with toroidal actions; it extends Lemma~\ref{Katolocal} to the equivariant case.

\begin{theorem}\label{flcilem}
Assume that a diagonalizable group $G=\bfD_L$ acts strictly locally on a toroidal scheme $(X,U)$ so that the action is toroidal at the $G$-invariant closed point $x$.\footnote{For example, this is automatically the case when the center of $X$ coincides with $x$.} Assume that $f\:(X,U)\to(\bfA_{M},\bfA_{M^\gp})$ is a sharp central equivariant chart and write $y=f(x)$. Let $X=\Spec(A)$, let $\hatA$ be the completion of $A$ at $m_x$, let $C_x$ and $C_y$ be the Cohen rings of $k(x)$ and $k(y)$,  and let $R=C_y[M]$. Then,

(i) The factorization
$$\xymatrix@R=3mm{C_y\llbracket M\rrbracket & C_x\llbracket M\rrbracket\llbracket t_1\. t_r\rrbracket \\
 \hatR\ar@{=}[u] \ar[r]^\phi & B\ar@{=}[u] \ar@{->>}[r]^\psi & \hatA
}$$
in Lemma~\ref{Katolocal} can be chosen so that the homomorphisms $\phi$ and $\psi$ are homogeneous in the sense of \cite[Section 4.5.5]{ATLuna} and the elements  $\theta$ and $t_1\. t_r$ are of degree zero.

(ii) For any factorization as in (i) write $M_0$ for the the trivially graded part of $M$. Then we have a description of the invariant subring
$$\hatA_0=B_0/(\theta)=C_x\llbracket M_0\rrbracket\llbracket t_1\. t_r\rrbracket/(\theta),$$ and $\dim(\hatA_0)=\rk(M_0)+r$.

(iii) The chart $f$ is strongly equivariant.

(iv) The quotient $(X_0,U_0)$=$(X\sslash G,U\sslash G)$ is a toroidal scheme and the morphism of quotients $f\sslash G\:(X_0,U_0)\to(\bfA_{M_0},\bfA_{M_0^\gp})$ is its toroidal chart.
\end{theorem}
\begin{proof}
(i) By the equivariance, the completion $\widehat{\alpha}\:\hatR\to\hatA$ of $\alpha\:R\to A$ is homogeneous: $\hatA=\prod_{l\in L}\hatA_l$, $\hatR=\prod_{l\in L}\hatR_l$, and $\widehat{\alpha}$ respects these formal gradings, see \cite[Proposition 1.6.8]{ATLuna}. The action is toroidal at $x$, hence $G$ acts trivially on $\Spec(\cO_{\fX,x})$, where $\fX=C(X, \cM_X) = \Spec(A/M^+A)$ is the center of $X$. Thus, the regular parameters $t'_1\.t'_r\in\cO_{\fX,x}$ appearing in Lemma~\ref{Katolocal} are of degree zero and lifting them to homogeneous elements of $A$ we obtain a homogeneous factorization $\hatR\to C_x\llbracket M\rrbracket\llbracket t_1\. t_r\rrbracket\to\hatA$ with $t_i$ of degree zero. It remains to show that $\theta$ can be chosen of degree zero. When  ${\rm char}\ k(x)=0$ we have $\Ker(\psi)=0$ and the result is trivial, so assume that ${\rm char}\ k(x)=p>0$. First, choose any $\theta$ as in Lemma~\ref{Katolocal} and let $\theta_l$ be the homogeneous components of $\theta$. For any $0\neq l\in L$ we have that $B_l\subset M^+B$, hence $\theta_0\equiv\theta \equiv p \mod (M^+, t_1\. t_r)$. Since $\alpha\: B\to\hatA$ is homogeneous, $\theta_0$ is an element of $\Ker(\psi)$ and hence $\theta_0$ is also a generator by Lemma~\ref{Katolocal}. Thus, $\Ker(\psi)=(\theta_0)$ is as asserted.

(ii) Since $B_0=C_x\llbracket M_0\rrbracket\llbracket t_1,\ldots t_r\rrbracket$, the formula for $\hatA_0$ follows from (i). Also $$\dim(B_0)=\dim(C_x\llbracket M_0\rrbracket\llbracket t_1,\ldots t_r\rrbracket) =\rk(M_0)+r+\dim(C_x).$$ Note that $\dim(C_x)=0 = ht\left((\theta)\right)$ if ${\rm char}\ k(x)=0$ and $\dim(C_x)=1=ht\left((\theta)\right)$ otherwise. In either case $\dim(\hatA_0)=\rk(M_0)+r$ as required.

(iii) We should check that $\alpha$ is strongly homogeneous and by \cite[Lemma 4.6.2]{ATLuna} it suffices to check that the completion $\widehat\alpha$ is strongly homogeneous, i.e. $\hatA_0\wtimes_{\hatR_0}\hatR=\hatA$. But this holds since $\hatA=\hatR\wtimes_{C_x}C_y\llbracket t_1\. t_r\rrbracket/(\theta)$ by (i) and $\hatA_0=\hatR_0\wtimes_{C_x}C_y\llbracket t_1\. t_r\rrbracket/(\theta)$ by (ii).

(iv) Set $f_0=f\sslash G$, $(Y,V)=(\bfA_M,\bfA_{M^\gp})$ and $(Y_0,V_0)=(\bfA_{M_0},\bfA_{M_0^\gp})$. We verify that $f_0: X_0 \to Y_0$ satisfies the requirements of  Theorem~\ref{localth}, giving the statement. By part (iii), we have the following Cartesian square
$$
\xymatrix{
X\ar[r]^f\ar[d]_\alpha & Y\ar[d]^\beta\\
X_0\ar[r]^{f_0} &Y_0.
}
$$
Note that $U=f^{-1}(V)$,  $\beta(V) = V_0$, and $\alpha(U)=U_0$, hence $U_0=f_0^{-1}(V_0)$.

{\sc Step 1: \em  $\beta$ induces an isomorphism of centers $\fY\to\fY_0$. } Indeed, both are isomorphic to $\Spec(\ZZ)$ since $M$ and $M_0$ are sharp.

{\sc Step 2:  \em $\alpha$ induces an isomorphism of centers $\fX=\fY\times_YX$ and $\fX_0=\fY_0\times_{Y_0}X_0$}, and so $\fX_0$ is regular. Indeed $\fX \to \fX_0$ is a base change of $\fY\toisom\fY_0$.

{\sc Step 3: \em Computation of dimensions.}  Write $x_0=\alpha(x)$ and recall that $r=\dim \fX_x=\dim (\fX_0)_{x_0}$. Since $\cO_{X_0,x_0}=A_0$, we obtain from (ii) that $$\dim(\cO_{X_0,x_0})=\dim(\hatA_0)=\rk(M_0)+r=\rk(M_0)+\dim(\cO_{\fX_0,x_0}),$$ and hence Theorem~\ref{localth} applies to $f_0$.
\end{proof}

\subsubsection{Existence of strongly equivariant charts}
Recall that if $G$ acts on $(X,U)$ so that there exists a strongly equivariant chart then the action is toroidal, see Lemma~\ref{toroidalactionlem}(iv). In particular, Theorem~\ref{flcilem} applies to such an action. Moreover, Theorem~\ref{flcilem} implies that, strictly locally, existence of such charts characterizes toroidal actions:

\begin{corollary}\label{toroidalprop}
Assume that a diagonalizable group $G=\bfD_L$ acts locally on a toroidal scheme $(X,U)$ so that the action is toroidal along the closed orbit $O$. Then

(i) The action on the whole $(X,U)$ is toroidal.

(ii) Any central equivariant toroidal chart $(X,U)\to(\bfA_{M},\bfA_{M^\gp})$ which is fixed-point reflecting along $O$ is strongly equivariant.

(iii) If the torsion $K_O^\tor$ of $K_O=\Ker(L\to L_O)$ is invertible on $X$ then $(X,U)\times\bfD_{K_O^\tor}$ possesses a central strongly equivariant chart. In particular, $(X,U)$ possesses such a chart whenever $K_O$ is torsion free (e.g. the action is strictly local or $L$ is torsion free).
\end{corollary}
\begin{proof}
(i) Since $G_x\subseteq G_O$ for any $x\in X$, it suffices to prove that the action of $G_O$ is toroidal. Note that $X$ is covered by the $G_O$-localizations $X_x$ at points $x\in O$. The action of $G_O$ on $X_x$ is strictly local hence $(X_x,U_x)$ possesses a sharp central equivariant chart $h$ by Proposition~\ref{simpleprop}. Then $h$ is strongly equivariant by Theorem~\ref{flcilem}(iii), and the action of $G_O$ on $(X_x,U_x)$ is toroidal by Lemma~\ref{toroidalactionlem}(iv).

(ii) Since the chart is central, $G_O$ is the stabilizer of the center of $(Y,V)=(\bfA_{M},\bfA_{M^\gp})$. Therefore, the actions of $G/G_O$ on $X\sslash G_O$ and $Y\sslash G_O$ have trivial stabilizers and hence are free by \cite[Lemma~5.4.4]{ATLuna}. By \cite[Corollary~5.4.8]{ATLuna} the morphism $X\sslash G_O\to Y\sslash G_O$ is strongly $G/G_O$-equivariant, so it suffices to prove that the morphism $X\to Y$ is strongly $G_O$-equivariant.

Since the chart is central, we have an isomorphism of $L_O$-graded modules $\oM_O=\oM$ and a toric $G_O$-equivariant morphism $Y\to\oY=\bfA_{\oM_O}$ arises. Clearly, the induced morphism $(X,U)\to(\oY,\bfA_{\oM_O^\gp})$ is an equivariant sharp central toroidal chart. As we already observed, for each $x\in O$ the induced chart of the $G_O$-localization $(X_x,U_x)$ is strongly equivariant by Theorem~\ref{flcilem}(iii). Therefore, the morphism $X\to\oY$ is strongly $G_O$-equivariant. In the same way, $Y\to\oY$ is strongly $G_O$-equivariant and hence $X\to Y$ is so.

Finally, (iii) follows from (ii) and Proposition~\ref{simpleprop}.
\end{proof}

\subsubsection{Openness of toroidal actions}
As another corollary we obtain that the set of points of $(X,U)$ where $G$ acts toroidally is open.

\begin{corollary}\label{gencor}
Assume that a diagonalizable group $G$ acts on a toroidal scheme $(X,U)$. Then the set of points where the action is toroidal (resp. simple) is open.
\end{corollary}
\begin{proof}
We have the stratifications of $X$ by inertia strata $X(H)$ and logarithmic strata $X(r)$. Let $T$ be the set of points $x\in X$ where the action is toroidal, then $x\in T$ if and only if the irreducible component of the logarithmic stratum through $x$ is contained in $X(G_x)$. Thus $T$ is constructible, and it suffices to prove that $T$ is closed under generization.

Assume that $x$ is a generization of $y$. If the action is simple at $y$ then $G_y$, and hence also its subgroup $G_x$, act trivially on $\ocM_{X,y}$. The cospecialization homomorphism $\phi\:\ocM_{X,y}\to\ocM_{X,x}$ is equivariant. Since $\phi$ is surjective by \cite[Lemma~2.12(1)]{Niziol} we obtain that $G_x$ acts trivially on $\ocM_{X,x}$, i.e. the action at $x$ is trivial.

Assume now that the action is toroidal at $y$. Since $G_x\subseteq G_y$, the action of $G$ is toroidal at $x$ or $y$ if and only if the action of the subgroup $G_y$ is toroidal at $x$ or $y$. In particular, we can replace $G$ by $G_y$ and assume that $y$ is $G$-invariant. Furthermore, we can now replace $X$ by the $G$-localization at $y$  and assume  that the action is strictly local and $y$ is the closed $G$-invariant point. Then the action on $(X,U)$ is toroidal by Corollary~\ref{toroidalprop}.
\end{proof}

\subsection{Toroidal quotients}
In this section we define quotients of toroidal schemes and prove that they always exist for relatively affine toroidal actions. This is the main property of toroidal actions used in applications.

\subsubsection{The definition}
Assume that a toroidal scheme $(X,U)$ is provided with a relatively affine action of $G=\bfD_L$. Set $Y=X\sslash G$ and let $V$ be the image of $U$ in $Y$. If $V$ is open and $(Y,V)$ is a toroidal scheme then we say that $(Y,V)$ is the {\em quotient} of $(X,U)$ by $G$ and denote it $(X,U)\sslash G$.\index{toroidal!quotient} It satisfies the universal property analogous to categorical quotients: any $G$-equivariant morphism of toroidal schemes $(X,U)\to(Z,W)$ with the trivial action on the target factors uniquely through $(Y,V)$.

\subsubsection{Taut and loose quotients}
Recall that the equivariant open subscheme $U\into X$ is called {\em strongly equivariant} if it is the preimage of $V$ and in this case $V=U\sslash G$, see \cite[\S5.1.4]{ATLuna}. We say that the toroidal quotient $(Y,V)=(X,U)\sslash G$ is {\em taut}\index{taut!quotient} if $U$ is strongly equivariant. More generally, the quotient will be called {\em  loose}\index{loose!quotient} if $V=U\sslash G$.

\begin{remark}
Let $f\:X\to Y$ be the quotient morphism and consider the toroidal divisors $D=X\setminus U$ and $E=Y\setminus V$. The quotient is taut if and only if the inclusion $f^{-1}(E)\subseteq D$ is an equality, and this happens if and only if $D$ has no horizontal components, i.e. $f(D)$ contains no generic points of $Y$. One can show that for taut quotients the induced logarithmic morphism $(X,\cM_X)\to(Y,\cM_Y)$ satisfies the universal property of quotients in the category of logarithmic schemes. This last observation will not be used, so we omit its verification.
\end{remark}

\subsubsection{The toric case}
We will describe when the quotient is loose or taut in terms of the action itself. We start with a toric case, so let us introduce a relevant combinatorial notion for a grading $\phi:M\to L$. We say that $\phi$ is {\em taut} if its kernel $M_0$ contains an inner element $v\in M$ (\S\ref{toricmonoidsec}), and we say that $\phi$ is {\em loose} if the inclusion $M_0^\gp\subseteq K$ is an equality, where $K$ is the kernel of $\phi^\gp\:M^\gp\to L$. The following lemma is very simple so we skip the proof.

\begin{lemma}\label{gradinglem}
Let $M$ be a toric monoid, $L$ a lattice and $\phi\:M\to L$ a grading.

(i) $\phi$ is taut if and only if $\mathrm{Im}(\phi)=\mathrm{Im}(\phi^\gp)$.

(ii) $\phi$ is loose if and only if $\rk(M_0)+\rk(\mathrm{Im}(\phi))=\rk(M)$.
\end{lemma}

Now, let us relate these combinatorial definitions to the geometry.

\begin{lemma}\label{torictautquotient}
Assume that $G=\bfD_L$ acts on $(X,U)=(\bfA_M,\bfA_{M^\gp})$ via a grading $\phi\:M\to L$. Then,

(i) The following conditions are equivalent: (a) the quotient is taut, (b) $\phi$ is taut, (c) $U\subset X^s(0)$, the Geometric Invariant Theory stable locus with respect to the trivial linearization $0\in L$.

(ii) The following conditions are equivalent: (d) the quotient is loose, (e) $\phi$ is loose, (f) $\dim X\sslash G = \dim X -\dim  G/G_X$, where $G_X$ is the subgroup acting trivially on $X$.
\end{lemma}

\begin{remark}
At least when  $\phi^\gp$ is surjective these properties can be characterized in a manner similar to toric geometry, in terms of the dual monoid $\sigma = M^\vee$ and the dual map $(\phi^\gp)^\vee : L^\vee \hookrightarrow N = (M^\gp)^\vee$ as follows:

(i') The quotient is taut if and only if $L^\vee \cap \sigma=\{0\}$.

(ii') The quotient is loose if and only if $L^\vee \cap \sigma\subset\sigma$ is a whole face of $\sigma$.

We will not use this language in the sequel, but the reader may find it helpful and more intuitive.
\end{remark}

\begin{proof}
Let $Y=X\sslash G$ and let $V$ be the image of $U$ in $Y$.

(d)$\Longleftrightarrow$(e) By definition, $Y=\bfA_{M_0}$ and $U\sslash G=\bfA_{K}$. Thus the quotient is loose if and only if $M_0^\gp=K$, i.e. $\phi$ is loose.

(e)$\Longleftrightarrow$(f)  Note that $G/G_X=D_{L'}$ where $L'=\mathrm{Im}(\phi^\gp)$. Since $\dim Y=\rk(M_0)$ and $\dim X=\rk(M)$, the equivalence follows from Lemma~\ref{gradinglem}(ii).

(a)$\Longleftrightarrow$(b) The quotient is taut if and only if $D=X\setminus U$ is the preimage of $Y\setminus V$. This happens if and only if there exists a Cartier divisor $D'$ on $X$ such that: (1) $D'$ is a toroidal divisor induced from $Y$, and (2) $|D'|=|D|$. Note that a Cartier divisor on $X$ satisfies (1) if and only if it is of the form $V(v)$ for $v\in M_0$, and such divisor satisfies (2) if and only if $v$ is inner. Thus, (1) and (2) are satisfied if and only if $\phi$ is taut.

(b)$\Longleftrightarrow$(c) An inner element $m\in M_0$ is an invariant function vanishing on $X\setminus U$, implying $U\subset X^s(0)$. Conversely, if $U\subset X^s(0)$ then there is a homogeneous invariant element separating it from any other torus orbit. The product of finitely many of  these corresponds to an inner element in $M_0$.

\end{proof}

\subsubsection{Taut and loose actions}
Lemma~\ref{torictautquotient} motivates the following definition that applies to arbitrary toroidal actions. We will later see that it is compatible with the notion of taut and loose quotients. The same definition could be spelled out for an arbitrary simple action but we did not find the resulting notion meaningful, see Remark~\ref{potentialrem} below.

Assume that $G=\bfD_L$ acts toroidally on a toroidal scheme $(X,U)$. For any point $x\in X$ let $L_x$ be the quotient of $L$ such that $G_x=\bfD_{L_x}$. By Lemma~\ref{simpletoroidallem2}, the action induces a homomorphism $\phi_x\:\ocM_{X,x}\to L_x$. We say that the action is {\em taut at} $x$ if $\Ker(\phi_x)$ contains an inner element. We say that the action is {\em  loose at}\index{taut!action} $x$ if $(\Ker(\phi_x))^\gp=\Ker(\phi_x^\gp)$. The action is {\em taut} or {\em  loose}\index{loose!action} if it is so at all points of $X$.

\begin{example}\label{tautexam}
(i) If the action of $G$ on a toroidal scheme $(X,U)$ is toroidal at a point $x\in X$ and $G_x$ is finite then the action is taut at $x$.

(ii) In particular, if $G=\GGm=\bfD_\ZZ$ then the action can fail to be taut only at $G$-invariant points. Assume that $G_x=G$. Then the action is taut if and only if the image of $\phi_x\:\ocM_{X,x}\to\ZZ$ is unbounded on both sides. The action is  loose but not taut if and only if $\Ker(\phi_x)$ contains a facet of $\ocM_{X,x}$.
\end{example}

\begin{lemma}\label{tautlem}
Assume that $f\:(X',D')\to(X,D)$ is a strict $G$-equivariant morphism of toroidal schemes.

(i) If the action on $(X,D)$ is taut then the action on $(X',D')$ is taut.

(ii) Assume, in addition, that $f$ is fixed point reflecting. If the action on $(X,D)$ is loose then the action on $(X',D')$ is loose. Conversely, if the action on $(X',D')$ is taut (resp. loose) and $f$ is surjective then the action on $(X,D)$ is taut (resp. loose).
\end{lemma}
\begin{proof}
If $x'\in X'$, $x=f(x')$, $G_{x'}=\bfD_{L'}$ and $G_x=\bfD_L$ then $L'$ is a quotient of $L$. If the kernel of  $M\to L$ contains an inner element then the same holds for the composition $M\to L'$. This implies (i). The assertion of (ii) holds since $L=L'$ whenever $f$ is fixed point reflecting.
\end{proof}

\subsubsection{Existence and properties of toroidal quotients}
Here is our main result about toroidal quotients.

\begin{theorem}\label{toroidalth}
Assume that a toroidal scheme $(X,U)$ is provided with a relatively affine toroidal action of a diagonalizable group $G=\bfD_L$, and $x\in X$ is contained in the special orbit over a point $x_0\in X_0=X\sslash G$. Then,

(i) The toroidal quotient $(X,U)\sslash G$ exists, $\ocM_{X_0,x_0}=(\ocM_{X,x})_0$, the $0$-graded component with respect to the $L_x$-grading of $\ocM_{X_0,x_0}$, and the quotient is taut or loose if and only if the action is taut or  loose, respectively.

(ii) The morphism $\pi\:(X,U)\to(X,U)\sslash G$ is toroidal whenever the torsion degree $N$ of $L$ is invertible on $X$.

(iii) If $h\:(Z,W)\to(X,U)$ is a strongly equivariant strict morphism of toroidal schemes then the quotient $h\sslash G\:(Z,W)\sslash G\to(X,U)\sslash G$ is a strict morphism of toroidal schemes.

(iv) If $I\subset\cO_X$ is a toroidal ideal then the ideal $\tI=\pi_*(I)\cap\cO_{X_0}$ is toroidal and the corresponding ideals $J\subset\ocM_{X,x}$ and $\tJ\subset\ocM_{X_0,x_0}$ (see Lemma~\ref{toroidaldivisorslem}(i)) are related by $\tJ=J_0$.
\end{theorem}
\begin{proof}
All claims are local on $X_0$, hence we can assume that the action is local and $x_0$ is the closed point of $X_0$. For brevity of notation set $M=\ocM_{X,x}$. We start with part (ii). It suffices to prove the claim \'etale-locally, hence by use of Corollary~\ref{toroidalprop} we can assume that there exists a strongly equivariant toroidal chart $(X,U)\to(\bfA_M,\bfA_{M^\gp})$ over $\ZZ[1/N]$. Note that $M$ may be non-toric, but its torsion is killed by $N$. The morphism $$(\bfA_M,\bfA_{M^\gp})\to(\bfA_M,\bfA_{M^\gp})\sslash G=(\bfA_{M_0},\bfA_{(M_0)^\gp})$$ is easily seen to be toroidal, hence $\pi$ is toroidal too.

The remaining proof deals with parts (i), (iii) and (iv). By \cite[Corollary~5.4.5]{ATLuna} the quotient can be obtained in two stages: first divide by the action of $G_x$, which is strictly local at $x$, then divide by the free action of $G/G_x$. This reduces the proof to two separate cases: (a) the action is strictly local, (b) the action is free.

Case (a). In this case $x$ is the closed $G$-invariant point. By Corollary~\ref{toroidalprop} there exists a strongly equivariant toroidal chart $f\:(X,U)\to(Y,V)=(\bfA_{M},\bfA_{M^\gp})$, so we obtain a Cartesian square
$$
\xymatrix{
X\ar[r]^f\ar[d]_\alpha & Y\ar[d]^\beta\\
X_0\ar[r]^{f_0} &Y_0
}
$$
where $Y_0=Y\sslash G=\bfA_{M_0}$  and $M_0$ is the trivially $L$-graded part of $M$. Now let us check the claims (i), (iii), (iv).

(i) Set $U_0=U\sslash G$ and $V_0=\bfA_{M_0^\gp}$, then $(X_0,U_0)$ is a toroidal scheme and $f_0\:(X_0,U_0)\to(Y_0,V_0)$ is a toroidal chart by Theorem~\ref{flcilem}(iv). Furthermore, $f$ is strict hence Lemma~\ref{tautlem}(ii) implies that the action is taut or loose at $x$ if and only if the action on $(Y,V)$ is taut or loose, respectively. By Lemma~\ref{torictautquotient} this happens if and only if the quotient $(Y_0,V_0)=(Y,V)\sslash G$ is taut or  loose, and since $\alpha$ is the base change of $\beta$, the latter is equivalent to the quotient $(X_0,U_0)=(X,U)\sslash G$ being taut or  loose, respectively.

(iii) The action on $(Z,W)$ is toroidal by Lemma~\ref{toroidalactionlem}(ii), and $f\circ h$ is a strongly equivariant toroidal chart. Write  $h_0=h\sslash G$ and $(Z_0,W_0)=(Z,W)\sslash G$. As in the proof of (i), it follows from  Theorem~\ref{flcilem}(iv) that $(Z_0,W_0)$ is a toroidal scheme and the morphism $f_0\circ h_0\:(Z_0,W_0)\to (Y_0,V_0)$ is a toroidal chart, and hence $h_0$ is strict.

(iv) This can be checked on the formal completions at $x$ and $x_0$. By parts (i) and (ii) of Theorem~\ref{flcilem} we have representations $\hatcO_{X,x}=C_x\llbracket M,t_1\. t_r\rrbracket/(\theta)$ and $\hatcO_{X_0,x_0}=C_x\llbracket M_0,t_1\. t_r\rrbracket/(\theta)$, and then it is clear that $J\hatcO_{X,x}\cap\hatcO_{X_0,x_0}=J_0\hatcO_{X_0,x_0}$.

Case (b). In this case the action is flat-split free by \cite[Lemma~5.4.4]{ATLuna}, i.e. there exists a flat surjective morphism $T\to X_0$ such that $T\times_{X_0}X$ is $T$-isomorphic to $T\times G$. Hence $q\:X\to X_0$ is flat by flat descent and one can even take $T=X$ obtaining the Cartesian square of flat surjective morphisms
$$
\xymatrix{
X\times G\ar[r]^m\ar[d]_p & X\ar[d]^q\\
X\ar[r]^q & X_0.
}
$$
where $p:X\times G\to X$ is the projection and $m\:X\times G\to X$ is the action morphism.

Let $U_0\subseteq X_0$ denote the image of $U$. Since $U$ is $G$-equivariant, it is the preimage of $U_0$ by flat descent. In particular, once we prove that $(X_0,U_0)$ is toroidal, the quotient is automatically taut. Let $\cM_X$ and $\cM_{X_0}$ be the logarithmic structures associated with $U$ and $U_0$, respectively. We claim that the morphism $g\:(X,\cM_X)\to(X_0,\cM_{X_0})$ is strict. Indeed, we should prove that any toroidal Cartier divisor on $X$ is the pullback of a Cartier divisor on $X_0$, which is automatically toroidal since $U=q^{-1}(U_0)$. The logarithmic structure $\cM_X$ is $G$-equivariant by our assumption and hence any toroidal divisor $D\into X$ is $G$-equivariant, i.e. $m^{-1}(D)=p^{-1}(D)$ scheme-theoretically. Therefore, $D$ is the pullback of a closed subscheme $D_0\into X_0$. Since $D\into X$ is a regular closed immersion, the same is true for $D_0\into X_0$ by flat descent, and this means that $D_0$ is a Cartier divisor.

Now, let us deduce the theorem. (i) $(X_0,U_0)$ is toroidal because logarithmic regularity descends with respect to the flat strict morphism $g$ by \cite[Proposition~7.5.46(i)]{Gabber-Ramero}. In addition, the grading of $M$ is trivial because $L_x=0$ and $\ocM_{X_0,x_0}=M=M_0$ since $g$ is strict.

(iii) The action on $(Z,W)$ also has trivial stabilizers and hence is flat-split free. If $(Z_0,W_0)=(Z,W)\sslash G$ then by the same argument $(Z,W)\to(Z_0,W_0)$ is strict, and since the composition $g\circ h\:(Z,W)\to(X_0,U_0)$ is strict, $h_0$ is strict too.

(iv) The same argument as in the above paragraph shows that $I=\tI'\cO_X$ for a toroidal ideal $\tI'$ and then $\tI=\pi_*(\tI'\cO_X)\cap\cO_\tilX=\tI'$. So, $I$ is the pullback of $\tI$ under the strict morphism $g$ and we obtain that $\tJ=J=J_0$.
\end{proof}

\subsection{Making actions toroidal by increasing the toroidal structure}
\label{Sec:increase-todoidal}
Our next goal is to show that $G$-locally any $G$-simple action on a toroidal scheme can be made toroidal simply by increasing the toroidal divisor - this increased divisor will later be used as a tool to show our main results, namely that after a suitable blowing up, adding only the exceptional divisor suffices.

In the sequel, if $Z$ is a closed subscheme of $X=\Spec(A)$ given by an ideal $I\subset A$ then we call the $A/I$-module $N_{Z\into X}=I/I^2$ the {\em conormal module}\index{conormal module} to $Z$ in $X$.

\begin{proposition}\label{pretoroidalprop}
Assume that a diagonalizable group $G=\bfD_L$ acts simply and locally on a toroidal scheme $(X,U)$ with center $C$. Let $i\:O=\Spec(K)\into C$ be the embedding of the closed orbit, $G_O=\bfD_{L_O}$ the stabilizer, $M=\oM_O$  the corresponding monoid with grading $\chi\:M\to L_O$ (Lemma~\ref{simpletoroidallem2}). Then

(i) There exists a natural homogeneous morphism of $L$-graded $K$-modules $\phi_O:N_{O\into C}\to i^*\Omega_C$, and the module $\mathrm{Im}(\phi_O)$ is a free $K$-module of a finite rank $d$.

(ii) There exists an equivariant open subset $U'\subseteq U$ such that $(X,U')$ is a toroidal scheme whose center $C'$ is of codimension $d$ in $C$.

(iii) If $\cM'_X$ is the logarithmic structure associated with $(X,U')$ then up to an isomorphism the monoid $M'=\oM'_O$ and the induced grading $\chi'\:M'\to L_O$ depend only on $(X,U)$ and $G$. Moreover, $M'=M\oplus\NN^d$, and if $e_1\.e_d$ is the basis of $\NN^d$ then $\{\chi'(e_1)\.\chi'(e_d)\}$ is precisely the multiset of characters of the action of $G_O$ on $\mathrm{Im}(\phi_O)$.

(iv) The action on $(X,U')$ is always toroidal, and the action on $(X,U)$ is toroidal if and only if $\chi'(e_i)=0$ for any $1\le i\le d$.
\end{proposition}
\begin{proof}
(i) We have that $X=\Spec(\oA)$ for an $L$-local ring $(\oA,\om)$ and $C=\Spec(A)$ for an $L$-local ring $(A,m)$, which is a quotient of $(\oA,\om)$. The exact sequence of homologies of the transitivity triangle $i^*\LL_C\to\LL_O\to\LL_{O/C}$ of cotangent complexes ends with $$H_1(\LL_{O})\to N_{O\into C}{\stackrel{\phi_O}\longrightarrow} i^*\Omega_{C}\to\Omega_{O}\to 0.$$ Since the morphism $O\to C$ is $G$-equivariant, these $K$-modules are $L$-graded. In addition, the action on $O$ is trivial, hence the grading on $H_1(\LL_{O})$ is trivial. Note that $K=\oA/\om=A/m$ is an $L$-graded field (\cite[\S4.4.3]{ATLuna}), hence any $L$-graded $K$-module is free. The  rank of $\mathrm{Im}(\phi_O)$ is finite since $N_{O\into C}$ is finitely generated.

(ii) Choose an $L$-homogeneous $K$-basis $t_1\.t_d$ of $\mathrm{Im}({\phi_O})$, lift it through $\om\onto m\onto N_{O\into C}\onto\mathrm{Im}(\phi_O)$ to homogeneous elements $\ot_i\in\om$, and consider the $G$-equivariant divisors $D_i=V(\ot_i)$. Since $G$ acts transitively on $O$, for any point $x\in O$ the images of $t_1\.t_d$ in $m_x/m_x^2$ remain linearly independent. Therefore $\ot_1\.\ot_d$ form a subfamily of a family of regular parameters at any point of $O$. In particular, if $Z=\Spec(\oA/(\ot_1\.\ot_d))$ then $C'=C\times_XZ$ is regular along $O$. Since $G$ acts locally on $C'$ this implies that $C'$ is regular of codimension $d$ in $C$, and therefore $U'=U\setminus(\cup_{i=1}^dD_i)$ defines a toroidal scheme $(X,U')$ with center $C'$.

(iii) It follows from Theorem~\ref{decreaseth} that $M'=M\oplus\NN^d$. By Lemma~\ref{simpletoroidallem2}(ii), the basis elements $e_i\in\NN^d$ can be lifted to homogeneous elements $\ot_i\in\oA$. Then $C'$ is the intersection of $C$ with $d$ equivariant divisors $D_i:=V(\ot_i)$, and hence $D_1\.D_d$ have simple normal crossings along $O$. It follows that the images $t_i\in i^*\Omega_{C}$ are linearly independent, and by the dimension counting they form a homogeneous basis of $\mathrm{Im}(\phi_O)$. Hence, the characters of $e_i$ are as claimed.

(iv) Let $\ot_i$ be as in (iii). In the first claim we should check that $G_O$ acts trivially on $C'$, i.e. the induced $L_O$-grading of $A':=A/(\ot_1\.\ot_d)$ is trivial. The maximal $L$-homogeneous ideal of $A'$ is $m'=mA'$. Clearly, the $L_O$-grading on $K=A'/m'$ is trivial. Since $A'$ is noetherian and hence embeds into the $m'$-adic completion, it suffices to show that the $L_O$-grading of $m'/{m'}^2 = N_{O\into C'}$ is trivial. Since $N_{O\into C'}$ is the quotient of $N_{O\into C}$ by the span of the images of $\ot_1\.\ot_d$, it follows from the commutative square
$$
\xymatrix{
N_{O\into C}\ar[r]^{\phi_O}\ar[d] & i^*\Omega_C\ar[d]\\
N_{O\into C'}\ar[r]^{\phi'_O}& i'^*\Omega_{C'}
}
$$
that $\mathrm{Im}(\phi'_O)=0$. Hence the map $H_1(\LL_{O})\to N_{O\into C'}$ is surjective and we obtain that the grading of $N_{O\into C'}$ is trivial.

If the action on $(X,U)$ is toroidal then $G_O$ acts trivially on $C$, hence the $L_O$-grading on $N_{O\into C}$ is trivial, and one necessarily has that $\chi'(e_i)=0$. Conversely, assume that all $\chi'(e_i)$ vanish. Then $\ot_i\in\oA_0$ and since $A'=A/(\ot_1\.\ot_d)A$ is trivially $L_O$-graded, the same is true for $A$. Thus, $G_O$ acts trivially on $C$.
\end{proof}

\begin{remark}\label{pretoroidalrem}
(i) Let us say that an action of $G$ on a toroidal scheme $(X,D)$ is {\em pretoroidal} if for any point $x\in X$ there exists a larger divisor $D'\supset D$ such that  $(X,D')$ is still toroidal, $D'$ is equivariant and the action on $(X,D')$ is toroidal at $x$. In fact, Proposition~\ref{pretoroidalprop} proves that a pretoroidal action is nothing else but a $G$-simple action. Pretoroidal actions were introduced in \cite{Abramovich-deJong} for finite groups, and are related to the locally toric actions of $\GG_m$ in \cite{AKMW}, where one did not have a given toroidal structure $(X,D)$. Note that the definition is local and for a $G$-simple action it may happen that there is no larger global equivariant toroidal structure such that the action is toroidal everywhere.

(ii) Proposition~\ref{pretoroidalprop} and Theorem~\ref{toroidalth} imply that for any toroidal scheme $(X,D)$ with a $G$-simple action of $G$ the singularities of $X\sslash G$ are locally isomorphic to the singularities of toroidal schemes; such schemes were called {\em locally toric} in \cite{AKMW}. However, there is no canonical way to find such an isomorphism; e.g. the cone $C=\Spec(k[x,y,z]/(xy-z^2))$ with the empty divisor is not a toroidal scheme, and there are many different ways to choose a divisor that makes it into a toroidal scheme. This makes locally toric schemes difficult to work with. For example, one can locally resolve their singularities in a combinatorial way, but W{\l}odarczyk \cite[Theorem 8.3.2]{W-Factor} had to develop a theory of stratified toroidal varieties to resolve them canonically, and hence globally.
\end{remark}

\subsubsection{Potentially taut actions}
Assume that $M$ and $M'=M\oplus\NN^d$ are $L$-graded monoids and the grading on $\NN^d$ is trivial. Then the grading on $M$ is taut or loose if and only if the grading on $M'$ is taut or loose, respectively. Assume that $G$ acts simply and strictly locally on $(X,U)$. It follows from Proposition~\ref{pretoroidalprop} that there exists $U'\subseteq U$ such that the action on $(X,U')$ is toroidal and the monoid $M'=\ocM_{X,x}$ is uniquely defined up to a trivially graded direct summand of the form $\NN^l$. We say that the action on $(X,U)$ is {\em potentially taut or loose} if the action on $(X,U')$ is taut or loose, respectively. As we showed, this is independent of the choice of $U'$.\index{taut!potentially taut action}\index{loose!potentially loose action}

In general, we define a $G$-action to be {\em potentially taut or loose} at $x\in X$ if it is simple at $x$ and the action of $G_x$ on the $G_x$-equivariant localization at $x$ is potentially taut or loose, respectively.

\begin{remark}\label{potentialrem}
There is no relation between the action on $(X,U)$ being potentially taut or loose and the grading of $M=\ocM_{X,x}$ being taut or loose, respectively. The latter information seems to be not so relevant to the properties of the action.
\end{remark}

\subsection{Decreasing the toroidal structure}
We will also need to know when a given toroidal structure can be decreased without loosing good properties of the action. In the following result we use divisors rather than open sets in the notation of toroidal schemes.

\begin{proposition}\label{decreaseprop}
Let $(X,D)$ be a toroidal scheme provided with a toroidal action of a diagonalizable group $G=\bfD_L$. Assume that the action on $X$ is local with closed orbit $O$, and let $M=\oM_O$ be the module defined in Lemma \ref{simpletoroidallem2}(i).

(i) If $Z$ is an irreducible component of $D$ and $D'=D-Z$, then the following conditions are equivalent: (a) $(X,D')$ is a toroidal scheme and the action of $G$ on $(X,D')$ is toroidal, (b) $Z$ is Cartier and the corresponding element $e\in M$ is of degree zero with respect to the grading $\chi\:M\to L_O$ induced by the action.

(ii) Let $E$ be obtained from $D$ by removing all irreducible Cartier subdivisors $Z\subseteq D$ such that the corresponding character in $L_O$ is trivial. Then $E\subseteq D$ is the minimal subdivisor such that the action on $(X,E)$ is toroidal.
\end{proposition}
\begin{proof}
It suffices to prove (i) as (ii) follows by induction. Assume that (a) holds. Then $Z$ is Cartier at $x$ by Theorem~\ref{decreaseth} and $\chi(e)=0$ by Proposition~\ref{pretoroidalprop}(iii). Conversely, assume that (b) holds. Then $(X,D')$ is a toroidal scheme by Theorem~\ref{decreaseth} and it remains to use Proposition~\ref{pretoroidalprop}(iii) again.
\end{proof}

\subsection{Combinatorial charts}\label{combsec}
Our next aim is to extend the theory of equivariant toroidal charts to arbitrary simple actions at the cost of considering charts $\bfA_P$ with a smaller toroidal divisor and a non-toroidal action. Moreover, we will see that the charts depend only on the following combinatorial data: graded monoids and non-trivial characters through which $G$ acts on the cotangent spaces of logarithmic strata.

\subsubsection{The models}\label{modelsec}
Assume given the following data: a finitely generated abelian group $L$, an $L$-graded toric monoid $M$, and a function $\sigma\:L\setminus\{0\}\to\NN$ with a finite support. Consider the monoid $\NN^\sigma=\oplus_{l\in L\setminus\{0\}}\NN^{\sigma(l)}$, where each $\NN^{\sigma(l)}$ is graded by $l$. Then $P=M\oplus\NN^\sigma$ is an $L$-graded toric monoid and the toric scheme $X=\bfA_P$ acquires an action of $G=\bfD_L$. The action is not toroidal whenever $\sigma\neq 0$.

\subsubsection{Signature}\label{sec:signature}
Assume that $G=\bfD_L$ acts on a finite-dimensional vector space $V$, i.e. $V$ is provided with an $L$-grading $V=\oplus_{l\in L}V_l$. We record this representation combinatorially as follows.

By the {\em signature}\index{signature} $\sigma'_V$ of $V$ we mean the multiset of characters through which $G$ acts. Equivalently, one can view the signature as a function $\sigma'_V\:L\to\NN$ that sends $l$ to $\dim(V_l)$. Clearly the data of $\sigma'_V$ is equivalent to the data of the representation $V$.

The multiset of all {\em non-trivial} characters will be called the {\em reduced signature}\index{signature!reduced} and denoted $\sigma_V\:L\setminus\{0\}\to\NN$. One may think of the reduced signature as the equivalence class of the representation $V$, where $V \sim V'$ if they differ by trivial characters.

Any homomorphism $\phi\:L\to\tL$ induces a new grading on $V$ and $\sigma'_{V,\tL}=\phi(\sigma'_{V,L})$. In addition, $\sigma_{V,\tL}$ is obtained from $\phi(\sigma_{V,L})$ by removing zeros, so we say that $\sigma_{V,\tL}$ is the {\em reduced image} of $\sigma_{V,L}$ and write $\sigma_{V,\tL}=\phi(\sigma_{V,L})_\red$.

If $G$ acts on a noetherian scheme $X$ then by the signature $\sigma'_x$ of the action at a point $x\in X$ we mean the signature of the cotangent space $V_x=m_x/m_x^2$ acted on by the stabilizer $G_x$. Note that $G_x=\bfD_{L_x}$ for a quotient $L_x$ of $L$ and $\sigma'_x$ is an $\NN$-valued function on $L_x$. The reduced signature $\sigma_x\:L_x\setminus\{0\}\to\NN$ is defined similarly.

\subsubsection{Functoriality of signature}
The reduced signature is compatible with strongly equivariant morphisms:

\begin{lemma}\label{Lem:functorsign}
Let $G=\bfD_L$ be a diagonalizable group, $f\:Y\to X$ a strongly $G$-equivariant morphism, $y\in Y$ a point and $x=f(y)$. Then $L_y=L_x$ and $\sigma_y=\sigma_x$.
\end{lemma}
\begin{proof}
Since $f$ is strongly equivariant, $G_y=G_x$. Applying the functor  $\cdot\otimes^{\rm L} k(y)$ to the transitivity triangle
$\LL_{Y/X}[-1]\to\LL_X\to\LL_Y\to\LL_{Y/X}$  one obtains the exact triangle $$\LL_{Y/X}\otimes^{\rm L}k(y)[-1]\to\LL_X\otimes^{\rm L} k(y)\to\LL_Y\otimes^{\rm L}k(y)\to\LL_{Y/X}\otimes^{\rm L}k(y).$$ The associated exact sequence of homologies ends with
$$H_1(\LL_{Y/X}\otimes^{\rm L}k(y))\to (m_x/m_x^2)\otimes_{k(x)}k(y)\to m_y/m_y^2\to H_0(\LL_{Y/X}\otimes^{\rm L}k(y)),$$ where $x=f(y)$. Since $f$ is strongly equivariant, the vector spaces $H_i(\LL_{Y/X}\otimes^{\rm L}k(y))$ with $i=0,1$ are trivially graded by \cite[Theorem~1.3.1(i)]{ATLuna}. Therefore, the multiplicities of a nontrivial character $l\in L_x=L_y$ of $G_y=G_x$ in the action on $m_y/m_y^2$ and in the action on $m_x/m_x^2$ coincide. The claim follows.
\end{proof}

\subsubsection{The signature is locally constant on fixed loci}
Recall that the action of $G$ on a regular scheme $X$ induces inertia stratification of $X$ by regular subschemes $X(H)$ along which the stabilizer is constant, see \cite[\S5.1.14]{ATLuna}.

\begin{lemma}\label{signlem}
Assume that $G=\bfD_L$ acts on a regular scheme $X$. Then for any subgroup $H\subseteq G$ the reduced signature of the action of $G$ is locally constant along the stratum $X(H)$.
\end{lemma}
\begin{proof}
Fix a subgroup $H=\bfD_{L'}$ of $G$, for some quotient $L'$ of $L$. We should prove that if $x\in X(H)$ and $y$ is the generic point of the component of $X(H)$ containing $x$ then $\sigma_x=\sigma_y$. The question is local at $x$ and only depends on the action of $H$, so shrinking $X$ we can assume that $X=\Spec(A)$ for a strictly local $L'$-graded ring. By \cite[Lemma~4.4.15]{ATLuna}, $m_x$ is generated by $d=\dim(X)$ homogeneous elements $t_1\.t_d$. We can assume that $t_i$ is trivially graded if and only if $1\le i\le s$, and then $\sigma_x$ is the multiset of gradings of $t_{s+1}\.t_d$. Recall that $X(H)$ is regular by \cite[Proposition~5.1.16]{ATLuna}, and clearly $H$ acts trivially on the cotangent space of $X(H)$ at $x$. Since $X(H)\subseteq V(t_{s+1}\.t_d)$ and the images of $t_1\.t_s$ span the maximal trivially graded subspace of $m_x/m_x^2$, we actually have that $V(t_{s+1}\.t_d)=X(H)$. Hence the maximal ideal of $y$ is generated by $t_{s+1}\.t_d$, and the images of $t_{s+1}\.t_d$ form a basis of the cotangent space at $y$. In particular, $\sigma_y=\sigma_x$.
\end{proof}

\begin{corollary}\label{signcor}
Assume that $G=\bfD_L$ acts on a regular scheme $X$. Let $x,y\in X$ be two points such that the closure of the orbit of $y$ contains $x$ and let $\phi\:L_x\to L_y$ be the map associated to $G_y\hookrightarrow G_x$. Then $\sigma_y=\phi(\sigma_x)_\red$.
\end{corollary}
\begin{proof}
Consider the action of $H=G_y$ on $X$. Then $H_y=H_x=\bfD_{L_y}$ and we denote by $\sigma^H_x$ and $\sigma^H_y$ the corresponding characteristic subsets of $L_y$. Clearly, $\sigma^H_y=\sigma_y$ and  $\sigma^H_x=\phi(\sigma_x)_\red$. It remains to note that $\sigma^H_x=\sigma^H_y$ by Lemma~\ref{signlem}.
\end{proof}

\subsubsection{Local combinatorial data}\label{localcombsec}
Assume now that $G=\bfD_L$ acts simply and in a relatively affine manner on a toroidal scheme $(X,U)$. To any point $x\in X$ one can associate the following combinatorial datum: the stabilizer $G_x=\bfD_{L_x}$, the $L_x$-graded monoid $\oM_x=\ocM_{X,x}$, and the reduced signature $\sigma_x\:L_x\setminus\{0\}\to\NN$ of the action of $G_x$ on the logarithmic stratum through $x$. The following result shows that these data depend only on the orbit $O=O_x$ of $x$, so we will also use the notation $\sigma_O$, in addition to $G_O$, $L_O$ and $\oM_O$ introduced earlier.

\begin{proposition}\label{combdataprop}
Assume that a diagonalizable group $G=\bfD_L$ acts on a toroidal scheme $X$.

(i) The triple $(\oM_x,L_x,\sigma_x)$ is locally constant along the logarithmic-inertia strata $X(H)\cap X(r)$. In particular, the triples are locally constant along orbits of $G$ and there are finitely many isomorphism classes of the triples $(\oM_x,L_x,\sigma_x)$.

(ii) The combinatorial data are compatible with strict strongly equivariant morphism of toroidal schemes: if $f\:Y\to X$ is such a morphism, $y\in Y$ and $x=f(y)$ then $(\oM_x,L_x,\sigma_x)=(\oM_y,L_y,\sigma_y)$.
\end{proposition}
\begin{proof}
(i) The monoid $\oM_x$ is locally constant along the logarithmic strata $X(r)$ by Lemma~\ref{centerlemma} and the reduced signature is locally constant along the inertia strata $X(H)$ by Lemma~\ref{signlem}. Since $X$ is noetherian the finiteness claim follows.

(ii) By the definition of strictness, $\oM_x=\oM_y$. Since $\sigma_x$ is determined by $\oM_x$ and the reduced signature of the action of $G$ on the whole $X$, and the same claim holds for $y$, the assertion of (ii) follows from Lemma~\ref{Lem:functorsign}.
\end{proof}

\subsubsection{Construction of charts}
Now we will construct local charts for simple $G$-actions. We assume that the action is simple for simplicity of the exposition -  the assumption can be removed  at the cost of working with charts $\bfA_P$ on which $G$ acts through a grading on $P$ and a (non-trivial) action on the graded monoid $P$.

\begin{theorem}\label{combchart}
Assume that a toroidal scheme $(X,D)$ is provided with a local $G$-simple action of a diagonalizable group $G=\bfD_L$ and the torsion $K_O^\tor$ of $K_O=\Ker(L\to L_O)$ is of order $N$ invertible on $X$. Set $P=\oM_O\oplus K_O\oplus\NN^{\sigma_O}$, where $O$ is the orbit. Then there exists a strongly equivariant strict morphism $$(X,D)\times\bfD_{K_O^\tor}\to(\bfA_P,E),$$ where $E=\bfA_P\setminus\bfA_{\oM_O^\gp\oplus K_O\oplus\NN^{\sigma_O}}$ and the charts are over $\ZZ[1/N]$.
\end{theorem}
\begin{proof}
By Proposition~\ref{pretoroidalprop} there exists a $G$-equivariant divisor $D'$ containing $D$ and such that $(X,D')$ is a toroidal scheme, the action of $G$ on $(X,D')$ is toroidal, and the $G$-graded monoids $\oM'_O=\ocM_{(X,D'),O}$ and $\oM_O$ are related by $\oM'_O=\oM_O\oplus\NN^{\sigma'_V}$, where $V$ is the image of the map $\phi_O:N_{O\into C}\to i^*\Omega_C$ and $C$ is the center of $(X,U)$. In addition, the components $D_1\.D_d$ of $D'$ not contained in $D$ correspond to the generators $e_1\.e_d$ of $\NN^{\sigma'_V}$. As in the proof of Proposition~\ref{pretoroidalprop}, $D_i=V(\ot_i)$, where $\ot_i$ is homogeneous of weight $\chi'(e_i)$. Without restriction of generality, there exists $n$ such that $1\le n\le d$ and $\chi'(e_i)=0$ precisely when $i>n$.

By Proposition~\ref{decreaseprop}, if $D''$ is obtained by decreasing $D'$ by removing $D_{n+1}\.D_d$, then both $(X,D'')$ and the action stay toroidal. After such decreasing we have that $\oM''_O=\oM_O\oplus\NN^{\sigma_V}$. We claim that $\sigma_V=\sigma_O$ and hence $\oM''_O=\oP$. Indeed, the center $C''$ of $(C,D)$ is the intersection of $C$ with $D_1\.D_n$. Since $G$ acts trivially on $C''$, we have that $\sigma_O=\{\chi'(e_1)\.\chi'(e_n)\}\  = \ \sigma_V$.

By Proposition~\ref{simpleprop}, there exists an equivariant central toroidal chart $$f\:(X,D'')\times\bfD_{K_O^\tor}\to(\bfA_P,\bfA_P\setminus\bfA_{P^\gp})$$ which is fixed-point reflecting along the preimage of $O$, and this chart is strongly equivariant by Corollary~\ref{toroidalprop}(ii). Finally, $E$ is obtained from $\bfA_P\setminus\bfA_{P^\gp}$ by omitting the components corresponding to the generators of $\NN^{\sigma_O}$, hence $D\times\bfD_{K_O^\tor}=f^{-1}(E)$ and $f\:(X,D)\times\bfD_{K_O^\tor}\to(\bfA_P,E)$ is a required combinatorial chart.
\end{proof}

\begin{corollary}\label{toroidalsigmacor}
Assume that a toroidal scheme $(X,D)$ is provided with a $G$-simple action of a diagonalizable group $G=\bfD_L$. Let $x,y\in X$ be points such that $x$ lies in the closure of the orbit of $y$ and let $\phi\:L_x\to L_y$ be the associated map. Then $\sigma_y=\phi(\sigma_x)_\red$.
\end{corollary}
\begin{proof}
Replacing $G$ by $G_x$ and $X$ by the equivariant localization at $x$ we can assume that the action is strictly local with closed orbit $x$. Set $M=\oM_x$ and $P=M\oplus\NN^{\sigma_x}$, then by Theorem~\ref{combchart} there exists a strongly equivariant morphism $f\:(X,D)\to(\bfA_P,E)$ with $E=\bfA_P\setminus\bfA_{M^\gp\oplus\NN^{\sigma_x}}$. Since $\sigma_x=\sigma_{f(x)}$ and $\sigma_y=\sigma_{f(y)}$ by Proposition~\ref{combdataprop}(ii), it suffices to prove the assertion for the target of $f$. Thus, we should prove that for any point $z$ of the toroidal scheme $(\bfA_P,E)$ the reduced image of $\sigma_x$ under $\psi\:L_x\to L_z$ is $\sigma_z$.

Set $N=\oM_z$ and let $K=\Ker(M^\gp\to N^\gp)$. Then the logarithmic stratum of $z$ is isomorphic to $S=\bfA_{K\oplus\NN^{\sigma_x}}$. The action of $G_x$ on $S$ corresponds to the composite homomorphism $K\oplus\NN^{\sigma_x}\to M^\gp\to L_x$. Note that $M^\gp\to L_x \to L_z$ factors through $N^\gp$, hence the image of $K$ under $L_x\to L_z$ vanishes. Thus, the map $S\to\bfA_{\NN^{\sigma_x}}$ is strongly $G_z$-equivariant and in view of Lemma~\ref{Lem:functorsign} it suffices to prove the analogous claim for $\bfA_{\NN^{\sigma_x}}$. In the latter case, the action is strictly local and the closed orbit $O$ (the origin) satisfies $\sigma_O=\sigma_x$. It remains to use Corollary~\ref{signcor}.
\end{proof}

\section{Torification}\label{Sec:torification}
Although a $G$-simple action on a toroidal scheme $(X,U)$ can be locally ``improved" to a toroidal action by increasing the toroidal divisor, this procedure is neither global nor canonical. Some drawbacks of this were discussed in Remark~\ref{pretoroidalrem}. The goal of this section is to establish a better way, called {\em torification}, to make actions toroidal. Torification, introduced in \cite{Abramovich-deJong} and developed in \cite{AKMW}, will be achieved by a functorial blowing up of $X$ and will only increase the toroidal divisor by adding the exceptional divisor. So, the exceptional divisor plays a role analogous to its role in  desingularization theory -- providing new parameters in a canonical way. The na\"ive procedure of adding divisors is still used to verify, by local computations, that the exceptional divisors provide the necessary parameters.

We will mostly follow the methods of \cite[\S3]{AKMW}, the main modification being the use of strongly equivariant charts instead of strongly \'etale charts.

\subsection{Making an action $G$-simple}\label{makesimple}
The restriction on an action to be $G$-simple is  very mild; one can always transform an action into a $G$-simple action using a \emph{barycentric subdivision},  a simple combinatorial construction recalled below.

\subsubsection{Kato Fans}\index{Kato fan}
To any toroidal scheme $(X,U)$ Kato associates in \cite[Section 10.1]{Kato-toric} a combinatorial structure we call a {\em Kato  fan} $F=F(X,U)$, to distinguish it from the fans of toric geometry. It is defined as follows: points of $F$ are the maximal points of the logarithmic stratification of $(X,\cM_X)$ and the structure sheaf of monoids $\cM_F$ is the pullback of $\ocM_X$. Since connected components of the logarithmic strata of $(X,\cM_X)$ are irreducible, one obtains a natural retraction map $c\:X\to F$, that can be viewed as a ``combinatorial chart" of $(X,U)$. The polyhedral cone complex with integral structure used in  \cite{KKMS} for toroidal varieties can be recovered as $F(\RR_{\geq 0})$, see \cite{Ulirsch}. Any subdivision $F'\to F$ can be pulled back to $X$: Kato defines a ``base change" modification $f\:X'=X\times_FF'\to X$ such that $(X',f^{-1}(U))$ is a toroidal scheme with Kato fan $F'$. In particular, a sequence of subdivisions $F_n\to\dots\to F_1\to F_0=F$ induces a sequence of toroidal modifications $(X_n,U_n)\to\dots\to(X,U)$. Moreover, if the subdivisions are given by order functions (see \cite[Section 1]{Abramovich-Wang}) then the modifications are toroidal blowings up.

\subsubsection{Barycentric subdivision}
For any Kato fan $F$ its {\em barycentric subdivision} is defined as a composition of subdivisions $B(F)=F_n\to\dots\to F_0=F$, where $F_1\to F$ performs simultaneous subdivisions of all cones of maximal dimension at their barycenters, $F_2\to F_1$ subdivides the preimages of the original cones of the next dimension at their barycenters, and so on.\index{barycentric subdivision}

For any toroidal scheme $(X,U)$ the barycentric subdivision $B(F)\to F$ of its Kato fan induces via a ``base change" a sequence of toroidal blowings up $(X',U'):=(X_n,U_n)\to\dots\to(X,U)$ and we say that $(X',U')\to (X,U)$ is the {\em barycentric modification}.

\begin{remark}
(i) In fact, one can even realize $(X',U')\to (X,U)$ as a single toroidal blowing up along a toroidal ideal $J$, see \cite[Theorem 5.6]{Niziol} for details.

(ii) The barycentric subdivision provides a standard procedure to subdivide any Kato fan to a simplicial one. There also is a much more delicate canonical desingularization procedure that subdivides $F$ into a non-singular one, i.e. a Kato fan whose cones are of the form $\NN^n$. The latter can be used for a canonical desingularization of toroidal schemes, see \cite{KKMS}, \cite[Theorem~3.4.9]{Illusie-Temkin}.
\end{remark}

\begin{proposition}[\protect{\cite[Proposition 2.3]{Abramovich-Wang}}]\label{makingsimpleprop}
Assume that a diagonalizable group $G$ acts on a toroidal scheme $(X,U)$. Then the barycentric modification $(X',U')\to(X,U)$ is $G$-equivariant and the action on $(X',U')$ is $G$-simple.
\end{proposition}
\begin{proof}
An action of a group $G$  on $(X,U)$ induces by functoriality an action of $G$ on $F$. Since the set of barycenters of cones of a given dimension is stable under $G$, the action lifts to $B(F)$ and by pullback the action on $(X,U)$ lifts to $(X',U')$. An irreducible component of $D'=X'\setminus U'$ is {\em old} if it is the strict transform of a component of $D=X\setminus U$, otherwise it is said to be \emph{new}. All new components of $D'$ are equivariant and all old components are disjoint. Thus the criteria of Lemma~\ref{simpletoroidallem} apply and the action on $(X',U')$ is $G$-simple.
\end{proof}

\subsection{Torific ideals}

\subsubsection{The definition}
Assume that a diagonalizable group $G=\bfD_L$ acts on an affine scheme $X=\Spec(A)$. For any element $l\in L$ let $I_l^A$ denote the ideal $A_lA$ generated by all $l$-homogeneous elements. The formation of such ideals is compatible with localization by elements of $A_0$ hence the definition globalizes to any scheme $X$ with a relatively affine $G$-action. We call the corresponding ideal the {\em $l$-torific}\index{torific!ideal} ideal and denote it $I^X_l\subseteq\cO_X$.
For any finite multiset $S$ in $L$ we define the {\em $S$-torific ideal} $I^X_S=\prod_{l\in S}I^X_l$.

\begin{remark}
The construction of torific ideals is local on $X\sslash G$ but not on $X$. One can easily give examples where $I^X_S$ is not compatible with restriction to equivariant open subschemes which are not strongly equivariant.
\end{remark}

\subsubsection{Functoriality}
Compatibility of torific ideals with restriction to strongly equivariant open subschemes is a particular case of the following result.

\begin{lemma}\label{Lem:functorideal0}
Torific ideals are compatible with strongly $G$-equivariant morphisms $f\:Y\to X$ in the sense that $I_S^Y=I_S^X\cO_Y$ for any finite multiset $S$ in $L$.
\end{lemma}
\begin{proof}
The claim is local on the quotients, hence we can assume that the quotients are affine, and so $X=\Spec(A)$ and $Y=\Spec(B)$ are affine. The morphism $f$ is strongly equivariant precisely when $A\otimes_{A_0}B_0\to B$ is an isomorphism. Then $I^B_l=I^A_lB$ for any $l\in L$, and hence $I^B_S=I^A_SB$.
\end{proof}

\subsubsection{Localization}
Next we study compatibility of torific ideals and $G$-localizations.

\begin{lemma}\label{localtorificlem}
Assume that a scheme $X$ is provided with a relatively affine action of $G=\bfD_L$. Let $O\subset X$ be a special orbit with stabilizer $G_O=\bfD_{L_O}$, let $S$ be a multiset in $L$, and let $S_O$ be the image of $S$ under the homomorphism $\phi\:L\onto L_O$. Then the restriction of the torific ideals $I_S^{X,G}$ onto the equivariant localization $X_O$ coincides with $I_{S_O}^{X_O,G_O}$.
\end{lemma}
\begin{proof}
The localization morphism $X_O\into X$ is strongly equivariant by definition, see \cite[Section 5.1.11]{ATLuna}.  Hence $I_S^{X,G}|_{X_O}=I_S^{X_O,G}$ and it suffices to prove the result when $X=X_O$. In this case write $X=\Spec(A)$ for an $L$-local ring $A$. By Lemma \ref{Lem:unit-criterion}, $A_n$ contains a unit whenever $n\in \Ker(\phi)$. Thus, $I_l^A=I_{l+n}^A$ for any $l\in L, n\in \Ker(\phi)$, and hence $I^{A,L}_S=I^{A,L_O}_{S_O}$.
\end{proof}

\subsubsection{Coherent families of characters}
Assume that $G=\bfD_L$ acts on $X$. By $L_X$ we denote the multiset $\coprod_{x\in X\sslash G}L_x$ of characters of stabilizers of special orbits of $X$. A multiset of characters $S$ in $L_X$ will be called {\em coherent} if for any $x\in X\sslash G$ the multiset $S_x$ of the elements of $S$ lying in $L_x$ is finite, $0$ is not in $S_x$, and for any $y\in X\sslash G$ with a specialization $x\in X\sslash G$ the multiset $S_y$ is the reduced image of $S_x$ under the map $L_x\onto L_y$.

Any finite multiset $S$ in $L\setminus\{0\}$ defines a coherent family of characters $S_X$ in $L_X$. We call $S_X$ a {\em constant} coherent family. For simplicity we will often write $S$ instead of $S_X$. A coherent family is called {\em locally constant} if there exists a strongly equivariant open covering $X=\cup_i X_i$ such that each $S_{X_i}$ is constant.

\subsubsection{General torific ideals}
Let $S$ be a locally constant coherent multiset of characters in $L_X$. It follows from Lemma~\ref{localtorificlem} that there exists a coherent ideal $I_S^X\subseteq\cO_X$ such that for any point $x\in X\sslash G$ with the corresponding special orbit $O$ the restriction of $I_S^X$ onto the equivariant localization $X_O$ coincides with $I_{S_x}^{X_O}$. We call $I_S^X$ the {\em torific ideal} associated with $S$. An ideal is a \emph{torific ideal} if it is the torific ideal associated to some locally constant coherent multiset of characters $S$.

\begin{remark}
(i) In the definition of coherent multisets we insisted that $0$ is not in $S_x$. This does not affect the ideal $I_S^X$ since $\cO_0=\cO_X$.

(ii) By the very definition, $I$ is a torific ideal if and only if there exist a strongly equivariant covering $X=\cup_iX_i$ and multisets $S_i$ in $L$ such that ${I|_{X_i}}=I^{S_i}_{X_i}$. The advantage of our definition in terms of multisets $S$ in $L_X$ is its canonicity.

(iii) In \cite{AKMW} one only considers torific ideals associated with constant sets of characters, and uses large enough sets $S$ for torification. We will achieve torification by a functorial blowing up, and for this one has to use minimal multisets of characters. This requires to consider non-constant families and extend the notion of torific ideals accordingly.

(iv) Our definition does not exclude the possibility that a torific ideal vanish on an irreducible component of $X$. We will need to avoid this possibility in order to assure that the blowing up is a birational transformation.
\end{remark}

\subsubsection{Extended functoriality}
Any strongly equivariant morphism $f\:Y\to X$ takes special orbits to special orbits and respects stabilizers. In particular, if $f\sslash G$ takes $y\in Y\sslash G$ to $x\in X\sslash G$ then $L_x=L_y$. For a multiset $S$ in $L_X$ we define the pullback $f^*(S)$ in $L_Y$ by $f^*(S)=\coprod_{y\in Y\sslash G} S_{f(y)}$. Clearly, if $S$ is locally constant and coherent then so is $f^*(S)$. Thus Lemma~\ref{Lem:functorideal0}  implies the following generalization.

\begin{lemma}\label{Lem:functorideal}
Torific ideals are compatible with strongly $G$-equivariant morphisms $f\:Y\to X$ in the sense that $I_{f^*(S)}^Y=I_S^X\cO_Y$ for any locally constant coherent multiset $S$ in $L_X$.
\end{lemma}

\subsubsection{Signature}
Assume now that $G$ acts on a toroidal scheme $X$. For any $x\in X$ we defined in \S\ref{localcombsec} the signature $\sigma_x$ as a multiset in $L_x\setminus\{0\}$. Recall that $\sigma_x$ only depends on the orbit $O_x$ of $x$. Indeed, $O_x$ lies in the logarithmic stratum through $x$, hence this follows from Proposition~\ref{combdataprop}. The signature $\sigma_X=\coprod_{x\in X\sslash G}\sigma_x$ of $X$ is the multiset of all characters through which $G$ acts on the logarithmic strata locally at the special orbits. Also, we call a multiset $S$ in $L_X$ {\em balanced}\index{balanced multiset of characters} if each $S_x$ is finite and satisfies $\sum_{l\in S_x}l=0$. By $\sigma^0_X$ we denote the natural balanced multiset containing $\sigma_X$ obtained by adding to each non-balanced $\sigma_x$ the element $-\sum_{l\in S_x}l$.

\begin{lemma}\label{Sigmalem}
Assume that $G=\bfD_L$ acts in a relatively affine manner on toroidal schemes $X$ and $Y$, then

(i) The multisets of characters $\sigma_X$ and $\sigma_X^0$ are coherent and locally constant.

(ii) If $f\:Y\to X$ is a strict strongly equivariant morphism then $f^*(\sigma_X)=\sigma_Y$ and $f^*(\sigma^0_X)=\sigma^0_Y$.
\end{lemma}
\begin{proof}
Claim (ii) follows from Proposition~\ref{combdataprop}(ii). It suffices to prove claim (i) for $\sigma_X$ since the case of $\sigma^0_X$ follows. The claim can be checked locally at a closed orbit $O$. Recall that by Lemma~\ref{signlem} there exists a multiset $\sigma_O$ in $L_O$ such that $\sigma_x=\sigma_O$ for any $x\in O$. Passing to a strongly equivariant neighborhood of $O$ we can achieve that $X=\Spec(A)$ is affine and the stabilizers of points $x\in X$ are subgroups of $G_O$. Furthermore, by a further shrinking we can remove from $X$ all connected components of the inertia strata $X(H)$ (for $H\subseteq G_O$) whose closure does not intersect $O$. Then, by Corollary~\ref{toroidalsigmacor}, we necessarily have that $\sigma_X$ is the constant coherent multiset associated with $\sigma_O$.
\end{proof}

\subsection{Torific blowings up}
Next, we will study normalized blowings up of torific ideals, but we should discuss equivariant normalization first.

\subsubsection{$G$-normalization}
In general, an action of $G=\bfD_L$ on $X$ does not have to extend to the reduction or normalization of $X$, as can be seen in examples with $G=\mu_p$ acting on varieties over $\FF_p$. However, one can define  natural equivariant versions of  normalization or reduction. We discuss both constructions, but only the case of normalization will be used later.

We start with the affine case, so assume that $A$ is an $L$-graded ring. By the {\em $L$-reduction} $Red_L(A)$ (resp. {\em degree-0 reduction} $Red_0(A)$) we mean the quotient of $A$ by the ideal generated by all homogeneous nilpotents (resp. nilpotents of degree 0); note that it does not have to be a reduced algebra. If $A$ is an $L$-domain, i.e. it contains no homogeneous zero-divisors, then $A$ embeds into the {\em $L$-graded fraction field} $Frac_L(A)$ obtained by inverting all homogeneous elements (again note that an $L$-graded field is not necessarily a field, \cite[4.4.3]{ATLuna}). We define the {\em $L$-normalization} $Nor_L(A)$ (resp. {\em degree-0 normalization} $Nor_0(A)$) to be the subalgebra of $Frac_L(A)$ generated by all elements of the form $\frac ab$, where $a,b$ are homogeneous elements of $A$ (resp. elements of $A_0$) and $\frac ab$ is integral over $A$ (resp. $A_0$).\index{normaization!degree-0 normalization}\index{reduction!degree-0 reduction}

\begin{remark}\label{Gnorrem}
(i) The degree-0 normalization is the base change of the normalization of $A_0$, i.e. $Nor_0(A)=A\otimes_{A_0}Nor(A_0)$. In general, $Nor_L(A)$ is a partial normalization of $A$ that dominates $Nor_0(A)$, i.e. we have a sequence $A\to Nor_0(A)\to Nor_L(A)\to Nor(A)$. Similar facts hold for $L$-graded reductions.

(ii) It is easy to see that the four constructions we discussed are compatible with localization at a degree-0 multiplicative set $T\subseteq A_0$. For example, $Nor_L(A_T)=(Nor_L(A))_T$, etc.
\end{remark}

In general, given a scheme $X$ provided with a relatively affine action of $G=\bfD_L$ we define its $G$-normalization, $G$-reduction, $0$-normalization and $0$-reduction as follows. Choose a strongly equivariant affine covering $X_i=\Spec(A_i)$ and glue the $G$-equivariant $X$-scheme $Nor_G(X)$ (resp. $Nor_0(X)$, resp. $Red_G(X)$, resp. $Red_0(X)$) from the $G$-equivariant $X$-schemes $\Spec(Nor_L(A_i))$ (resp. $\Spec(Nor_0(A_i))$, resp. $\Spec(Red_L(A_i))$, resp. $\Spec(Red_0(A_i))$). Note that the gluing is possible due to Remark~\ref{Gnorrem}(ii).\index{normaization!$G$-normalization}\index{reduction!$G$-reduction}

\subsubsection{Torific blowings up}\label{Sec:torific-blowups}
Assume $X$ is reduced and $S$ is a locally constant coherent multiset of characters in $L_X$. The $S$-torific ideal $I_S^X$ is, by definition, $G$-equivariant. Therefore, the action of $G$ naturally lifts to the blowing up $Bl_{I^X_S}(X)\to X$ and by the {\em $S$-torific blowing up}\index{torific!blowing up} $b_{X,S}\:X_S\to X$ we mean the $G$-normalized blowing up $Nor_GBl_{I^X_S}(X)\to X$ along $I_S^X$. If needed, we will also indicate $G$ by writing $b_{X,S,G}$. Sometimes we will also need the {\em degree-0 torific blowing up} $Nor_0Bl_{I^X_S}(X)\to X$ that will be denoted $b^0_{X,S}\:X_S^0\to X$.\index{torific!degree-0 torific blowing up}

\begin{remark}\label{torblowrem}
(i) In our applications to toroidal schemes, torific blowings up will have normal sources, i.e. they will coincide with the corresponding normalized blowings up, see Corollary~\ref{torificcor}(i).

(ii) In general, the normalization $X'\to X$ of a noetherian scheme $X$ does not have to be of finite type, and $X'$ can even be non-noetherian. However, we will show in Section~\ref{Sec:Torification-bu} that torific blowings up of toroidal schemes are of finite type.

(iii) Note that a product of ideals is invertible if and only if all factors are invertible, and hence $Bl_{I_1^{n_1}\dots I_m^{n_m}}(X)=Bl_{I_1\dots I_m}(X)$ for any choice of  $n_1\.n_m>0$. It follows that the $X$-scheme $X_S$ depends only on the support $|S|\subset L_X$ of $S$. However, without considering multiplicities, sets of characters $S\subset L_X$ have worse functoriality properties, for example, an analog of Lemma~\ref{localtorificlem} fails. This is the reason we choose to work with multisets.
\end{remark}

\subsubsection{Charts of torific blowings up}
Note that $b_{X,S}$ can be covered by charts of the form $Y'\to X'$ where $X'=\Spec(A)$ is an open affine strongly $G$-equivariant subscheme of $X$ and $Y'$ is a chart of the $G$-normalized blowing up along $I^A_S$, where $S$ is a constant coherent family. In particular, $I^A_S$ is generated by elements $t_1\.t_n\in A_d$ where $d=\sum_{l\in S}l$, hence it suffices to consider charts $Y'=\Spec(B)$ with $B=\Spec(Nor_L(A[\frac{t_1}{t_i}\.\frac{t_n}{t_i}]))$. Charts of the degree-0 torific blowings up are described similarly.

\subsubsection{Preservation of relatively affine actions and torific ideals}
In the following Lemma we say that a multiset $S$ in $L$ is balanced if $\sum_{l\in S}l=0$.

\begin{lemma}\label{Lem:canonicalblowup}
Assume that $G=\bfD_L$ is given a relatively affine action on a reduced scheme $X$ and $S$ is a locally constant coherent multiset in $L_X$, and consider the corresponding torific blowing up $b_{X,S}\:Y\to X$ and the degree-0 torific blowing up $b^0_{X,S}\:Y^0\to X$. Then,

(i) The induced action of $G$ on $Y$ and $Y^0$ is relatively affine.

(ii) For any finite multiset $T$ in $L$ the equality $I^{Y^0}_T=I_T^X\cO_{Y^0}$ holds.

(iii) If $S$ is balanced then the morphism of quotients $Y\sslash G\to X\sslash G$ is the normalized blowing up of $X_0=X\sslash G$ along $\pi_*(I^S_X)\cap\cO_{X_0}$, where $\pi\:X\to X_0$ is the quotient morphism and the intersection is taken inside of $\pi_*(\cO_X)$.
\end{lemma}
\begin{proof}
All questions are local on $X_0$, so we can replace $X$ with an open affine strongly $G$-equivariant subscheme $\Spec(A)$ such that $S$ is constant on $\Spec(A)$. Then $Y$ is covered by $G$-equivariant charts $Y_i=\Spec(B_i)$ for $B_i=Nor_L(A[\frac{t_1}{t_i}\.\frac{t_n}{t_i}])$. The intersection $Y_{ij}=Y_i\cap Y_j$ is the localization of $Y_i$ along $\frac{t_j}{t_i}$. Since the latter is of degree 0, the open embedding $Y_{ij}\into Y_i$ is strongly $G$-equivariant and we obtain that the action on $Y$ is relatively affine. A similar argument works for $Y^0$ and its affine covering $Y_i^0=\Spec(B_i^0)$, where $B_i^0=Nor_0(A[\frac{t_1}{t_i}\.\frac{t_n}{t_i}])$. This proves (i).

Since each $C=B_i^0$ is generated over $A$ by elements of degree 0, we have that $A_lC_0=C_l$ and therefore $I^C_T=I^A_TC$, proving (ii).

To prove (iii) we note that $X_0=\Spec(A_0)$ and $t_1\..t_n\in A_0$ by our assumption on $S$. Furthermore, $\pi_*(I^S_X)\cap\cO_{X_0}$ corresponds to the ideal $J=\sum_{i=1}^n t_iA_0$ and hence $NorBl_J(X_0)$ is covered by the charts $Z_i=\Spec(Nor(A_0[\frac{t_1}{t_i}\.\frac{t_n}{t_i}]))$. It remains to note that $Y_i\sslash G=Z_i$ since $Nor(A_0[\frac{t_1}{t_i}\.\frac{t_n}{t_i}])$ is the degree-0 part of $B_i$.
\end{proof}

\subsubsection{Composition of torific blowings up}
Torific blowings up are compatible with compositions in the following weak sense.

\begin{lemma}\label{Lem:composeblowups}
Given a relatively affine action of $G=\bfD_L$ on a reduced scheme $X$ and finite multisets $R,S,T$ of characters of $L$ with supports $|R|=|S|\cup|T|$, the $R$-torific blowing up $b_{X,R}$ factors into the composition of the degree-0 torific blowing up $b^0_{X,S}\:Y\to X$ and the torific blowing up $b_{Y,T}\:X_R\to Y$.
\end{lemma}
\begin{proof}
Set $Z=X^0_R$ for simplicity of notation. It suffices to prove that $b^0_{X,R}=b^0_{X,S}\circ b^0_{Y,T}$, since composing both sides of this equality with the $G$-normalization $Nor_G(Z)\to Z$ we will obtain the asserted equality $b_{X,R}=b^0_{X,S}\circ b_{Y,T}$.

By Lemma~\ref{Lem:canonicalblowup}(ii), $I^{Z}_R$ is the inverse image of $I^X_R$. So, $I^Z_R$ is invertible, and hence all ideals $I^Z_l$ with $l\in R$ are invertible. In particular, $I^Z_S$ is invertible. Note that a degree-0 normalized blowing up along $I$ is the universal $G$-equivariant morphism such that the inverse image of $I$ is invertible and the quotient of the source by the action of $G$ is normal. Therefore $b_{X,R}$ factors through $Y$ and $Z\to Y$ is the degree-0 normalized blowing up along the inverse image of $I_R^X$. The latter coincides with $I_R^Y$ by Lemma~\ref{Lem:canonicalblowup}, and so $Z=Y^0_R$. The same argument as above shows that $I_l^Y$ is invertible for any $l\in S$. In particular, the factors $I_l^Y$ in the formula $I_R^Y=I_T^Y\cdot\prod_{l\in S\setminus T}I_l^Y$ are invertible, and so $Y^0_T=Y^0_R=Z$, as required.
\end{proof}

\subsection{Torific blowings up of toroidal schemes}
In the sequel we denote toroidal schemes using divisors, i.e., instead of $(X,U)$ we will write $(X,D)$, where $D=X\setminus U$.

\subsubsection{Blowings up of toroidal schemes}
Let $(X,D)$ be a toroidal scheme and $Z\subseteq X$ a closed subscheme. Consider the normalized blowing up $f\:X'\to X$ along $Z$ and set $D'=f^{-1}(D\cup Z)$. In other words, $D'$ is the union of the preimage of $D$ and the exceptional divisor $f^{-1}(Z)$. The pair $(X',D')$ will be called {\em normalized blowing up} of $(X,D)$ along $Z$. For concreteness, we will also use the notation $f\:(X',D')\to(X,D)$, which formally means that $f\:X'\to X$ is a morphism such that $f^{-1}(D)\subseteq D'$. Furthermore, if $(X',D')$ is a toroidal scheme and the morphism $f\:(X',D')\to(X,D)$ is toroidal then we say that the normalized blowing up $f$ is {\em permissible}\index{permissible blowing up}.

\subsubsection{Toroidal blowings up}
The normalized blowing up $f\:(X',D')\to(X,D)$ along $Z$ is called a {\em toroidal blowing up} if $Z$ is a toroidal subscheme, see Section \ref{subschsec}.\index{toroidal!blowing up}

\begin{lemma}\label{toroidalblowlem}
Assume that $(X,D)$ is a toroidal scheme, $I\subset\cM_X$ is an ideal, and $f\:(X',D')\to(X,D)$ is the toroidal blowing up along $Z=V(I\cO_X)$. Then,

(i) $f$ permissible.

(ii) If $g\:(Y,E)\to(X,D)$ is a morphism of toroidal schemes and $(Y',E')\to(Y,E)$ is the toroidal blowing up along $Z\times_XY$ then $(Y',E')=(X',D')\times_{(X,D)}(Y,E)$, where the product is taken in the category of fine logarithmic schemes. In addition, if $g$ is strict then $Y'=X'\times_XY$.
\end{lemma}
\begin{proof}
Let $(\widetilde X,\cM_{\widetilde X})$ be the logarithmic blowing up of $(X,\cM_X)$ along $I$ in the sense of \cite[Section~4]{Niziol}. By \cite[Theorem~4.7]{Niziol} and \cite[Proposition~10.3]{Kato-toric}, $(\widetilde X,\cM_{\widetilde X})$ is logarithmically regular. Since $\widetilde X=X'$ by \cite[Lemma~4.3]{Niziol} and it is easy to see that the preimage of $X\setminus D$ in $\widetilde X$ is the triviality locus of $\cM_{\widetilde X}$, it follows that $(X',D')$ is the toroidal scheme corresponding to $(\widetilde X,\cM_{\widetilde X})$. In addition, any logarithmic blowing up is the pullback of a logarithmic blowing up of charts, hence it is logarithmically smooth. Thus, $f$ is toroidal and we obtain (i).

We observed above that logarithmic blowings up of toroidal schemes correspond to toroidal blowings up, hence the first assertion of (ii) follows from \cite[Corollary~4.8]{Niziol}. To prove the second claim of (ii) consider the product $(Z,\cM_Z)=(X',\cM_{X'})\times_{(X,\cM_X)}(Y,\cM_Y)$ in the category of all logarithmic schemes. Since $Z=X'\times_XY$ by \cite[1.6]{Kato-log}, it suffices to show that if $g$ is strict then $(Z,\cM_Z)$ is fine and hence this is also the product in the category of fine schemes. By \cite[1.6]{Kato-log}, $\cM_Z$ is the logarithmic structure associated with the pushout of the pullbacks of $\cM_{X'}$ and $\cM_Y$ over the pullback of $\cM_X$. Since $\cM_Y$ is the pullback of $\cM_X$, we also have that $\cM_Z$ is the pullback of $\cM_{X'}$, and hence $\cM_Z$ is fine.
\end{proof}

\subsubsection{Torific blowings up: toroidal action}\label{torificblowsec}
If $G$ acts on $(X,D)$ and $X'\to X$ is a torific blowing up in the sense of Section \ref{Sec:torific-blowups}, then we say that the normalized blowing up $f\:(X',D')\to(X,D)$ is a {\em torific blowing up} of the toroidal scheme $(X,D)$. Obviously, the morphism $f$ is $G$-equivariant. One can easily construct examples of torific blowings up which are not toroidal blowings up, for example, with $D=\emptyset$. On the other hand, if the action is toroidal then one can tightly control torific blowings up by toroidal methods:

\begin{lemma}\label{torificblowlem}
Assume that $G=\bfD_L$ acts in a relatively affine and toroidal fashion on a toroidal scheme $(X,D)$ and let $b_{X,S}\:(X',D')\to(X,D)$ be the torific blowing up along a torific ideal $I^X_S$. Then,

(i) The ideal $I_S^X$ and the blowing up $b_{X,S}$ are toroidal.

(ii) The action on $(X',D')$ is toroidal. Furthermore, if the action on $(X,D)$ is taut or loose, then the same holds for the action on $(X',D')$.

(iii) If $g\:(Y,E)\to(X,D)$ is a strongly equivariant strict morphism of toroidal schemes then the torific blowing up $b_{Y,g^*(S)}\:Y'\to Y$ is the base change of $b_{X,S}\:X'\to X$.
\end{lemma}
\begin{proof}
Step 1: toric schemes. Assume first that $X=\bfA_{M}$ with the toric divisor $D$, $S$ is constant, and $G$ acts via $\chi\:M\to L$, and let us check (i) and (ii). The ideal $I^X_S$ is generated by elements $m_1\.m_n\in M$, so $b_{X,S}$ is toric, hence toroidal. Furthermore, $X'$ is glued from blowing up charts $X'_i$ of the form $\bfA_{M'_i}$ where $M'_i$ is the saturation of $M_i=M[m_1-m_i\.m_n-m_i]$ and $G$ acts through the homomorphism $M'_i\to L$ extending $\chi$. Indeed, $\bfA_{M'_i}$ is normal and hence it is both the normalization and $G$-normalization of $\bfA_{M_i}$, which is the chart of the blowing up along $I^X_S$. By Lemma~\ref{toroidalactionlem}(iii) the action is toroidal. If the action on $(X,D)$ is taut then $\Ker(\chi)$ contains an inner element $v\in M$. Since $v$ is also inner in the larger monoids $M'_i$, the action on $(X',D')$ is taut. The claim about looseness is also simple, so we omit the justification.

Step 2: the general case. It suffices to prove the claims locally over a point $x\in X$, hence we can replace $G$, $X$ and $S$ with $G_x$, the $G_x$-equivariant localization of $X$ at $x$, and $S_x$, respectively, reducing to the case where the action on $X$ is strictly local and $S$ is constant. By Corollary~\ref{toroidalprop}, the toroidal scheme $(X,D)$ admits a strongly equivariant chart $h\:(X,D)\to(X_0,D_0)$ with $X_0=\bfA_{M}$ and $D_0=\bfA_M\setminus\bfA_{M^\gp}$. The statement has been proven in Step 1 above for the target, hence (i) and (ii) follow from the fact that all ingredients are compatible with strongly equivariant strict morphisms $(Y,E)\to(X,D)$: Lemma~\ref{Lem:functorideal} says that $I_S^Y = I^{X}_S \cO_Y$, Lemma~\ref{toroidalblowlem}(ii) shows that $b_{Y,S}$ is the base change of $b_{X,S}$, Lemma~\ref{toroidalactionlem}(ii) shows that the action on $(Y',E')$ is toroidal, and compatibility of tautness and  looseness follows from Lemma~\ref{tautlem}. In particular, this argument shows that $b_{X,S}$ is the base change of $b_{X_0,S}$, and the same is true for $b_{Y,S}$ since the induced morphism $(Y,E)\to(X_0,D_0)$ is a strongly equivariant chart. This implies (iii).
\end{proof}

\subsubsection{Torific blowings up: simple action}
In case of simple actions we still have the following result.

\begin{corollary}\label{torificcor}
Assume that $G=\bfD_L$ acts in a relatively affine and $G$-simple fashion on toroidal schemes $(X,D)$ and $(Y,E)$ and $f\:(Y,E)\to(X,D)$ is a strongly equivariant strict morphism. Then for any locally constant coherent multiset $S$ in $L_X$

(i) the torific blowing up $b_{X,S}$ is the normalized blowing up along $I_S^X$, and

(ii) the torific blowings up are compatible: $b_{Y,f^*(S)}$ is the base change of $b_{X,S}$.
\end{corollary}
\begin{proof}
It suffices to check (ii) at a point $y\in Y$, hence we can replace $G$ with $G_y$ and $Y$ and $X$ with their $G_y$-equivariant localizations at $y$ and $f(y)$, respectively. By Proposition~\ref{pretoroidalprop} there exists a larger divisor $\oD\supseteq D$ such that the action on $(X,\oD)$ is toroidal. Set $\oE=\oD\times_XY$, then $(Y,\oE)\to(X,\oD)$ is a strict morphism and hence $b_{Y,S}$ is the base change of $b_{X,S}$ by Lemma~\ref{torificblowlem}(iii).

In the same manner, to check (i) we can pass to a $G_x$-equivariant localization $X_x$, and then after increasing the toroidal structure the action becomes toroidal. Then $b_{X,S}$ becomes a toroidal blowing up by Lemma~\ref{torificblowlem}(i), in particular, it is a normalized blowing up.
\end{proof}

\subsection{The torifying property}\label{torifsec}
The main property of torific blowings up is that they \emph{torify} simple actions on toroidal schemes. This is the content of the following theorem, whose proof occupies the whole Section~\ref{torifsec}.

\begin{theorem}\label{torificationth}
Assume that a toroidal scheme $(X,D)$ is provided with a $G$-simple relatively affine action of a diagonalizable group $G=\bfD_L$ and $S$ is a locally constant coherent multiset of characters in $L_{(X,D)}$ containing the support of $\sigma_{(X,D)}$. Then the torific blowing up $b_{X,S}\:(X',D')\to(X,D)$ is permissible and $G$ acts on $(X',D')$ toroidally. Moreover, if the action on $(X,D)$ is potentially taut or loose then the action on $(X',D')$ is taut or  loose, respectively.
\end{theorem}

Note that it may happen that $I_S^X=0$. In this case the assertion of the theorem holds true but becomes vacuous.

\subsubsection{Proof of Theorem~\ref{torificationth}: the plan}
The proof will be in three steps. First, we establish the model case of toric varieties over $\ZZ$ (Section \ref{modelcasesec}). Then, we  deduce the case when the action is strictly local (Section \ref{strloccase}). Finally we deduce the general case in Section \ref{torifgencase}.

\subsubsection{The model case}\label{modelcasesec}
Let $G=\bfD_L$. As in Section \ref{combsec}, consider the following situation that models $G$-simple actions: $M$ is an $L$-graded toric monoid, $P=M\oplus\NN^\sigma$ where $\sigma\:L\setminus\{0\}\to\NN$ is a function with finite support, $X=\bfA_P$, $D=\bfA_P\setminus\bfA_{M^\gp\oplus\NN^\sigma}$ and $\oD=\bfA_P\setminus\bfA_{P^\gp}$. We denote the grading of $P$ by $\chi\:P\to L$. Both $(X,D)$ and $(X,\oD)$ are $G$-equivariant toroidal schemes, and we have a $G$-equivariant chart $h\:(X,D)\to(\bfA_M,\bfA_{M^\gp})$. The actions on $(X,\oD)$ and $(\bfA_M,\bfA_{M^\gp})$ are toroidal by Proposition~\ref{toroidalactionlem}(iii).

\begin{lemma}\label{modelcaselem}
Keeping the above notation, let $S$ be a finite multiset in $L$ containing the support of $\sigma$. Consider the $S$-torific blowing up $b_{X,S}\:X'\to X$ and let $D'\subset X'$ be the union of the preimage of $D$ and the exceptional divisor. Then the action of $G$ on $(X',D')$ is toroidal. In addition, if the action on $(X,\oD)$ is taut or  loose then the same is true for $(X',D')$.
\end{lemma}

The following special model case is established in \cite[Proposition 3.2.5(2)]{AKMW}: $L=\ZZ$ and $X=\bfA_{P,K}$, where $K$ is a field of characteristic zero. However, the arguments are combinatorial and apply to our more general situation verbatim. For the sake of completeness we briefly outline the proof.

\begin{proof}
We can assume that $I_S^X\neq 0$ as otherwise $X'$ is empty and there is nothing to prove. Let $z^m$ denote the image of $m\in P$ in $\ZZ[P]$ and $t_i=z^{e_i}$, where $e_1\.e_r$ are the generators of $\NN^\sigma$. Note that $D$ is obtained from $\oD$ by removing $r$ divisors $D_i:=V(t_i)$. Hence $D'$ is obtained from the preimage $\oD'$ of $\oD$ by removing the strict transforms $D'_1\.D'_r$ of $D_1\.D_r$. The action on $(X',\oD')$ is toroidal by Lemma~\ref{torificblowlem}(ii), and by Proposition~\ref{decreaseprop}(ii) it suffices to prove that it remains toroidal if we remove from $\oD'$ a single component $D'_i$. Thus, replacing $D$ and $D'$ by $\oD-D_i$ and $\oD'-D'_i$ we can assume in the sequel that $r=1$ and hence also $i=1$.

Set $e=e_1$ and $l=\chi(e)$, and split $P$ as $Q\oplus\NN e$ (i.e. $e$ is the generator of the second summand). First we consider the particular case when $S=\{l\}$. We can choose generators $z^e,z^{q_1}\.z^{q_n}$ of $I^{\ZZ[P]}_l$ such that $q_j\in Q$. The strict transform of $D_i$ is non-empty on the charts $\bfA_{P_j}$ where $P_j$ is the saturation of $$N_j=P[e-q_j,q_1-q_j\.q_n-q_j].$$ Clearly, $$N_j=Q[q_1-q_j\.q_n-q_j]\oplus\NN(e-q_j)$$ and the generator of the second summand is of degree zero. Hence after removing from $\oD'$ the strict transform $D'_1=V(e-q_j)$ the action remains toroidal by Proposition~\ref{decreaseprop}(i). In addition, if $\Ker(\chi)$ contains an inner vector of $P$ then the same vector is inner in $P_j$ and lies in the kernel of $\chi_j\:P_j\to L$, and if $P\cap\Ker(\chi)$ spans $\Ker(\chi^\gp)$ then the same is true for $\Ker(\chi_j)\supseteq\Ker(\chi)$. It follows that tautness and looseness are preserved in this case.

We further claim that $b_{X,\{l\}}=b_{X,\{l\}}^0$. Indeed, this follows from the fact that $P_j$ is the direct sum of the saturation of $Q[q_1-q_j\.q_n-q_j]$ and $\NN(e-q_j)$, and hence $P_j$ is obtained from $N_j$ by saturating its degree-0 part.

Now let us consider the case of an arbitrary $S$.  Lemma~\ref{Lem:composeblowups} applies to the sum $S=S'+\{l\}$ and we obtain that $b_{X,S}$ is the composition of the torific blowings up $b_{X,\{l\}}\:X''\to X$ and $b_{X'',S'}\:X'\to X''$. Let $D''\subset X''$ be the union of  $b_{X,\{l\}}^{-1}(D)$ and the exceptional divisor of $b_{X,\{l\}}$. By the above case, $(X'',D'')$ is a toroidal scheme acted on toroidally by $G$. Since $D'$ is the union of the preimage of $D''$ and the exceptional divisor, Lemma~\ref{torificblowlem} tells us that $b_{X'',S'}\:(X',D')\to(X'',D'')$ is a toroidal blowing up and the action on the source is toroidal. The same lemma also says that tautness and looseness are preserved by $b_{X'',S'}$.
\end{proof}

\subsubsection{The strictly local case}\label{strloccase}
Back to the proof of Theorem~\ref{torificationth}. Assume that the action on $X$ is strictly local, and let $x$ be the closed $G$-invariant point. Recall that the family $S$ is constant, so we can identify it with a multiset $S_x$ in $L$. By Theorem~\ref{combchart} there exists a strongly equivariant chart $f\:(X,D)\to(X_0,D_0)$, where $P=\oM_x\oplus\NN^{\sigma_x}$, $X_0=\bfA_P$ and $D_0=\bfA_P\setminus\bfA_{\oM_x^\gp\oplus\NN^{\sigma_x}}$. Applying Corollary~\ref{torificcor} to $f$ we obtain that the torific blowings up $b_{X,S}$ and $b_{X_0,S}\:(X'_0,D'_0)\to(X_0,D_0)$ are compatible, thus giving rise to an equivariant chart $(X',D')\to(X'_0,D'_0)$. Since $S$ contains the support of $\sigma_x$, Lemma~\ref{modelcaselem} applies to $b_{X_0,S}$ and we obtain that the action on $(X'_0,D'_0)$ is toroidal. Thus, the action on $(X',D')$ is toroidal by Lemma~\ref{toroidalactionlem}(ii).

Finally, if the action on $(X,D)$ is potentially taut (resp.  potentially loose) then the same is true for $(X_0,D_0)$, and hence the action on $(X_0,\oD_0)$ is taut (resp. loose), where $\oD_0=\bfA_P\setminus\bfA_{P^\gp}$. Then the action on $(X'_0,D'_0)$ is taut (resp. loose) by Lemma~\ref{modelcaselem}, and using Lemma~\ref{tautlem} we obtain that the action on $(X',D')$ is taut (resp. loose).

\subsubsection{The general case}\label{torifgencase}
Now, let us prove Theorem~\ref{torificationth} in general. By Lemma~\ref{localtorificlem} torific blowings up are compatible with the equivariant localization $X_O\into X$ at a closed orbit $O$. Since $X$ is covered by such localizations, we can replace $X$ with $X_O$. Moreover, we can replace $G$ with $G_O$. Indeed, compatibility of torific blowings up is guaranteed by the same Lemma~\ref{localtorificlem} and $G$ acts on $(X',D')$ toroidally if and only if $G_O$ acts toroidally, since the stabilizer of any point of $X$ and $X'$ is contained in $G_O$. Thus, we reduce to the situation when the closure of any orbit contains a $G$-invariant point, and the same localization argument implies that it now suffices to prove the theorem for equivariant localizations of $X$ at $G$-invariant points. The latter localizations are strictly local, and the theorem was proved for them in Section \ref{strloccase}. This completes the proof of Theorem \ref{torificationth}. \qed

\subsection{Main torification theorems}
Our main torification results are obtained by applying Theorem~\ref{torificationth} to the locally constant coherent multisets $\sigma_X$ and $\sigma^0_X$ (see Lemma~\ref{Sigmalem}(i)). To simplify notation, set $I^X_\sigma=I^X_{\sigma_X}$ and $I^X_{\sigma^0}=I^X_{\sigma^0_X}$.

\subsubsection{Full actions}
First, let $X$ be a toroidal scheme provided with a relatively affine $G$-action, where $G = \bfD_L$. If $U \subset X\sslash G$  is affine we write  $$X_U = U \times_{X\sslash G} X = \Spec \bigoplus_{l\in L} A^U_l.$$ We say that the action of $G$ on $X$ is {\em full}\index{full action} if for every nonempty affine $U \subset X\sslash G$ and every $l\in L$ we have $A^U_l \neq 0$. Note that if there is a strongly equivariant dense open set $X_U$ where the action is free, then the action is full.

\subsubsection{Balanced torification}
We start with torification by use of the balanced multiset $\sigma_X^0$.

\begin{theorem}\label{maintheorem}
Assume that a toroidal scheme $(X,D)$ is provided with a relatively affine,  $G$-simple action of a diagonalizable group $G=\bfD_L$. Consider the torific blowing up $F^0_{(X,D)}\:X'\to X$ along $I_{\sigma^0}^X$, and let $D'$ be the union of the preimage of $D$ and the exceptional divisor of $F^0_{(X,D)}$. Then,

(i) The pair $(X',D')$ is toroidal and the natural $G$-action on $(X',D')$ is toroidal.

(ii) The morphism of quotients  $X'\sslash G\to X_0=X\sslash G$ is the normalized blowing up along $\pi_*(I^X_{\sigma^0})\cap\cO_{X_0}$, where $\pi\:X\to X_0$ is the quotient morphism.

(iii) The torific blowing up $F^0_{(X,D)}$ is functorial with respect to strongly equivariant strict morphisms $(Y,E)\to(X,D)$ of toroidal schemes in the sense that $F^0_{(Y,E)}\:Y'\to Y$ is the base change $F^0_{(X,D)}\times_XY$ of $F^0_{(X,D)}$.

(iv) If moreover the action of $G$ is full, then $F^0_{(X,D)}\:X'\to X$  is birational.
\end{theorem}
\begin{proof}
Part (i) follows from Theorem~\ref{torificationth} and (ii) follows from Lemma~\ref{Lem:canonicalblowup}(iii). Recall that $\sigma_Y^0=h^*(\sigma_X^0)$ by Lemma~\ref{Sigmalem}(ii), hence part (iii) follows from Corollary~\ref{torificcor}. Finally, if the action is full the ideal $I_{\sigma_0}^X$ is nowhere 0, and since $X$ is toroidal, hence normal, the resulting blowing up is birational, giving part (iv).
\end{proof}

\subsubsection{Torification of arbitrary actions}
Alternatively, one can use the unbalanced  multiset $\sigma_X$ for torification. Then there is no good description of the quotients, but the advantage is that the blowing up is birational without requiring a full action. Moreover, we can combine this blowing up with barycentric subdivisions obtaining a torification in general.

\begin{theorem}\label{maintheorem2}
Assume that a diagonalizable group $G$ acts in a relatively affine manner on a toroidal scheme $(X,D)$. Let $F_{(X,D)}\:X'\to X$ be the composition of the barycentric modification $(Y,E)\to(X,D)$ and the torific blowing up along $I_{\sigma}^{{Y}}$ and let $D'$ be the union of the preimage of $D$ and the exceptional divisor of $F_{(X,D)}$. Then,

(i) The morphism $X'\to X$ is birational.

(ii) The pair $(X',D')$ is toroidal and the natural $G$-action on $(X',D')$ is toroidal.

(iii) The morphism $F_{(X,D)}$ is functorial with respect to strongly equivariant strict morphisms $h\:(Y,E)\to(X,D)$ of toroidal schemes in the sense that $F_{(Y,E)}$ is the base change of $F_{(X,D)}$.
\end{theorem}
\begin{proof}
In (i) one should only check that $I_{\sigma}^X$ is non-zero at generic points of $X$. Torific ideals and the multisets $\sigma$ are compatible with equivariant localizations (see Lemmas~\ref{localtorificlem} and \ref{Sigmalem}(i)), hence we can assume that the action is strictly local with $G$-invariant point $x$ and $\sigma=\sigma_x$ is in $L$. Then each $l\in\sigma$ appears as a character through which $G$ acts on $m_x/m_x^2$ and hence $I^X_l\neq 0$. Therefore, $I_\sigma^X\neq 0$ and since $X$ is normal the torific blowing up $X'\to X$ is birational.

Part (ii) follows from Theorem~\ref{torificationth}.

By Lemma~\ref{Sigmalem}(ii) the formation of $\sigma_X$ is functorial. By  Corollary~\ref{torificcor} the corresponding normalized blowing up is functorial. By Proposition~\ref{makingsimpleprop} the barycentric blowings up are equivariant, Also barycentric subdivisons are compatible with strict morphisms. Hence Part (iii) follows.

\end{proof}

\subsection{Some cases with trivial torification}
In this section we will describe two cases when the torification is trivial. This will be used in \cite{AT2}.

\subsubsection{Free actions}
As one might expect, free actions (see \cite[Section 5.4]{ATLuna}) on regular schemes do not require a torification.

\begin{lemma}\label{trivtor}
Assume that $(X,\emptyset)$ is a toroidal scheme (i.e. $X$ is a regular scheme) and a diagonalizable group $G=\bfD_L$ acts freely on $X$. Then $I_X^S=1$ for any finite multiset $S$ in $L$, and hence $F_{(X,\emptyset)}=\Id_X$.
\end{lemma}
\begin{proof}
The action is simple so the torification is described by Theorem~\ref{maintheorem}. The question is local over $X/G$, hence we can assume that $X$ is affine. Moreover, by \cite[Lemma 5.4.4]{ATLuna} we can assume that there exists a strongly equivariant flat covering $Y\to X$ such that the action on $Y$ is split free. Since the $S$-torific ideals are compatible with strongly equivariant morphisms by Lemma~\ref{Lem:functorideal}, it suffices to check the claim for $Y=\Spec(A)$ with a split free action. But then $A=\oplus_L A_n$, hence each $A_n$ contains a unit and therefore $I_Y^S=1$.
\end{proof}

\subsubsection{Deformation to the normal cone}
Here is a more interesting example. Assume that $X$ is a regular scheme and $Z\into X$ is a regular closed subscheme. Set $B_0=X\times\bfA^1$ and let $Z_0=Z\times\{0\}$ be the preimage of $Z$ in the zero section of $B_0$. We provide $B_0$ with the action of $\GG_m$ via the standard action on $\bfA^1$. Since $B_0$ and $Z_0$ are regular the blowing up $B=Bl_{Z_0}(B_0)$ is regular too, and we can view it as a toroidal scheme with empty divisor. By $E=Z_0\times_{B_0}B$ we denote the exceptional divisor. Furthermore, $Z_0$ is $\GG_m$-equivariant hence the action lifts to $B$.

\begin{lemma}
Keep the above notation. Then the action of $\GG_m$ on $B$ is relatively affine, $B\sslash\GG_m=Bl_Z(X)$, and the torification $F{(X,\emptyset)}$ is the blowing up of the exceptional divisor $E$, in particular, $F_{(X,\emptyset)}$ is an isomorphism.
\end{lemma}
\begin{proof}
Again, the action is simple since the toroidal divisor is empty. The claim is local on $X$ hence we can assume that $X=\Spec(A)$ for a regular local ring $A$ with local parameters $t_1\.t_n$ and $Z$ is given by the vanishing of $I=(t_1\.t_m)$. Then $B_0=\Spec(A[x])$ and $B=\Proj_A(\oplus_d J^d)$, where $J=I[x]$, and we provide $B$ with the grading such that $I$ is of weight zero and $x$ is of weight 1. Then, the charts $$B_i=\Spec\left(A\left[\frac{t_1}{t_i}\.\frac{t_m}{t_i},t_{m+1}\.t_n,x\right]\right)$$ for $1\le i\le m$ of $B$ are strongly equivariant and hence we have a strongly equivariant isomorphism $B=Bl_Z(X)\times\bfA^1$, where $\GG_m$ acts through the standard action on $\bfA^1$. Thus, the quotient is $Bl_Z(X)$, the multiset $\sigma_B$ is constant and coincides with $\{1\}$, and the torific ideal $I_X^S$ is given by $(x)$.
\end{proof}

\section{Torification by blowings up} \label{Sec:Torification-bu}
Our last goal in the paper is to show that normalized blowings up in the torification theorem can be replaced by non-normalized ones. In particular, they are always of finite type without any additional assumption on the schemes beyond the noetherian hypothesis. Many ideas of these section, especially those related to normalization and saturation thresholds, were developed in the joint work \cite{Illusie-Temkin} of the second author with Luc Illusie in an attempt to remove the quasi-excellence assumption in Gabber's torification theorem \cite[Theorem~1.1]{Illusie-Temkin}. Gabber's original argument applied to all noetherian schemes, but working with qe schemes simplified some points and the plan was to remove the qe assumption in the end via uniform normalization of relevant ideals. In the end, this was not fully worked out due to publication deadline, and we are grateful to Luc Illusie for allowing us to include some of that material in our current work.

\subsection{Toric case}
Our first aim is to show that a toric blowing up (i.e. the normalized blowing up of a toric ideal) can always be realized as the ordinary blowing up of an appropriate toric ideal. The basic idea is to saturate the ideal $I$ instead of saturating the blowing up along $I$. This, indeed, works once one replaces $I$ with an appropriate power $I^n$ -- an operation that does not affect the blowing up morphism.

\subsubsection{Charts of toric blowings up}
Let $P$ be a toric monoid and $(X,U)=(\bfA_P,\bfA_{P^\gp})$ the corresponding toric $\ZZ$-scheme. There is a natural bijection between ideals $I\subseteq P$ and toroidal ideals $\cI\subseteq\cO_X$, it is given by $\cI=\ZZ[I]$ and $I=\cI\cap P$. If $I$ is generated by elements $a_1\.a_n$ then the blowing up $Y=Bl_\cI(X)$ is covered by charts $Y_i=\bfA_{P_i}$ for $1\le i\le n$, where $P_i=P[I-a_i]$ is the submonoid of $P^\gp$ generated by $P$ and $a_1-a_i\.a_n-a_i$. The toric blowing up $\tilY=Y^\nor$ along $\cI$ is covered by the normalizations $\tilY_i=Y_i^\nor$, and we claim that $\tilY_i=\Spec\ZZ[P_i^\sat]$. Indeed, the normalization of $\ZZ[P_i]$ contains $P_i^\sat$. The ring $\ZZ[P_i^\sat]$ is normal since $P_i^\sat$ is a toric monoid, and so the inclusion $\ZZ[P_i^\sat]\subseteq\ZZ[P_i]^\nor$ is an equality.

\subsubsection{Saturation of ideals}\label{satidealsec}
We will also consider {\em saturation} of ideals in a monoid $P$. By definition, $I^\sat$ consists of elements $a\in P$ such that $na\in nI$ for some $n>0$.\index{saturation!of a monomial ideal}

\begin{lemma}\label{satideallem}
Let $P$ be a toric monoid with an ideal $I$. Then for any $n\ge n(P,I)$, the toric blowing up of $\bfA_P$ along $\ZZ[I]$ coincides with the usual blowing up of $X$ along the toric ideal $\ZZ[(nI)^\sat]$.
\end{lemma}
\begin{proof}
First, we claim that for any $a\in I$, the equality $P[I-a]^\sat=P[(nI)^\sat-na]$ holds for $n\gg 0$. The monoid $P[I-a]^\sat$ is fine, hence we can fix $x\in P[I-a]^\sat$ and it suffices to prove that $x\in P[(nI)^\sat-na]$ for $n\gg 0$. Since $mx\in P[I-a]$ for some $m>0$, we have that $mx=b-ka$ for $k>0$ and $b\in kI$. Take any $n$ such that $nm\ge k$. Then $mx=c-mna$ with $c=b+(mn-k)a\in mnI$. Since $P$ is saturated, $x+na=\frac{1}{m}c\in P$ and therefore $\frac{1}{m}c\in(nI)^\sat$. So, $x=\frac{1}{m}c-na\in P[(nI)^\sat-na]$ as claimed.

Fix generators $a_1\.a_l$ of $I$ and let $P_i=P[I-a_i]^\sat$. Then the toric blowing up $X'\to X$ along $\ZZ[I]$ is covered by the charts $X'_i=\bfA_{P_i}$. By the above paragraph, for a large enough $n$ we have that $P_i=P[(nI)^\sat-na_i]$ for $1\le i\le l$. Then $X'_i$ are also charts of the blowing up $X''\to X$ along $\ZZ[(nI)^\sat]$. So, $X'$ embeds into $X''$ as an open subscheme, and since both are projective birational over $X$, they coincide.
\end{proof}

\subsubsection{Saturation thresholds of monoidal ideals}\label{Sec:sat-monoidal}
By the saturation threshold $n^\sat(P,I)$ of the pair $(P,I)$ we mean the minimal number $n(P,I)$ that satisfies the assertion of Lemma~\ref{satideallem}.\index{saturation!threshold}

\subsection{The toroidal case}
Next, let us extend the above theory to toroidal blowings up of toroidal schemes.

\subsubsection{Saturation of toroidal ideals}\label{Sec:sat-toroidal}
Assume that $(X,U)$ is a toroidal scheme with a global monoidal chart $\alpha\:P\to\cO_X$. Then any toroidal ideal $\cI\subset\cO_X$ is of the form $\alpha(I)\cO_X$ for an ideal $I$ of $P$. In fact, this representation is unique by \cite[Lemma~3.4.3]{Illusie-Temkin}. By the {\em saturation} of $\cI$ we mean the ideal $\cI^\sat:=\alpha(I^\sat)\cO_X$. It is easy to see that this construction is independent of the choice of a chart and hence globalizes to toroidal ideals $\cI$ on arbitrary toroidal schemes $(X,U)$. Also, for any $x\in X$ we will use the notation $\oM_x=\ocM_{X,x}$ and denote by $\ocI_x$ the image of $\alpha_x^{-1}(\cI_x)$ under the map $\cM_{X,x}\to\oM_x$.

\subsubsection{Saturation thresholds of toroidal ideals}
We define the {\em saturation threshold} of a toroidal ideal $\cI$ to be the minimal number $n^\sat(X,U,\cI)$ such that for any $n\ge n^\sat(X,U,\cI)$ the toroidal blowing up along $\cI$ coincides with the blowing up along $(\cI^n)^\sat$.

\begin{theorem}\label{satidealth}
Assume that $(X,U)$ is a toroidal scheme with a toroidal ideal $\cI$.
Then the saturation threshold $n^\sat(X,U,\cI)$ is finite and $$n^\sat(X,U,\cI)=\max_{x\in X}n^\sat(\oM_x,\ocI_x).$$
\end{theorem}
\begin{proof}
The finiteness of the righthandside is clear since both $\oM_x$ and $\ocI_x$ are locally constant along the logarithmic strata of $(X,U)$. Also, it suffices to prove the assertion of the theorem for elements of a finite open covering of $X$, hence we can assume that $(X,U)$ possesses a global monoidal chart $\alpha\:P\to\cO_X$ such that the associated toroidal chart $h\:X\to\bfA_P$ is central and sharp. In particular, $\cI=\alpha(I)\cO_X$ for an ideal $I\subseteq P$ and hence $\cI$ is the pullback of the toric ideal $\ZZ[I]$ on the associated toroidal chart. Moreover, since the chart is central and sharp, $P=\oM_x$ for a point $x\in X$ and then clearly $I=\ocI_x$.

To deduce the theorem from Lemma~\ref{satideallem} it now suffices to check that all constructions are compatible with the charts. First, it is shown in the proof of \cite[Proposition~4.3]{Niziol} that $\alpha(nI)\cO_X=\cI^n$ (one can also deduce this from \cite[Lemma~3.4.3]{Illusie-Temkin}) and hence $\alpha((nI)^\sat)\cO_X=(\cI^n)^\sat$. Second, both toroidal blowings up and (non-normalized) blowings up along toroidal ideals are compatible with charts by \cite[Lemma~3.4.6]{Illusie-Temkin}: the toroidal blowing up of $X$ along $\cI$ is the pullback of the toroidal blowing up of $\bfA_P$ along $\ZZ[I]$, and similarly for blowings up along $(\cI^n)^\sat$ and $\ZZ[(nI)^\sat]$.
\end{proof}

\subsection{Normalization threshold}

\subsubsection{Normalization of ideals}\label{noridealsec}
Let $A$ be a reduced ring with an ideal $I$. The {\em normalization} $I^\nor$ of $I$ consists of all elements $a\in A$ that satisfy a monic equation $t^n+\sum_{i=0}^{n-1}f_{n-i}t^i$ with $f_i\in I^i$. This construction is compatible with localizations and hence generalizes to ideals $\cI\subseteq\cO_X$ on a scheme $X$.\index{normalization!of an ideal}

\begin{lemma}\label{norideallem} Let $X$ be a scheme with an ideal $\cI\subseteq\cO_X$. Then $Bl_{(\cI^n)^\nor}(X)$ refines $Bl_\cI(X)$ and the modification $Bl_{(\cI^n)^\nor}(X)\to Bl_\cI(X)$ is finite. In particular, the normalized blowings up along $\cI$ and $(\cI^n)^\nor$ are equal for any integral $n>0$.
\end{lemma}
\begin{proof}
It is well known that $Bl_\cI(X)= Bl_{\cI^n}(X)$. Indeed, a product of two ideals is invertible if and only if each of the ideals is invertible, hence this follows from the universal property of blowings up. So, we can assume that $n=1$ in the sequel.

The claim is local on $X$, hence we can assume that $X=\Spec(A)$ and $\cI,\cJ=\cI^\nor$ are given by ideals $I,J=I^\nor\subseteq A$. Note that for any $a\in I$ the  ring extension $A[a^{-1}J]/A[a^{-1}I]$ is integral and hence finite. So, for any chart $\Spec(A[a^{-1}I])$ of $ Bl_\cI(X)$, we have a finite morphism $\Spec(A[a^{-1}J])\to\Spec(A[a^{-1}I])$ whose source is a chart of $ Bl_\cJ(X)$. Therefore, $Y=\cup_{a\in I}\Spec(A[a^{-1}J])$ is an open subscheme of $ Bl_\cJ(X)$ which is finite over $ Bl_\cI(X)$. It follows that $Y= Bl_\cJ(X)$ and we obtain a finite morphism $h: Bl_\cJ(X)\to Bl_\cI(X)$, thus proving the lemma.
\end{proof}

\subsubsection{Normalization threshold}\label{northreshsec}
Let $\cI$ be an ideal on $X$. We define the {\em normalization threshold} of $\cI$ as the minimal number $n^\nor(X,\cI)$ such that $Bl_{\cI}(X)^\nor\toisom Bl_{{(\cI^n)}^\nor}(X)$ for any $n\ge n^\nor(X,\cI)$. We write $n^\nor(X,\cI)=\infty$ if no such number exists. For example, this is the case when $Bl_I(X)^\nor$ is not of finite type over $X$.\index{normalization!threshold}

\subsubsection{Normalization of toroidal ideals}
Our next goal is to show that for toroidal ideals normalization and saturation thresholds coincide. This is based on the fact that normalization of toroidal ideals is already achieved by saturation:

\begin{lemma}
Assume that $P$ is a toric monoid and $I\subset P$ is an ideal. Then $\ZZ[I^\sat]=(\ZZ[I])^\nor$.
\end{lemma}
\begin{proof}
First, assume that $I$ is saturated and let us show that $J=\ZZ[I]$ is a normal ideal of $A=\ZZ[P]$. Consider the Rees algebra $A'=\oplus_{n=0}^\infty J^n$; it can be realized as the subalgebra $A[Jt]$ in $A[t]$. Clearly, if $J$ is not normal then $A'$ is not normal, hence it suffices to check that $A'$ is normal. Note that $A'=\ZZ[P']$, where $P'=\oplus_{n=0}^\infty nI$ is the submonoid of $P\oplus\NN$ consisting of the elements of the form $(a_n,n)$ with $a_n\in nI$. Also, it is clear that $P'$ is saturated if and only if $P$ and $I$ are saturated, which is our case. Thus, $P'$ is a toric monoid and hence $A'=\ZZ[P']$ is normal, as claimed.

Let now $I$ be an arbitrary ideal of $P$. By the above case, $\ZZ[I^\sat]$ is normal. Since, clearly $\ZZ[I]\subseteq\ZZ[I^\sat]\subseteq\ZZ[I]^\nor$, the right inclusion is necessarily an equality.
\end{proof}

\begin{corollary}\label{norsatcor}
Let $(X,U)$ be a toroidal scheme with a toroidal ideal $\cI$. Then $\cI^\sat=\cI^\nor$ and $n^\nor(X,\cI)=n^\sat(X,U,\cI)$, in particular, $n^\nor(X,\cI)$ is finite.
\end{corollary}
\begin{proof}
We prove that $\cI^\sat=\cI^\nor$ first. This claim is local, so we can assume that $X=\Spec(A)$ and work with the ideal $I\subseteq A$ corresponding to $\cI$. We can also assume that $(X,U)$ admits a sharp central toroidal chart $f\:X\to\Spec\ZZ[M]$ corresponding to a homomorphism $\alpha\:M\to A$ and $I=\alpha(J)A$ for an ideal $J\subseteq M$. We should prove that the inclusion $I^\sat=\alpha(J^\sat)A\subseteq I^\nor$ is an equality, and it suffices to show that if $J$ is saturated then $I$ is normal. As in the above lemma, it suffices to show that the Rees algebra $A'=\oplus_{n=0}^\infty I^n$ is normal. We will do this by showing that the logarithmic structure on $X'=\Spec(A')$ given by the natural homomorphism $\alpha'\:M'\to A'$ extending $\alpha$ makes $X'$ a toroidal scheme, i.e. the morphism $f'\:X'\to\Spec\ZZ[M']$ is a toroidal chart.

Let $d=\rk(M)$. Since $M'=\oplus_{n=0}^\infty nJ$ is saturated and $\rk(M')=d+1$, it suffices to show that the center $C'$ of $X'$ is regular and satisfies $\dim\cO_{X',x}=d+1+\dim\cO_{C',x}$ for any $x\in C'$. First, we claim that $f'$ is the base change of $f$. Indeed, $A$ and $\ZZ[M/nJ]$ are Tor-independent over $\ZZ[M]$ by \cite[Theorem~6.1(ii)]{Kato-toric}, and it follows easily that $A\otimes_{\ZZ[M]}\ZZ[nJ]=I^n$ and hence $A\otimes_{\ZZ[M]}\ZZ[M']=A'$ (e.g., see the proof of \cite[Proposition~4.3]{Niziol}). Since, $C'$ and $C$ are the preimages of the origins of $\Spec\ZZ[M']$ and $\Spec\ZZ[M]$, respectively, $C'\toisom C$. So, $C$ is regular and $\dim\cO_{X,x}=d+\dim\cO_{C,x}$, and it remains to show that $\dim\cO_{X',x}=\dim\cO_{X,x}+1$. Since $V(I)$ contains no generic points of $X$, the latter follows from the standard theory of Rees algebras.

Note that $Bl_\cI(X)^\nor=Bl_\cI(X)^\sat$, and by the first claim we have that $(\cI^n)^\sat=(\cI^n)^\nor$ for any $n$. So, by the definitions of the thresholds, $n^\nor(X,\cI)=n^\sat(X,U,\cI)$.
\end{proof}

\subsubsection{Normalization threshold of torific ideals}\label{finnorsec}
Now, we will study thresholds for torific ideals. The main tool is to reduce everything to the case of toroidal ideals by enlarging the toroidal structure.


\begin{theorem}\label{functortorific}
Assume that toroidal schemes $X$ and $Y$ are provided with simple and relatively affine actions of $G=\bfD_L$, and $f\:Y\to X$ is a strict strongly equivariant morphism. Consider the projections $\pi_X\:X\to\tilX=X\sslash G$ and $\pi_Y\:Y\to\tilY=Y\sslash G$ and the quotient morphism $\tilf=f\sslash G\:\tilY\to\tilX$. Let $S$ be a locally finite coherent multiset in $L_X$ and set $I=I^X_S$, $\tI=(\pi_X)_* I\cap\cO_\tilX$, $J=I_{f^*(S)}^Y$ and $\tJ=(\pi_Y)_* J\cap\cO_\tilY$. Then

(i) $n^\nor(X,I)$ and $n^\nor(\tilX,\tI)$ are finite. In particular, the torific blowing up $b_{X,S}$ is a projective morphism.

(ii) $n^\nor(Y,J)\le n^\nor(X,I)$ and $n^\nor(\tilY,\tJ)\le n^\nor(\tilX,\tI)$, and the equalities hold whenever $f$ is surjective.

(iii) $f^*((I^n)^\nor)=(J^n)^\nor$ and $Bl_{(J^n)^\nor}(Y)=Bl_{(I^n)^\nor}(X)\times_XY$ for any $n>0$.

(iv) $\tilf^*((\tI^n)^\nor)=(\tJ^n)^\nor$ and $Bl_{(\tJ^n)^\nor}(\tilY)=Bl_{(\tI^n)^\nor}(\tilX)\times_\tilX\tilY$ for any $n>0$.
\end{theorem}
\begin{proof}
Step 1. {\it The theorem holds when the action on $X$ is local.} All assertions do not involve the toroidal structure, so we will enlarge it on $X$ and $Y$ compatibly. On $X$ we can enlarge the toroidal structure using Proposition~\ref{pretoroidalprop}, thus making the action  on $X$ toroidal. At the same time,  if $D'$ is the new toroidal divisor then we also enlarge the toroidal structure of $Y$ to $E'=f^{-1}(D')$ so that $f$ remains a strict morphism. As we showed in the proof of Theorem~\ref{combchart}, we can choose $D'$ so that $\oM'_O=\oM_O\oplus\NN^{\sigma_O}$. By Lemma~\ref{toroidalactionlem}, the action on $(Y,E')$ is toroidal too. So, by Theorem~\ref{toroidalth} the quotients $\tilX$ and $\tilY$ acquire a structure of toroidal schemes, $\tilf$ becomes a strict morphism of toroidal schemes, and $\oM_{\tilX,\tilx}=(\oM'_O)_0$.

Since the actions with respect to the new structures are toroidal, the torific ideal $I$ is toroidal, and the invariant part $\tI$ of $I$ is toroidal by Theorem~\ref{toroidalth}(iv). Thus, $(I^n)^\nor=(I^n)^\sat$ and $(\tI^n)^\nor=(\tI^n)^\sat$ by Corollary~\ref{norsatcor}. Recall that $f^*(I)=J$ by Lemma~\ref{Lem:functorideal}, hence $J$ and also $\tJ$ are toroidal and $(J^n)^\nor=(J^n)^\sat$, $(\tJ^n)^\nor=(\tJ^n)^\sat$. Saturation is compatible with strict morphisms by its very definition (see Section \ref{Sec:sat-toroidal}). By Corollary \ref{norsatcor}, $f^*((I^n)^\nor)=(J^n)^\nor$ and $\tilf^*((\tI^n)^\nor)=(\tJ^n)^\nor$. The second assertions in (iii) and (iv) follow from Lemma \ref{toroidalblowlem}: blowings up of toroidal ideals are compatible with strict morphism. By Corollary~\ref{norsatcor}, we can replace normalization thresholds with saturation thresholds in parts (i) and (ii), and then the assertion follows from Theorem~\ref{satidealth}: in the notation of Section \ref{Sec:sat-monoidal}, if $O$ is the closed orbit and $\tilx$ its image in $\tilX$ then $n^\nor(X,I)=n^\sat(\oM'_O,\oI_O)$ and $n^\nor(\tilX,\tI)=n^\sat(\oM_{\tilX,\tilx},\tI_\tilx)$.

Step 2. {\it The general case.} Recall that $X$ is covered by its $G$-localizations $X_O$ at special orbits $O$ and torific ideals are compatible with such localizations (e.g. by Lemma~\ref{Lem:functorideal}). Since the theorem was proved in step 1 for the toroidal schemes $X_O$ acted on by $G$, we obtain that parts (ii), (iii) and (iv) hold for $X$. Concerning (i), we obtain equalities $$n^\nor(X,I)=\max_{O}\ n^\sat(\oM_O\oplus\NN^{\sigma_O},\oI_O),\ \ n^\nor(\tilX,\tI)=\max_{\tilx\in\tilX}\ n^\sat(\oM_\tilx,\tI_\tilx),$$ where the first maximum is over all special orbits (in fact, it suffices to consider closed orbits of $X$ and closed points of $\tilX$).

To prove that the maxima are finite we note that by Proposition~\ref{combdataprop}(i) and the local finiteness of $S$, there are finitely many different triples $(L_O,S_O,\oM_O\oplus\NN^{\sigma_O})$. So, it suffices to show that such a triple determines $n^\sat(\oM_O\oplus\NN^{\sigma_O},\oI_O)$ and $n^\sat(\oM_\tilx,\tI_\tilx)$. To check the latter set $P=\oM_O\oplus\NN^{\sigma_O}$. Then it is easy to see that $\oI_O=\sum_{l\in S_O}(P_l+P)$, and so $\tI_\tilx=(\oI_O)_0$ by Theorem~\ref{toroidalth}(iv).
\end{proof}

\subsection{Application to torification}
As an application, we can achieve torification by usual blowings up rather than normalized blowings up. In particular, it is always of finite type. Moreover, we will show that one can choose the ideal functorially, though this time only functorially with respect to surjective morphisms.

\subsubsection{Balanced torification}\label{torifyingsec}
If $G$ acts on a toroidal scheme $X$ we define the {\em torifying blowing up} $f_X\:X'\to X$ to be the blowing up along the ideal $I_X=(I^n)^\nor$, where $I=I_{\sigma^0}^X$ and $n=n^\nor(X,I)$. Similarly, set $\tI=(\pi_X)_*I\cap\cO_\tilX$ and $\tn=n^\nor(\tilX,\tI)$, where $\pi_X\:X\to\tilX=X\sslash G$ is the projection, and define the {\em quotient torifying blowing up} $\tilf_X\:\tilX'\to\tilX$ to be the blowing up along $\tI_X=(\tI^\tn)^\nor$. Here we use that $n$ and $\tn$ are finite by Theorem ~\ref{functortorific}.

\begin{theorem}\label{Maintheorem}
Assume that a toroidal scheme $(X,D)$ is provided with a relatively affine,  $G$-simple action of a diagonalizable group $G=\bfD_L$. Consider the torifying blowing up $f_{(X,D)}\:X'\to X$ and let $D'$ be the union of the preimage of $D$ and the exceptional divisor of $f_{(X,D)}$. Then,

(i) The pair $(X',D')$ is toroidal and the natural $G$-action on $(X',D')$ is toroidal.

(ii) The quotient torifying blowing up $\tilf_{(X,D)}$ is the quotient $X'\sslash G\to\tilX=X\sslash G$ of $f_{(X,D)}$.

(iii) The blowings up $f_{(X,D)}$ and $\tilf_{(X,D)}$ are functorial with respect to surjective strongly equivariant strict morphisms $h\:(Y,E)\to(X,D)$ of toroidal schemes: $h^*(I_X)=I_Y$, $f_{(Y,E)}=f_{(X,D)}\times_XY$, $\tilh^*(\tI_X)=\tI_Y$, and $\tilf_{(Y,E)}=\tilf_{(X,D)}\times_{\tilX}\tilY$, where $\tilh=h\sslash G$.

(iv) If the action of $G$ on $(X,D)$ is full, then $f_{(X,D)}$ and $\tilf_{(X,D)}$ are birational.

(v) If $V\subseteq X$ is a strongly equivariant open subset such that the action on $(V,D|_V)$ is toroidal then $I_X$ restricts to the unit ideal on $V$ and $V\times_XX'=V$.
\end{theorem}
\begin{proof}
As a morphism, $f_{(X,D)}$ is the torific blowing up $F^0_{(X,D)}$ from Theorem~\ref{maintheorem}, hence (i) follows from Theorem~\ref{maintheorem}(i). In the same fashion, (ii) and (iv) follow from the analogous parts of Theorem~\ref{maintheorem}. Recall that $h^*(I_{\sigma^0}^X)=I_{\sigma^0}^Y$ by Lemmas~\ref{Sigmalem}(ii) and \ref{Lem:functorideal}, hence compatibility of $I_X$ and $\tilf_{(X,D)}$ with $h$ follows from Theorem~\ref{functortorific}. Finally, using the functoriality of (iii) we can replace $V$ with $X$ in (v) so that the action on $X$ is toroidal. Then the actions on the logarithmic strata have locally constant stabilizers and hence $\sigma^0_X=\emptyset$. In particular, $I_{\sigma^0}^X=\cO_X$ and we obtain that $I_X=\cO_X$.
\end{proof}

\begin{remark}
Unlike Theorem \ref{maintheorem}(ii), it is not true in general that $\tI_X=(\pi_X)_*I_X\cap\cO_\tilX$. Thus we functorially realize both the toryfiying morphism $f_{(X,D)}$ and its quotient $\tilf_{(X,D)}$ as blowings up of ideals, but the relation between these ideals is not so tight anymore.
\end{remark}

\subsubsection{General torification}
In the same way one can upgrade the general torification from Theorem~\ref{maintheorem2} to torification by blowings up, and we leave the details to the interested reader.

\begin{theorem}\label{Th:general-normal}
The torification sequence $F_{(X,D)}$ from Theorem~\ref{maintheorem2} can be naturally realized as a sequence of blowings up along ideals $I_i\subseteq\cO_{X_i}$. In particular, the torification morphism is projective. Moreover, the choice of $I_i$ can be made functorial with respect to surjective strongly equivariant strict toroidal morphisms.
\end{theorem}

\subsection{$\GG_m$-action}
In \cite{AT2} we will also need a few results on compatibility of torification with the subschemes $X_{\pm}\into X$ in the case of $\GG_m$-actions. Recall that if $X$ is provided with a relatively affine action of $G=\GGm$ then we defined in \cite[Section~5.1.17]{ATLuna} the open subschemes $X_+$ and $X_-$ of $X$, on which the action of $G$ is still relatively affine (\cite[Lemma 5.1.18]{ATLuna}). If $X$ underlies a toroidal $G$-equivariant scheme $(X,D)$ then $D_+=X_+\cap D$ and $D_-=X_-\cap D$ define equivariant toroidal subschemes of $(X,D)$. We will write $X=(X,D)$ and $X_\pm=(X_\pm,D_\pm)$ for brevity of notation.

\subsubsection{Compatibility with toroidal quotients}
\begin{proposition}\label{gmprop}
Assume that $G=\GGm$ acts toroidally and relatively affine on a toroidal scheme $X$. Then the morphisms $X_\pm\sslash G\to X\sslash G$ are toroidal.
\end{proposition}
\begin{proof}
By symmetry it suffices to consider $X_+$. The claim is local on $X\sslash G$, hence we can assume that the action on $X$ is local. If the stabilizer of the closed orbit $O$ is $\mu_n$ then $X=X_+$ and the claim becomes obvious. So we can assume that the stabilizer is $\GG_m$ and $O=\{x\}$, i.e. the action is strictly local. Then by Corollary~\ref{toroidalprop} we can find a strongly equivariant chart $\phi\:X\to Y=\bfA_M$, where the action on $Y$ is via a grading $M\to\ZZ$.

Let us prove that the morphism $f\:Y_+\sslash G\to Y\sslash G$ is toroidal. Clearly, $Y\sslash G=\bfA_{M_0}$. By \cite[Lemma~5.1.18]{ATLuna}
$Y_+$ is covered by open strongly equivariant subschemes $(\bfA_M)_m=\bfA_{M[-m]}$, where $m\in M$ is negatively graded, and hence $Y_+\sslash G$ is covered by affine open subschemes $\bfA_{M[-m]_0}$. We claim that $(M[-m]_0)^\gp/(M_0)^\gp$ are torsion free and hence $f$ is toroidal. Indeed, this group is a subgroup of $M^\gp/(M_0)^\gp$ hence it suffices to show that the latter is torsion free. It remains to note that $M_0$ is saturated in $M^\gp$, hence the lattice $(M_0)^\gp$ is saturated in $M^\gp$, and the quotient $M^\gp/(M_0)^\gp$ is torsion free.

Now, the assertion of the proposition would follow once we prove that the bottom square in the following commutative diagram is cartesian
$$\xymatrix{
X_+\ar@{^(_->}[r] \ar[d]\ar[drr]^(.7){\phi_+}& X\ar[d]|!{[l];[dr]}\hole\ar[drr]^\phi\\
X_+\sslash G \ar[r]\ar[drr]_{\phi_+\sslash G} & X\sslash G \ar[drr]^(.7){\phi\sslash G}|!{[r];[dr]}\hole & Y_+\ar@{^(_->}[r]\ar[d] &Y\ar[d]\\
&& Y_+\sslash G \ar[r]^(.4)f &Y\sslash G.
}$$

By the strong equivariance, $\phi$ is the base change of $\phi\sslash G$. Since $\phi_+$ is the base change of $\phi$ by \cite[Lemma~5.3.8]{ATLuna}, it is also a base change of $\phi\sslash G$, and then the bottom square is cartesian by Lemma~\ref{gmlem} below.
\end{proof}

\begin{lemma}\label{gmlem}
Assume that $X,Y,X',Y'$ are acted on by a diagonalizable group $G$ so that the actions on $X$ and $Y$ are relatively affine and the actions on $X'$ and $Y'$ are trivial. If a $G$-equivariant morphism $f\:X\to Y$ is a base change of a morphism $f'\:X'\to Y'$, then also $f\sslash G$ is a base change of $f'$.
\end{lemma}
\begin{proof}
Since the quotients are categorical, the morphisms $X\to X'$ and $Y\to Y'$ factor through $X\sslash G$ and $Y\sslash G$, respectively. The problem is local on $X'$ and $Y'$, so we can assume that $X'=\Spec A'$ and $Y'=\Spec B'$. In particular, the morphism $f$ is affine. In addition, the problem is local on $Y\sslash G$, so we can assume that it is affine and then $Y=\Spec B$, $Y\sslash G=\Spec B_0$, $X=\Spec A$ and $X\sslash G=\Spec A_0$. By our assumption, $A=B\otimes_{B'}A'$. Since the gradings on $A'$ and $B'$ are trivial, this implies that $A_0=B_0\otimes_{B'}A'$, as required.
\end{proof}

\subsubsection{Compatibility of balanced torification}
Assume now that a toroidal scheme $X$ is provided with a relatively affine action of $G=\GGm$ and consider the torifying blowing up $f_X\:Y=Bl_I(X)\to X$ and the quotient torifying blowing up $\tilf_X\:\tilY=Bl_{\tI}(\tilX)\to\tilX$ from \S\ref{torifyingsec}. Let also $\widetilde{X_\pm}=X_\pm\sslash G$, $\widetilde{Y_\pm}=Y_\pm\sslash G$ and $\tI_\pm=\tI\cO_{\widetilde{Y_\pm}}$.

\begin{lemma}\label{Lem:Gm-action}
Keeping the above notation, $\widetilde{Y_\pm}=Bl_{\tI_\pm}(\tilY)$.
\end{lemma}

We will prove a more general result that deals with arbitrary torific ideals. So, let $S$ be a balanced locally constant coherent multiset in $L_X$, $J=I_X^S$, $Y=NorBl_J(X)$ and $J_\pm=J\cO_{X_\pm}$. Consider the quotients $\tilX=X\sslash G$, $\tilY=Y\sslash G$, $\widetilde{X_\pm}=X_\pm\sslash G$, $\widetilde{Y_\pm}=Y_\pm\sslash G$, and let $\tJ\subset\cO_\tilX$ be the $\GG_m$-invariant part of $J$ and $\tJ_\pm=\tJ\cO_{\widetilde{X_\pm}}$. Recall that $\tilY=NorBl_\tJ(\tilX)$ by Lemma~\ref{Lem:canonicalblowup}(iii). Finally, let $n=n^\nor(X,J)$, $I=(J^n)^\nor$, $I_\pm=I\cO_{X_\pm}$ and $\tn=n^\nor(\tilX,\tJ)$, $\tI=(\tJ^\tn)^\nor$, $\tI_\pm=\tI\cO_{\widetilde{X_\pm}}$.

\begin{proposition}
Keeping the above notation the following isomorphisms hold:

(i) $Y_\pm=Y\times_XX_\pm$ and $\widetilde{Y_\pm}=\tilY\times_\tilX\widetilde{Y_\pm}$.

(ii) $Y_\pm=NorBl_{J_\pm}(X_\pm)$ and $\widetilde{Y_\pm}=NorBl_{\tJ_\pm}(\tilX_\pm)$.

(iii) $Y_\pm=Bl_{I_\pm}(X_\pm)$ and $\widetilde{Y_\pm}=Bl_{\tI_\pm}(\tilX_\pm)$.
\end{proposition}
\begin{proof}
Let us work with $X_+$ for concreteness. As in the proof of Proposition \ref{gmprop}, it suffices to consider the case when the action is local. Moreover, we can assume that the action is strictly local, since otherwise $X_+=X$ and the assertions are vacuous. In particular, we can assume that $X=\Spec A$ for $A=\oplus_{n\in\ZZ}A_n$ and $S$ is a balanced multiset in $\ZZ$. Moreover, since all claims are independent of the toroidal structure of $X$, we can enlarge it by Proposition~\ref{pretoroidalprop} making the action toroidal. Once the action on $X$ is toroidal, the actions on $X_\pm$, $Y$ and $Y_\pm$ are toroidal too, and the quotients are toroidal schemes by Theorem~\ref{toroidalth}(i). Furthermore, the ideal $J$ is toroidal by Lemma~\ref{torificblowlem}(i), hence all other ideals are toroidal and all normalizations are, in fact, saturations. In particular, the normalizations of ideals and the (normalized) blowings up from (ii) and (iii) are compatible with base changes, and hence (ii) and (iii) follow from (i).

To prove (i) we use Corollary~\ref{toroidalprop} to pick up a strongly equivariant chart $X\to Z=\bfA_M$ with a $\ZZ$-graded toric monoid $M$. Since $X_\pm=X\times_{Z}Z_\pm$ by \cite[Lemma~5.3.8]{ATLuna}, it suffices to prove (i) for $Z$ instead of $X$, so we can assume that $X=\bfA_M$. By definition, $X_+$ has an open strongly equivariant covering by subschemes $X_f=\bfA_{M[-f]}$ with negatively graded $f\in M$. Since $S$ is balanced, $J$ is generated by elements $t_1\.t_l\in M_0$ and hence $Y$ is covered by the charts $Y_i=\bfA_{M[t_1-t_i\.t_l-t_i]^\sat}$ and $X_+\times_XY$ is covered by charts $Y_{f,i}=\bfA_{M[t_1-t_i\.t_l-t_i]^\sat[-f]}$. Similarly, $Bl_{J_+}(X+)$ is covered by the charts $\bfA_{M[-f][t_1-t_i\.t_l-t_i]^\sat}$. The localization of a saturated monoid is saturated, hence $$M[t_1-t_i\.t_l-t_i]^\sat[-m]=M[-m][t_1-t_i\.t_l-t_i]^\sat,$$ and we obtain that $Y_+=Y\times_XX_+$. Since $t_i\in M_0$, one checks in the same manner that $\widetilde{Y_+}=\tilY\times_\tilX\widetilde{Y_+}$ since both schemes are covered by the charts $\bfA_{M_0[-m,t_1-t_i\.t_l-t_i]^\sat}$.
\end{proof}

\bibliographystyle{amsalpha}
\bibliography{factor-qe}
\printindex
\end{document}